\documentclass[acmlarge, authorversion, nonacm]{acmart}

\newcommand{\shortorlong}[2]{#2}

\newcommand{\narrowbreak}{}

\newenvironment{narrowmultline}{\csname equation\endcsname}{\csname endequation\endcsname}
\newenvironment{narrowmultline*}{\csname equation*\endcsname}{\csname endequation*\endcsname}

\copyrightyear{2020}
\acmYear{2020}
\setcopyright{acmlicensed}
\acmConference[LICS '20]{Proceedings of the 35th Annual ACM/IEEE Symposium on Logic in Computer Science (LICS)}{July 8--11, 2020}{Saarbr\"ucken, Germany}
\acmBooktitle{Proceedings of the 35th Annual ACM/IEEE Symposium on Logic in Computer Science (LICS '20), July 8--11, 2020, Saarbr\"ucken, Germany}
\acmPrice{15.00}
\acmDOI{10.1145/3373718.3394755}
\acmISBN{978-1-4503-7104-9/20/07}

\usepackage{booktabs}   
\usepackage{subcaption} 

\usepackage[
  all
  , 2cell
]{xy}
\xyoption{rotate}
\SilentMatrices\UseAllTwocells
\usepackage{tikz-cd}
\usetikzlibrary{decorations.pathmorphing}
\usepackage[capitalise]{cleveref}
\usepackage{mathtools}
\usepackage[T1]{fontenc}

\usepackage[draft]{optional}
\usepackage{xcolor}
\newcommand{\plan}[1]{}
\newcommand{\BA}[1]{}
\newcommand{\PN}[1]{}
\newcommand{\MS}[1]{}
\newcommand{\DT}[1]{}
\opt{draft}{
\renewcommand{\plan}[1]{\textcolor{blue}{Plan: #1}\PackageWarning{TODO}{TODO: #1}}
\renewcommand{\BA}[1]{\textcolor{orange}{BA: #1}\PackageWarning{TODO}{TODO: #1}}
\renewcommand{\PN}[1]{\textcolor{purple}{PN: #1}\PackageWarning{TODO}{TODO: #1}}
\renewcommand{\MS}[1]{\textcolor{magenta}{MS: #1}\PackageWarning{TODO}{TODO: #1}}
\renewcommand{\DT}[1]{\textcolor{red}{DT: #1}\PackageWarning{TODO}{TODO: #1}}
}

\newcommand{\Sig}{\mathsf{Sig}}
\newcommand{\FinSet}{\mathsf{FinSet}}
\newcommand{\Nat}{\mathbb{N}}

\renewcommand{\U}{\mathcal{U}}
\newcommand{\Ustrict}{\U^s}
\newcommand{\steq}{\stackrel{\mathrm s}{=}}

\newcommand{\converts}{\equiv}
\newcommand{\define}{:\converts}
\newcommand{\bottomlevel}{\bot}
\newcommand{\bottom}[1]{{#1}_\bottomlevel}
\newcommand{\Set}{\ensuremath{\mathsf{Set}}}
\newcommand{\strictiso}{\cong}

\newcommand{\derivcat}[2]{\ensuremath{{#1}'_{#2}}}
\newcommand{\derivdia}[1]{\ensuremath{{#1}'}}
\newcommand{\doublederivdia}[1]{\ensuremath{{#1}''}}

\renewcommand{\L}{\mathcal{L}}
\newcommand{\M}{\mathcal{M}}
\newcommand{\N}{\mathcal{N}}
\newcommand{\Struc}[1]{\mathsf{Str}({#1})}
\newcommand{\isIso}{\mathsf{isIso}}
\newcommand{\isEquiv}{\mathsf{isEquiv}}
\newcommand{\refl}{\mathsf{refl}}
 \newcommand{\uStruc}[1]{\mathsf{uStr}({#1})}

\newcommand{\cat}{\ensuremath{\mathsf{Cat}}}

\newcommand{\VSS}{\mathsf{SSEquiv}}
\newcommand{\SplitS}[1]{\sum_{s} \left( #1 \circ s = 1\right)}

\newcommand{\streqv}{\mathsf{SEquiv}}
\newcommand{\strucequiv}{\ensuremath{\twoheadrightarrow}} 

\newcommand{\fanout}[2]{\ensuremath{\mathsf{Fanout}_{#2}(#1)}}
\newcommand{\fanoutfun}[3]{\ensuremath{{#1}_{#3}(#2)}}
\newcommand{\IC}{\mathsf{F\Sig}}
\newcommand{\Lrg}{\L_{\text{rg}}}
\newcommand{\Lcat}{\L_{\text{cat}}}
\newcommand{\LcatE}{\L_{\text{cat+E}}}

\newcommand{\eqdef}{\define}
\newcommand{\MatK}[2]{\ensuremath{#1 #2}}
\newcommand{\copair}[2]{\langle #1,#2\rangle}
\newcommand{\congto}{\xrightarrow{\raisebox{-1ex}[0ex][0ex]{$\sim$}}}

\newcommand{\ob}[1]{\ensuremath{{#1}_0}}

\newcommand{\concat}{\ensuremath{\cdot}}
\newcommand{\idtoiso}{\ensuremath{\mathsf{idtoiso}}}

\newcommand{\idtovss}{\ensuremath{\mathsf{idtosse}}}
\newcommand{\isotovss}{\ensuremath{\mathsf{isotosse}}}
\newcommand{\idtoeqv}{\ensuremath{\mathsf{idtoeqv}}}
\newcommand{\ua}{\ensuremath{\mathsf{ua}}}

\newcommand{\istype}[1]{\mathsf{is}\mbox{-}{#1}\mbox{-}\mathsf{type}}

\newcommand{\nminusone}{\ensuremath{(n-1)}}

\newcommand{\defemph}[1]{\textbf{#1}}
\newcommand{\PropU}{\mathsf{Prop}_\U}
\renewcommand\C{\mathcal{C}}
\newcommand{\D}{\mathcal{D}}

\newcommand{\foldsiso}{\asymp}
\newcommand{\fiso}{\foldsiso}
\newcommand{\trans}[2]{{#1}_*(#2)}
\newcommand{\transfun}[1]{{#1}_*}
\newcommand{\zerotype}{\mathbf{0}}
\newcommand{\onetype}{\mathbf{1}}

\AtEndPreamble{
  \theoremstyle{acmdefinition}
  \newtheorem{examples}[theorem]{Examples}
}

\makeatletter

\def\prd#1{\@ifnextchar\bgroup{\prd@parens{#1}}{\@ifnextchar\sm{\prd@parens{#1}\@eatsm}{\prd@noparens{#1}}}}
\def\prd@parens#1{\@ifnextchar\bgroup%
  {\mathchoice{\@dprd{#1}}{\@tprd{#1}}{\@tprd{#1}}{\@tprd{#1}}\prd@parens}%
  {\mathchoice{\@dprd{#1}}{\@tprd{#1}}{\@tprd{#1}}{\@tprd{#1}}}}
\def\@eatsm\sm{\sm@parens}
\def\prd@noparens#1{\mathchoice{\@dprd@noparens{#1}}{\@tprd{#1}}{\@tprd{#1}}{\@tprd{#1}}}
\def\lprd#1{\@ifnextchar\bgroup{\@lprd{#1}\lprd}{\@@lprd{#1}}}
\def\@lprd#1{\mathchoice{{\textstyle\prod}}{\prod}{\prod}{\prod}({\textstyle #1})\;}
\def\@@lprd#1{\mathchoice{{\textstyle\prod}}{\prod}{\prod}{\prod}({\textstyle #1}),\ }
\def\tprd#1{\@tprd{#1}\@ifnextchar\bgroup{\tprd}{}}
\def\@tprd#1{\mathchoice{{\textstyle\prod_{(#1)}}}{\prod_{(#1)}}{\prod_{(#1)}}{\prod_{(#1)}}}
\def\dprd#1{\@dprd{#1}\@ifnextchar\bgroup{\dprd}{}}
\def\@dprd#1{\prod_{(#1)}\,}
\def\@dprd@noparens#1{\prod_{#1}\,}

\def\sm#1{\@ifnextchar\bgroup{\sm@parens{#1}}{\@ifnextchar\prd{\sm@parens{#1}\@eatprd}{\sm@noparens{#1}}}}
\def\sm@parens#1{\@ifnextchar\bgroup%
  {\mathchoice{\@dsm{#1}}{\@tsm{#1}}{\@tsm{#1}}{\@tsm{#1}}\sm@parens}%
  {\mathchoice{\@dsm{#1}}{\@tsm{#1}}{\@tsm{#1}}{\@tsm{#1}}}}
\def\@eatprd\prd{\prd@parens}
\def\sm@noparens#1{\mathchoice{\@dsm@noparens{#1}}{\@tsm{#1}}{\@tsm{#1}}{\@tsm{#1}}}
\def\lsm#1{\@ifnextchar\bgroup{\@lsm{#1}\lsm}{\@@lsm{#1}}}
\def\@lsm#1{\mathchoice{{\textstyle\sum}}{\sum}{\sum}{\sum}({\textstyle #1})\;}
\def\@@lsm#1{\mathchoice{{\textstyle\sum}}{\sum}{\sum}{\sum}({\textstyle #1}),\ }
\def\tsm#1{\@tsm{#1}\@ifnextchar\bgroup{\tsm}{}}
\def\@tsm#1{\mathchoice{{\textstyle\sum_{(#1)}}}{\sum_{(#1)}}{\sum_{(#1)}}{\sum_{(#1)}}}
\def\dsm#1{\@dsm{#1}\@ifnextchar\bgroup{\dsm}{}}
\def\@dsm#1{\sum_{(#1)}\,}
\def\@dsm@noparens#1{\sum_{#1}\,}

\begin{document}


\allowdisplaybreaks



\theoremstyle{acmdefinition}\newtheorem{notation}[theorem]{Notation}
\theoremstyle{acmdefinition}\newtheorem{remark}[theorem]{Remark}
\theoremstyle{acmdefinition}\newtheorem{convention}[theorem]{Convention}


\title[A Higher Structure Identity Principle]{A Higher Structure Identity Principle}         


\author{Benedikt Ahrens}
\orcid{0000-0002-6786-4538}             
\affiliation{
  \department{School of Computer Science}              
  \institution{University of Birmingham}            
  \streetaddress{Street1 Address1}
  \postcode{Post-Code1}
  \country{United Kingdom}                    
}
\email{b.ahrens@cs.bham.ac.uk}          

\author{Paige Randall North}
\orcid{0000-0001-7876-0956}             
\affiliation{
  \department{Department of Mathematics}              
  \institution{The Ohio State University}            
  \country{United States}                    
}
\email{north.138@osu.edu}          

\author{Michael Shulman}
\orcid{0000-0002-9948-6682}             
\affiliation{
  \department{ Department of Mathematics}              
  \institution{ University of San Diego}            
  \country{United States}                    
}
\email{shulman@sandiego.edu}          

\author{Dimitris Tsementzis}
\affiliation{
  \institution{Princeton University}            
  \institution{Rutgers University}
  \country{United States}                    
}
\email{dimitrios.tsementzis@gmail.com}          

\begin{abstract}
  The ordinary Structure Identity Principle states that any property of set-level structures (e.g., posets, groups, rings, fields) definable in Univalent Foundations is invariant under isomorphism: more specifically, identifications of structures coincide with isomorphisms.
  We prove a version of this principle for a wide range of higher-categorical structures, adapting FOLDS-signatures to specify a general class of structures, and using two-level type theory to treat all categorical dimensions uniformly.
  As in the previously known case of 1-ca\-te\-go\-ries (which is an instance of our theory), the structures themselves must satisfy a local univalence principle, stating that identifications coincide with ``isomorphisms'' between elements of the structure.
  Our main technical achievement is a definition of such isomorphisms, which we call ``indiscernibilities,'' using only the dependency structure rather than any notion of composition.

\end{abstract}

\ccsdesc[500]{Theory of computation~Type theory}
\ccsdesc[300]{Theory of computation~Logic and verification}
\ccsdesc[100]{Theory of computation~Constructive mathematics}

\keywords{homotopy type theory, univalent foundations, structure identity principle, categories, equivalence principle}  

\maketitle


\section{Introduction}

\subsection{The Structure Identity Principle (SIP)}
\label{sec:sip}

A fundamental logical principle is the ``indiscernibility of identicals'', i.e., equal objects have the same properties:
\begin{equation}
             x = y  \to   \forall \text{ properties } P, \left(P(x) \leftrightarrow P(y)\right) .
       \label{eq:indiscernability_of_identicals}
\end{equation}
However, properties invariant under weaker notions of sameness are also important.
For instance, group-theoretic properties satisfy a similar principle for \emph{isomorphic} groups:
\begin{equation*}
             G \cong H \to \forall \text{ group-theoretic properties } P, \left(P(G) \leftrightarrow P(H)\right);  \label{eq:indiscernability_of_identical_groups}
\end{equation*}
while category-theoretic properties are invariant even under \emph{equivalence} of categories:
\begin{equation*}
             A \simeq B \to   \forall \text{ category-theoretic properties } P, \left(P(A) \leftrightarrow P(B)\right). 
\end{equation*}

The idea is summarized in Aczel's Structure Identity Principle (SIP) \cite{Aczel_SIP}: \textit{
Isomorphic (or equivalent) mathematical structures are structurally identical;
i.e., have the same structural properties.
}
But it remains to characterize the ``structural properties'' for a given notion of structure.
For instance, Blanc \cite{Blanc} and Freyd \cite{Freyd} devised a syntax for category-theoretic properties and showed that they are invariant under equivalence.
And Makkai~\cite{MFOLDS} introduced general notions of signature and equivalence for higher-categorical structures, along with a language for their properties called First Order Logic with Dependent Sorts (FOLDS), and proved that FOLDS-properties are invariant under FOLDS-equivalence.

\subsection{The SIP in Univalent Foundations}
\label{sec:sip0}

Inspired by Makkai (see \cite[p.~1279]{voevodsky_2015}), Voevodsky conceived Univalent Foundations (UF) with a similar but more ambitious goal: a foundational language for mathematics, \emph{all} of whose constructions are invariant under equivalences of structures.
Since proofs are particular constructions, this implies a similar invariance of properties. 

The formal language of UF and the closely related Homotopy Type Theory (HoTT) is Martin-L\"{o}f type theory, with types regarded as (higher) groupoids; see~\cite{HTT} for background and notation.
Voevodsky's univalence principle
\begin{equation}
 \mathsf{univalence} : \Pi_{(x, y: \U)}  (x =_\U y  \congto  x \simeq y)
 \label{eq:ua}
\end{equation}
ensures that all properties of \emph{types} are invariant under equivalences, since equivalences can be made into identifications (i.e., elements of the Martin-L\"{o}f identity type $x=y$):
\begin{equation*}           \Pi_{(x,y:\U)} \left( x \simeq y   \to   \Pi_{(P: \U \to \U)} \left(P(x) \simeq P(y)\right) \right) .
\end{equation*}
It was proven in~\cite[Section~9.9]{HTT} and~\cite{COQUAND20131105}\footnote{The formalization of \cite{COQUAND20131105} compares the two independent results.} that the same approach works for a wide range of mathematical structures.
Both use a notion of ``signature'' to define general classes of structures and isomorphisms, and show (using~\eqref{eq:ua} for the underlying types) that isomorphisms of structures are equivalent to identities of structures.
Thus, indiscernibility of identicals implies indiscernibility of isomorphs.

\subsection{SIPs for Categories and Higher Categories}
\label{sec:sip1}

The main restriction of the SIP of \cref{sec:sip0} is that it applies only to structures that naturally form a 1-category (or 1-groupoid).
In particular, it excludes categories themselves.
An SIP for categories is proved in \cite[Theorem~6.17]{AKS13}: equivalence of (certain) categories is equivalent to identity of categories, yielding an analogous transport principle for categorical equivalences $x \simeq y$:
\begin{equation*}             \Pi_{(x,y:\cat)} \left( x \simeq y   \to   \Pi_{(P: \cat \to \U)} \left(P(x) \simeq P(y)\right) \right) .
\end{equation*}
However, to make this true, the categories themselves must satisfy a local univalence principle saying that isomorphism is equivalent to identity for their {objects}.

In the present paper, grown out of \cite{Tsementzis_HSIP} and inspired by FOLDS, we generalize this to other (higher-)cat\-egorical structures.
We give general notions of signature, structure, univalence, and equivalence, such that equivalence of univalent structures is equivalent to identity.
Therefore, indiscernibility of identicals implies indiscernibility of equivalents:
\begin{equation*}             \Pi_{(x,y:\uStruc{\L})} \left( x \simeq_\L y   \to   \Pi_{(P: \uStruc{\L} \to \U)} \left(P(x) \simeq P(y)\right) \right) .
\end{equation*}

The primary difficulty is to give a suitable notion of iso\-mor\-phism---which we also call \emph{indiscernibility}---between \emph{elements} of a general structure, so that we can define a structure to be univalent if this notion of isomorphism of its elements is equivalent to their identity.
Our definition is a relativized form of the converse of~\eqref{eq:indiscernability_of_identicals}, the \emph{identity of indiscernibles}:
\begin{equation}
  \left(\forall \text{ properties } P, \left(P(x) \leftrightarrow P(y)\right)\right) \to (x = y) .\label{eq:identity-of-indiscernibles}
\end{equation}
As long as we generalize ``property'' to ``construction'', the global form of this principle in HoTT/UF is trivially true because we have the haecceity $P(u) \eqdef (x=u)$.
We get our notion of indiscernibility by relativizing this to a structure, allowing only constructions involving the data of a structure to which $x,y$ belong.
Thus, our univalent structures satisfy a local form of~\eqref{eq:identity-of-indiscernibles}.

\begin{remark}
  The SIP of \cref{sec:sip0} is equivalent to the statement that the 1-category \emph{of} structures is univalent in the sense of~\cite{AKS13}.
  This suggests a ``Baez--Dolan microcosm principle'' at work: the SIP \emph{for} a given kind of structure should state that the (higher) category \emph{of} such structures satisfies the necessary univalence principle for \emph{its} SIP.
  We have not attempted to state or prove this precisely.
\end{remark}

\section{A Fresh Look at Univalent Categories}\label{sec:folds-cats}


First we review~\cite{AKS13} and~\cite[Chapter 9]{HTT} with an eye to generalization.
Both start by defining a \textbf{precategory} $\C$ as follows.
\begin{itemize}
\item A type $\ob{\C}$ of objects.  
\item For each $a,b:\ob{\C}$, a set $\C(a,b)$ of {morphisms}.
\item For each $a:\ob{\C}$, a morphism $1_a:\C(a,a)$.
\item For each $a,b,c:\ob{\C}$, a function
  \[  \C(b,c) \to \C(a,b) \to \C(a,c). \]
\item For each $a,b:\ob{\C}$ and $f:\C(a,b)$, we have $f = {1_b\circ f}$ and $f = {f\circ 1_a}$.
\item For each $a,b,c,d:\ob{\C}$ and $f:\C(a,b)$, $g:\C(b,c)$, $h:\C(c,d)$, we have ${h\circ (g\circ f)} ={(h\circ g)\circ f}$.
\end{itemize}
Note $\ob{\C}$ may not be a set, and for ``large'' precategories it almost never is.
For instance, $\ob{\Set}$ is the type of sets, which by univalence is a proper 1-type.
However, allowing arbitrary types of objects is problematic too. 
For instance, while the statement ``a fully faithful and essentially surjective functor is an equivalence'' in ZF is equivalent to the axiom of choice, for precategories in HoTT/UF it is generally false, even with the axiom of choice.

The solution is to impose a ``local univalence'' condition: 
for any $a,b:\ob{\C}$ there is a map
$\idtoiso_{a,b}:(a=_{\ob{\C}} b) \to (a\cong b)$,
defined by path induction, and a precategory $\C$ is a \textbf{univalent category} if $\idtoiso_{a,b}$ is an equivalence for all $a,b : \ob{\C}$.
This implements the idea that ``isomorphic objects are equal''.
Note that it implies that $\ob{\C}$ is a 1-type, since its identity types are all sets (0-types).

One then proves, using the univalence axiom, that this ``local'' form of univalence for objects of a category implies a ``global'' form of univalence for categories themselves:

\begin{theorem}[{\cite[Theorem~6.17]{AKS13}}]\label{thm:univalence-for-categories}
  For univalent categories $\C$ and $\D$, let $\C \simeq \D$ be the type of categorical equivalences between $\C$ and $\D$; then
  \[ (\C= \D) = (\C\simeq \D). \]
\end{theorem}

We will generalize this to other categorical structures, starting with a general vocabulary for expressing such things.

\subsection{FOLDS-Signature for Categories}

In~\cite{MFOLDS}, Makkai presents a definition of category in a language called First-Order Logic with Dependent Sorts (FOLDS).
In contrast to HoTT/UF, FOLDS is not a foundational system for mathematics, but a kind of first-order logic designed for higher categorical structures.
We will not use the logical syntax of FOLDS, but we adopt and generalize its notions of signature and structure. 

A FOLDS-signature is an \emph{inverse category with finite fan\-outs}, whose objects are called \emph{sorts}.
The FOLDS-signature $\Lcat$ of categories is shown in \Cref{fig:signatures}, along with the related FOLDS-signatures $\Lrg$ of reflexive graphs and $\LcatE$ of categories with equality (see \Cref{subsec:equations_and_theories}).
There are some relations on the composite arrows (e.g., the two composites $I\to A\rightrightarrows O$ are equal).
The intent is that $O$ is the sort of objects, $A$ the sort of arrows, $I$ the sort of identity arrows, and $T$ the sort of composable pairs of arrows (with their composite; $T$ stands for ``triangle'').


\begin{figure}
\[
\xymatrix{
2 &
    I \ar[d]_{i}   & 
    T \ar@/^5pt/[rd]^{t_2} \ar[rd]_*-<.5em>{^{t_1}} \ar@/_10pt/[rd]_{t_0} &
    I \ar[d]^{i} &
    T \ar@/^5pt/[rd]^{t_2} \ar[rd]_*-<.5em>{^{t_1}} \ar@/_10pt/[rd]_{t_0} & 
    I \ar[d]^{i}  &
    E \ar@/_/[ld]_{e_1} \ar@/^/[ld]^{e_2}
    \\
1 &
    A \ar@/^/[d]^{d} \ar@/_/[d]_{c}  & 
      &
    A \ar@/^/[d]^{d} \ar@/_/[d]_{c} &
      &
    A \ar@/_/[d]_{c} \ar@/^/[d]^{d}
    \\
0 &
    O & 
      &
    O &
      &   
    O
}
\]


\begin{equation*}
 ci = di,
 ct_0 = dt_1,
 dt_0 = dt_2,
 ct_1 = ct_2,
 de_1 = de_2,
 ce_1 = ce_2
\end{equation*}

\caption{The FOLDS-signatures $\Lrg$, $\Lcat$, and $\LcatE$ (from left to right) for reflexive graphs, for categories, and for categories with equality predicate on arrows.
The morphisms are subject to the indicated equalities.
}
\Description[]{Three FOLDS-signatures, one with objects $O$, $A$, and $I$, one with a objects $O$, $A$, $T$, $I$, and one with an additional object $E$. Two arrows from $A$ to $O$, three arrows from $T$ to $A$, one arrow from $I$ to $A$, and two arrows from $E$ to $A$.}

\label{fig:signatures}
\end{figure}
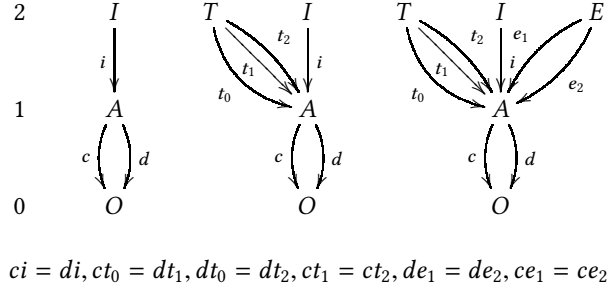

Using a set-theoretic metatheory, Makkai defined \emph{structures} for a FOLDS-signature as certain functors into $\mathbf{Set}$.
In a dependently typed theory like HoTT/UF, however, it is more natural to interpret each sort as a {dependent} type indexed by the interpretations of all the sorts below it.
For instance, in HoTT/UF a structure $M$ for the FOLDS-signature $\Lcat$ in \Cref{fig:signatures} consists of
\begin{align*}
  MO&:\U\\
  MA&:MO\times MO \to \U\\
  MI &: \tprd{x:MO} MA(x,x) \to \U\\
  MT &: \tprd{x,y,z:MO} MA(x,y) \to MA(y,z) \to MA(x,z) \to \U
\end{align*}
This forms the underlying data of a category: a type of objects, types of morphisms, and properties of ``being an identity'' and ``being the composite''.
Later we will see that our univalence condition implies that the families $MI$ and $MT$ consist of propositions and $MA$ consists of sets.

\subsection{Axioms and Theories}\label{subsec:equations_and_theories}
To express which such structures are actually categories, Makkai introduced a logic over FOLDS-signatures, taking top-level sorts as ``relation symbols''. 
For instance, the axiom that any two composable arrows have a composite would be
\begin{narrowmultline*} \forall (x,y,z:O). \forall (f:A(x,y)). \forall (g:A(y,z)). \narrowbreak \exists (h:A(x,z)). T_{x,y,z}(f,g,h).
\end{narrowmultline*}
Any such axiom can be interpretated as a predicate on $\L$-structures in HoTT/UF.
(Note that the interpretation of $\exists$ and $\vee$ involves propositional truncation.)

The axioms of a category also involve equality of arrows (though not of objects), e.g., the uniqueness of composites
\begin{narrowmultline}
  \forall (x,y,z:O). \forall (f:A(x,y)). \forall (g:A(y,z)). \forall (h,h':A(x,z)).\narrowbreak T_{x,y,z}(f,g,h)\land T_{x,y,z}(f,g,h') \to (h=h').
  \label{eq:comp-unique}
\end{narrowmultline}
Just as in ordinary first-order logic, one can consider FOLDS either \emph{with equality} or \emph{without equality}.
In the former, equality is only allowed between elements of sorts that are ``one level below the top'' like $A$. 
As usual, FOLDS with equality can be embedded in FOLDS without equality by adding equality relations to the signature, as with $\LcatE$ in \Cref{fig:signatures}, along with axioms making them congruences (where all free variables should be considered to be universally quantified):
\begin{gather}
  E_{x,y} (f,f)\notag\\
  E_{x,y}(f,g) \to E_{x,y}(g,f)\notag\\
  E_{x,y}(f,g) \land E_{x,y}(g,h) \to E_{x,y}(f,h)\notag\\
  E_{x,x}(f,g) \land I_x(f) \to I_x(g) \notag\\
  E_{x,y} (f,f') \land E_{y,z} (g,g') \land E_{x,z} (h,h') \land T_{x,y,z}(f,g,h) \shortorlong{\notag}{} \narrowbreak \to T_{x,y,z}(f',g',h').\label{eq:E-cong}
\end{gather}
A model $M$ is \textbf{standard} if $ME_{x,y}$ is equivalent to the actual equality of $MA(x,y)$.\footnote{With reference to ``identity of indiscernibles'' from \cref{sec:sip1}, standard equality amounts to adding haecceities into the structure explicitly.}
Like truncatedness, this will turn out to be a special case of our univalence condition.


\subsection{FOLDS-Categories in Univalent Foundations}
\label{sec:foldscat-uf}

As noted above, in HoTT/UF we must consider what truncation level $MO$ and $MA(x,y)$ should have.
In a precategory, we require the types of arrows to be sets, suggesting the following analogous definition.

\begin{definition}\label{defn:fc1}
  A \textbf{$1$-univalent FOLDS-category} $M$ is
  \begin{itemize}
  \item A type $MO:\U$;
  \item A family $MA:MO\times MO \to \U$;
  \item A family $MI : \prd{x:MO} MA(x,x) \to \U$;
  \item A family $MT : \prd{x,y,z:MO} MA(x,y) \to MA(y,z) \to MA(x,z) \to \U$; and
  \item A family $ME : \prd{x,y:MO} MA(x,y) \to MA(x,y) \to \U$,
  \end{itemize}
  such that
  \begin{itemize}
  \item Each type $MI_x(f)$, $MT_{x,y,z}(f,g,h)$, and $ME_{x,y}(f,g)$ is a proposition;
  \item Each type $MA(x,y)$ is a set;
  \item $ME_{x,y}(f,g) \leftrightarrow (f=g)$;
  \end{itemize}
  and the axioms of a category are satisfied.
\end{definition}

\begin{lemma}\label{lem:1-sat-folds-cat}
  The type of $1$-univalent FOLDS-categories is equivalent to the type of precategories.
\end{lemma}
\begin{proof}
  The underlying data of $MO$ and $MA$ are the same.
  In one direction, let $MI_{x}(f) \eqdef (f = 1_x)$ and $MT_{x,y,z}(f,g,h) \eqdef (h = g\circ f)$.
  In the other, let $1_x$ be the unique $f:MA(x,x)$ with $MI_x(f)$, and $g\circ f$ the unique $h$ with $MT_{x,y,z}(f,g,h)$.
\end{proof}

\begin{convention}
 Below, we sometimes abuse notation by writing $x : O$ instead of $x : MO$, and similarly for the other sorts, when the particular structure $M$ is clear from context.
\end{convention}

Let us now consider how to define ``univalent categories'' using only the FOLDS-structure. 
The central problem is to characterize the type $(a\cong b)$ of isomorphisms in such a way as can be readily generalized to other signatures.

To start with, recall that by the Yoneda lemma, an isomorphism $\phi : a\cong b$ in a category $\C$ is equivalently a natural family of isomorphisms of sets $\phi_{x\bullet} : \C(x,a) \cong \C(x,b)$, where naturality in $x$ means that $\phi_{y\bullet}(g) \circ f = \phi_{x\bullet}(g\circ f)$.
In the language of FOLDS-categories the operation $\circ$ is replaced by the relation $T$, with a new variable $h$ for the composite $g\circ f$:
\begin{itemize}
\item For each $x:O$, an isomorphism $\phi_{x\bullet}:A(x,a) \cong A(x,b)$; and \label{item:foldsiso1}
\item For each $x,y:O$, $f:A(x,y)$, $g:A(y,a)$, and $h:A(x,a)$, we have\label{item:foldsiso2}
  $T_{x,y,a}(f,g,h) \leftrightarrow T_{x,y,b}(f,\phi_{y\bullet}(g),\phi_{x\bullet}(h))$.
\end{itemize}
This looks more promising, but it still privileges one of the variables of $A$ over the other, and the relation $T$ over $I$ (and $E$).
More natural from the FOLDS point of view is to give equivalences between hom-sets with $a$ and $b$ substituted into \emph{all} possible ``collections of holes'':
\begin{align}
&\text{For any $x:O$, an isomorphism $\phi_{x\bullet}:A(x,a) \cong A(x,b)$;} \label{item:foldsiso1a} \\
&\text{For any $z:O$, an isomorphism $\phi_{\bullet z}:A(a,z) \cong A(b,z)$;} \label{item:foldsiso1b} \\
&\text{An isomorphism $\phi_{\bullet\bullet}:A(a,a) \cong A(b,b)$.}\label{item:foldsiso1c}
\end{align}
and similar logical equivalences between all possible ``relations with holes'':
\begin{align}
  T_{x,y,a}(f,g,h) &\leftrightarrow T_{x,y,b}(f,\phi_{y\bullet}(g),\phi_{x\bullet}(h)) \label{eq:Txya}\\
  T_{x,a,z}(f,g,h) &\leftrightarrow T_{x,b,z}(\phi_{x\bullet}(f),\phi_{\bullet z}(g),h) \label{eq:Txaz}\\
  T_{a,z,w}(f,g,h) &\leftrightarrow T_{b,z,w}(\phi_{\bullet z}(f),g,\phi_{\bullet w}(h)) \label{eq:Tazw}\\
  T_{x,a,a}(f,g,h) &\leftrightarrow T_{x,b,b}(\phi_{x\bullet}(f),\phi_{\bullet\bullet}(g),\phi_{x\bullet}(h)) \label{eq:Txaa}\\
  T_{a,x,a}(f,g,h) &\leftrightarrow T_{b,x,b}(\phi_{\bullet x}(f),\phi_{x\bullet}(g),\phi_{\bullet\bullet}(h)) \label{eq:Taxa}\\
  T_{a,a,x}(f,g,h) &\leftrightarrow T_{b,b,x}(\phi_{\bullet\bullet}(f),\phi_{\bullet x}(g),\phi_{\bullet x}(h)) \label{eq:Taax}\\
  T_{a,a,a}(f,g,h) &\leftrightarrow T_{b,b,b}(\phi_{\bullet\bullet}(f),\phi_{\bullet\bullet}(g),\phi_{\bullet\bullet}(h)) \label{eq:Taaa}\\
  I_{a,a}(f) &\leftrightarrow I_{b,b}(\phi_{\bullet\bullet}(f)) \label{eq:Iaa}  \\
   E_{x,a}(f,g) &\leftrightarrow E_{x,b}(\phi_{x\bullet }(f),\phi_{x\bullet}(g)) \label{eq:Exa}\\
   E_{a,x}(f,g) &\leftrightarrow E_{b,x}(\phi_{\bullet x}(f),\phi_{\bullet x}(g)) \label{eq:Eax}\\
   E_{a,a}(f,g) &\leftrightarrow E_{b,b}(\phi_{\bullet \bullet}(f),\phi_{\bullet \bullet}(g)) \label{eq:Eaa}
\end{align}
for all $x,y,z,w:O$ and $f,g,h$ of appropriate types.
Fortunately, the additional data here are redundant. Since $\phi_{x\bullet}$, $\phi_{\bullet z}$, and $\phi_{\bullet\bullet}$ preserve identities and $E$ is equivalent to identity by hypothesis, we obtain~\eqref{eq:Exa} to~\eqref{eq:Eaa}.
Just as~\eqref{eq:Txya} means the $\phi_{x\bullet}$ form a natural isomorphism,~\eqref{eq:Tazw} means the $\phi_{\bullet z}$ form a natural isomorphism, and~\eqref{eq:Txaz} means these natural isomorphisms arise from the same $\phi : a\cong b$.
Given this, any one of~\cref{eq:Txaa,eq:Taxa,eq:Taax} ensures that $\phi_{\bullet\bullet}$ is conjugation by $\phi$, and then the other two follow automatically, as do~\cref{eq:Taaa,eq:Iaa}.
This suggests the following definition.

\begin{definition}\label{defn:folds-iso-obj}
  For $a,b$ objects of a $1$-univalent FOLDS-category, an \defemph{indiscernibility} from $a$ to $b$ consists of data as in \cref{item:foldsiso1a,item:foldsiso1b,item:foldsiso1c} satisfying \crefrange{eq:Txya}{eq:Iaa}.
  We write $a \foldsiso b$ for the type of such indiscernibilities.
\end{definition}

\begin{theorem}\label{thm:iso-foldsiso}
  In any $1$-univalent FOLDS-category, the type of indiscernibilities from $a$ to $b$ is equivalent to the type of isomorphisms $a\cong b$.\qed
\end{theorem}


\begin{definition}
  A \defemph{$0$-univalent FOLDS-category} is a 1-uni\-val\-ent FOLDS-category such that for all $a,b:MO$, the canonical map $(a=b)\to (a\foldsiso b)$ is an equivalence.
\end{definition}


\begin{theorem}
  A 1-univalent FOLDS-category is $0$-univalent iff its corresponding precategory is a univalent category. \qed
\end{theorem}

The point is that the definition of indiscernibility can be derived algorithmically from the FOLDS-signature for categories, by an algorithm which applies equally well to any FOLDS-signature.
We will give this mechanism explicitly in \Cref{sec:FOLDS-iso-uni}. 
Then, for any $a,b:MK$ in some structure $M$, there will be a canonical map $(a=_K b)\to (a\foldsiso b)$, and we call $M$ \textbf{univalent} if these are equivalences.

However, 
there are two mismatches between this example so far and the general theory we have just proposed.
Firstly, we have assumed \textit{ad hoc} that $MA$ consists of sets. 
Secondly, we have just proposed that \emph{all} sorts should satisfy a univalence property, but in the example of categories we have only considered this for the sort $O$.
Fortunately, these two problems solve each other, and moreover remove the need to postulate ``standardness'' of equality. 

\begin{definition}
  A \textbf{$2$-univalent FOLDS-category} $M$ consists of the same type families $MO$, $MA$, $MI$, $MT$, $ME$ as a $1$-univalent FOLDS-category, 
  such that $ME$ is a congruence,
    each type $MI_x(f)$, $MT_{x,y,z}(f,g,h)$, and $ME_{x,y}(f,g)$ is a proposition,
  and the axioms of a category are satisfied with $ME$ used in place of equality.
\end{definition}

An indiscernibility between $f,g:A(a,b)$ in a $2$-uni\-valent FOLDS-category should consist of logical equivalences between instances of $T$, $I$, and $E$ with $f$ replaced by $g$ in ``all possible ways'', clearly beginning with
\begin{align}
  T_{x,a,b}(u,f,v) &\leftrightarrow T_{x,a,b}(u,g,v) \label{eq:fia1}\\
  T_{a,x,b}(u,v,f) &\leftrightarrow T_{a,x,b}(u,v,g) \label{eq:fia2}\\
  T_{a,b,x}(f,u,v) &\leftrightarrow T_{x,a,b}(g,u,v) \label{eq:fia3}
\end{align}
for all $x:O$ and $u,v$ of appropriate types.
But 
how do we put $f$ in two or three of the places in $T$ in the most general way?
In \Cref{sec:FOLDS-iso-uni} we will see that the answer is to assume an equality between objects and transport $f$ along it. 

\begin{definition}\label{defn:foldsiso-arrows}
  For $f,g:A(x,y)$ in a $2$-univalent FOLDS-category, an \defemph{indiscernibility} from $f$ to $g$ consists of the logical equivalences shown in~\crefrange{eq:fia1}{eq:fian}, for all $p:a=a$, $q:b=a$, and $r:b=b$.
\end{definition}

\begin{align}
  T_{a,a,b}(\trans{q}{f},f,u) &\leftrightarrow T_{a,a,b}(\trans{q}{g},g,u) \label{eq:fia4}\\
  T_{a,b,b}(\trans{p}{f},u,f) &\leftrightarrow T_{a,b,b}(\trans{p}{g},u,g) \label{eq:fia5}\\
  T_{a,a,b}(u,\trans{r}{f},f) &\leftrightarrow T_{a,a,b}(u,\trans{r}{g},g) \label{eq:fia6}\\
  T_{a,a,b} (\trans{(p,q)}f,\trans r f,f) &\leftrightarrow T_{a,a,b}(\trans{(p,q)}g,\trans r g,g) \label{eq:fia7}\\
  I_{a} (\trans q f) &\leftrightarrow I_a(\trans q g) \label{eq:fia8}\\
  E_{a,b}(f,u) &\leftrightarrow E_{a,b}(g,u) \label{eq:fia9}\\
  E_{a,b}(u,f) &\leftrightarrow E_{a,b}(u,g) \label{eq:fia10}\\
  E_{a,b} (\trans{(p,r)}f,f) &\leftrightarrow E_{a,b} (\trans{(p,r)}g,g)\label{eq:fian}
\end{align}

Since $T$, $I$, and $E$ are propositions, so is the type $f\foldsiso g$ of indiscernibilities. 
And $f\foldsiso f$, so by path induction we have $(f=g) \to (f\foldsiso g)$. 

\begin{theorem}\label{thm:2-univ-is-1-univ}
  A $2$-univalent FOLDS-category is $1$-univalent iff the map $(f=g) \to (f\foldsiso g)$ is an equivalence for all $f,g$.
\end{theorem}
\begin{proof}
  Since $f\foldsiso g$ is a proposition, the latter condition implies that each $A(a,b)$ is a set.
  Thus, for ``if'' it suffices to show $E_{a,b}(f,g)\Rightarrow (f\foldsiso g)$, which holds since $E$ is a congruence for $T$ and $I$.
  For ``only if'', we must show $(f\foldsiso g) \Rightarrow (f=g)$ in a FOLDS-category.
  But since $E_{a,b}(f,f)$ always, $f\foldsiso g$ implies $E_{a,b}(f,g)$, hence $f=g$ by standardness.
\end{proof}

Thus, by extending the ``univalence'' condition of a category from the sort $O$ to the sort $A$, we encompass automatically the assumption that the hom-types in a precategory are sets and that the equality is standard.

Finally, we can even stop treating the top sorts specially.

\begin{definition}
  A \textbf{FOLDS-category} consists of the same data and axioms as a $2$-univalent FOLDS-category, but without the assumption that the types $T$, $I$, and $E$ are propositions.
\end{definition}

Now, the type $t\foldsiso t'$ of indiscernibilities between $t,t':T_{x,y,z}(f,g,h)$ should consist of consistent equivalences between all types dependent on $t$ and $t'$.
But there are no such types in the signature, so $t\foldsiso t'$ is contractible.
The same reasoning applies to $I$ and $E$.
Thus, the univalence condition for these sorts will assert simply that all of their path-types are contractible, i.e., that they are propositions.

\begin{theorem}
  A FOLDS-category is $2$-univalent if and only if the canonical maps
    $(t= t') \to (t\foldsiso t')$,
    $(i=i') \to (i\foldsiso i')$, and
    $(e=e') \to (e\foldsiso e')$
   are equivalences
  for all inhabitants of the types $T$, $I$, and $E$ respectively. \qed
\end{theorem}

Thus, the notion of univalent category is determined only by the signature $\LcatE$, plus axioms (which are irrelevant for indiscernibilities and univalence).
Our goal is to define notions of indiscernibility and univa\-lence for any signature $\L$, generalizing the theory of univalent categories to arbitrary higher-categorical structures.

\section{Background: Two-Level Type Theory}\label{sec:2ltt}

In the following sections, we work in a two-level type theory (2LTT) as in \cite{2LTT}, with axioms (M2) (Russell-style universes), (T1), (T2), and (T3) from \cite[Section~2.4]{2LTT}.
Building on \cite{KL12}, 2LTT is shown in \cite[\S2.5]{2LTT} to be modeled by simplicial sets.

%



2LTT has an ``outer'' (a.k.a.\ ``strict'') level, a Martin-L\"{o}f type theory with intensional identity types and uniqueness of identity proofs (UIP), and an ``inner'' level, a homotopy type theory with univalent universes.
Both have their own---\textit{prima facie} distinct---type formers $\Pi$, $\Sigma$, $+$, $\onetype$, $\zerotype$, $\Nat$, intensional ``$=$'' with function extensionality, and universes.
By axioms (T1) and (T2), we can \defemph{identify inner types with particular outer ones} so that $\Pi$, $\Sigma$, and $\onetype$ are ``shared'' between the levels \cite[Lemma~2.11]{2LTT}, so we need not distinguish those notationally.
For other type constructors we annotate the outer variants with $^s$ (for ``strict''), e.g., in $\Nat^s$. (In \cite{2LTT}, the \emph{inner} type formers are annotated.)
We use the conventional typical ambiguity \cite[Section~1.3]{HTT} and hence refer to any universe by $\U$ (inner) resp.\ $\Ustrict$ (outer).
We use $\converts$ and $\define$ to denote judgmental equality, e.g., in definitions.

We call a type $A$ \defemph{fibrant} when it is isomorphic to an inner type $A'$ (in the strict sense, modulo $\steq$). 
Fibrancy is structure rather than property, but following \cite{2LTT} we abuse language by talking about a type ``being fibrant'' for simplicity.
Axiom (T3) states that \defemph{every fibrant type is inner}.
Thus, fibrant types are closed under $\Pi$ and $\Sigma$ (\cite[Lemma~3.5]{2LTT}).

A fibrant $A$ has two identity types: for $a,b : A$ we have the strict identity type $a \steq b$ that satisfies UIP, and the homotopical identity type $a = b$ that is at the center of HoTT.
We refer to elements of $a=b$ as ``\textbf{identifications}'', and elements of $a\steq b$ as ``\textbf{strict equalities}''.
Note that $a = b$ only eliminates into fibrant types, while $a \steq b$ eliminates into any type, fibrant or not. Consequently, for any fibrant type $A$ and $a, b : A$, we have a map $(a \steq b) \to (a = b)$. We sometimes use this implicitly to ``coerce'' a strict equality to an identification.
Similarly, we will frequently prove statements by induction on the strict natural numbers $\Nat^s$; there, we do not need to pay attention to the return type.
We write $A : \U$ to indicate that $A$ is a fibrant, or inner, type.

We write 
$A \simeq B$ for the type of equivalences between two (necessarily fibrant) types $A$ and $B$, in the usual sense of HoTT/UF.
The \emph{truncation level} of a fibrant type is defined as in \cite{HTT}, with contractible types the $(-2)$-types, and an $n$-type being a type whose homotopical identity types are $(n-1)$-types.
A \emph{proposition} is a $(-1)$-type, and a \emph{set} is a $0$-type.

An \textbf{$s$-category} $\C$ (see also \cite[Definition~3.1]{2LTT}) is given by the following data (``s'' for ``strict''):
\begin{enumerate}
\item A type $\ob{\C}$ of \emph{objects} (also often denoted $\C$);
\item For each $x,y \colon \C$ a type $\C(x,y)$ of \emph{arrows};
\item For each $x \colon \C$ an arrow $1 \colon \C(x,x)$; and
\item A \emph{composition} map $\circ \colon \C(y,z) \rightarrow \C(x,y) \rightarrow \C(x,z)$ that is strictly associative and for which $1$ is a strict left and right unit.
\end{enumerate}
A universe $\Ustrict$ gives rise to an s-category, also called $\Ustrict$, with objects $A : \Ustrict$ and morphisms $\Ustrict(A,B) \eqdef A \to B$.
An $s$-functor $F : \C \to \D$ consists of a function $\ob{F} : \ob{\C} \to \ob{\D}$ and functions $F_{x,y} : \C(x,y) \to \D(\ob{F}x,\ob{F}y)$ preserving identity and composition up to strict equality. We denote both $\ob{F}$ and $F_{x,y}$ by just $F$.
A strict natural transformation $\alpha : F \Rightarrow G : \C \to \D$ consists of a family of morphisms $(\alpha_x : \D(Fx, Gx))_{x : \ob{\C}}$ satisfying the naturality axiom strictly.


Our signatures will involve strict equality, and will thus live in the outer level of 2LTT. For a fixed signature, the types of structures, of maps between structures, and of indiscernibilities within a structure, will live entirely within the fibrant fragment of 2LTT.

\section{Signatures and Structures}\label{sec:signatures}

In traditional logic, a \emph{signature} specifies the sorts, functions, and relations of a structure.
A signature in dependent type theory must also specify the dependencies between sorts; Makkai \cite{MFOLDS} observed that this enables relations and, to a certain extent, functions, to be expressed merely in terms of sorts.%
\shortorlong{}
{%
\footnote{It is not unreasonable to directly include functions in addition to dependent sorts in a signature, obtaining something like Cartmell's \cite{Cart86} Generalized Algebraic Theories.  Indeed, a FOLDS-signature can be regarded as an especially simple sort of GAT; see also \cite[pp.~1--6]{MFOLDS}.  It is an interesting question for future work whether the results of this paper can be extended to more general GATs; for now we restrict ourselves to the simple case.
The relationship of our abstract signatures to GATs is unclear to us.}
}
Thus we could adapt Makkai's FOLDS-signatures to 2LTT and define structures as s-functors to $\U$ that are ``Reedy fibrant.'' 
However, it will be more convenient to formulate the notions of signature and structure \emph{inductively}, leading to a more general class of signatures.

Consider the FOLDS-signature $\Lrg$, for which a na{\"\i}ve structure consists of (fibrant) types and families $MO:\U$, $MA : MO \to MO \to \U$, and $MI : \prd{x:MO} MA(x,x) \to \U$.
If we strip off the \emph{top} sort $I$, the resulting structure contains only $MO$ and $MA$, and an inductive definition can be formulated along these lines.
But our inductions will be ``bottom-up'', so we want to strip off the \emph{bottom} sort $O$.
Once $MO:\U$ is fixed, the rest of an $\Lrg$-structure is determined by an ordinary structure over a \emph{derived} signature $\derivcat{(\Lrg)}{MO}$, with a rank-0 sort $A(x,y)$ for each $x,y:MO$, and a rank-1 sort $I(x)$ for each $x:MO$, with morphisms $I(x) \to A(x,x)$.
That is, we take the ``indexing'' of all sorts by $O$ and move it ``outside'' the signature, incorporating it into the types of sorts.\footnote{This would be impossible if our inverse categories were metatheoretic in the ordinary sense, e.g., syntactic and externally finite.
  2LTT is just right.}

This notion of \emph{derived FOLDS-signature} determines the notion of structure: a structure for $\L$ of height $p>0$ consists, inductively, of a family $\bottom{M}: \L(0) \to \U$ and a structure for $\derivcat{\L}{\bottom{M}}$ (which is of height $p-1$).
We can therefore abstract away from the inverse category underlying a FOLDS-signature, remembering only that each signature $\L$ of height $p>0$ has (1) a type $\L(0)$, and (2) for any $\bottom{M}: \L(0) \to \U$, a signature $\derivcat{\L}{\bottom{M}}$ of height $p-1$.


\begin{definition}[\defemph{Abstract signature}]
\label{def:abstract_signature}

We define a family of s-categories $\Sig(n)$ of signatures of height $n$ by induction.
Let $\Sig(0)$ be the trivial s-category on $\onetype$.

An object $\L$ of $\Sig(n+1)$ consists of
\begin{enumerate}
\item a fibrant type $\bottom{\L} : \U$;
\item a functor $\derivcat{\L}{} : (\bottom{\L} \to \U) \to \Sig(n)$, where $\bottom{\L} \to \U$ is the functor s-category from the discrete s-category $\bottom{\L}$ to the canonical s-category $\U$.
\end{enumerate}

\noindent
Arguments of $\derivcat{\L}{}$ will be written as subscripts, as in $\derivcat{\L}{M}$.

For $\L, \M: \Sig(n+1)$, an element $\alpha$ of $\hom_{\Sig(n+1)}(\L, \M)$ consists of the following:
 \begin{enumerate}
 \item a function $\bottom{\alpha} : \bottom{\L} \to \bottom{\M}$
 \item a strict natural transformation $\derivcat{\alpha}{}$ as in the diagram
   \[ 
    \begin{xy}
     \xymatrix@R=.5em@C=10em{
                    \bottom{\M} \to \U \ar[rd]^{\derivcat{\M}{}}
                                            \ar[dd]_{\_ \circ \bottom{\alpha}}
                    \\
                        & \Sig(n) \ltwocell<\omit>{\derivcat{\alpha}{}}
                    \\
                    \bottom{\L} \to \U \ar[ru]_{\derivcat{\L}{}}
     }
    \end{xy}
   \]

 \end{enumerate}
  Arguments of $\alpha'$ will also be written as subscripts, as in $\derivcat{\alpha}{M}$.
 
 Composition and identities are given by function composition and identity at $\bottomlevel$, and inductively for the derivative.
 Similarly, the categorical laws are easily proved by induction. 
\end{definition}

Similarly, we define \emph{$\L$-structures} inductively for $n : \Nat^s$ and each $L: \Sig(n)$.
The rank-0 part of an $\L$-structure is a type family $\bottom{M}:\bottom{\L}\to \U$, while the rest of an $\L$-structure consists of a structure for the derived signature $\derivcat{\L}{\bottom{M}}$.

\begin{definition}[\defemph{$\L$-structure}]

If $\L: \Sig(0)$, we define the type of \emph{$\L$-structures} to be $\Struc{\L} \define \onetype$.

If $\L:\Sig(n+1)$, we define the type of \emph{$\L$-structures} to be \[\Struc{\L} \define  \sm{\bottom{M} : {\bottom{\L}} \to \U } \Struc{\derivcat{\L}{\bottom{M}}}.\]

\end{definition}

For $\L: \Sig(n+1)$ we write $M \define (\bottom{M},M'): \Struc{\L}$. 

\begin{lemma}
 For any signature $\L$, the type $\Struc{\L}$ of $\L$-structures is fibrant. \qed
\end{lemma}


In the rest of the paper we work exclusively with abstract signatures, calling them simply ``signatures''.
However, since most intended examples arise naturally as FOLDS-signatures, we need to be able to translate FOLDS-signatures to abstract ones.
To this end, we define general FOLDS-signatures in 2LTT and translate them to abstract signatures, in such a way that ``Reedy fibrant'' diagrams on a FOLDS-signature coincide with structures for the corresponding abstract signature.
In our FOLDS-signatures, the equations of \Cref{fig:signatures} are formulated modulo \emph{strict equality}.
\shortorlong{This rather technical part of our work is explained in the extended version \cite{hsip_arxiv} of this article.}{We postpone the description of this rather technical work to \Cref{sec:folds-sigs}.}
Here we sketch the results of the translation for the examples in \Cref{fig:signatures}.

\begin{examples}\label{egs:folds-sigs}
  All three examples have only one sort of rank 0, so that $\bottom{\L} = \{O\} = \onetype$, and $\bottom{M}$ consists of a single type $MO$.
  Moreover, since all three examples have only one sort $A$ of rank 1 that depends on $O$ twice, their derivatives $\derivcat{\L}{MO}$ have $\bottom{(\derivcat{\L}{MO})} = MO\times MO$, and $\bottom{(\derivdia{M})}$ is a single type family $MA : MO \times MO \to \U$.
  Finally, since all three have height 3, $\derivcat{(\derivcat{\L}{MO})}{MA}$ has height 1, hence is just a single type.
  \begin{itemize}
  \item For $\Lrg$, this type is $\tsm{x:MO} MA(x,x)$, so that a structure is completed by a type family
    \[MI : \big(\tsm{x:MO} MA(x,x)\big) \to \U. \]
  \item For $\Lcat$, this type is
    \begin{narrowmultline*}
      \big(\tsm{x,y,z:MO} MA(x,y) \times MA(y,z)\times MA(x,z)\big) \narrowbreak + \big(\tsm{x:MO} MA(x,x)\big),
    \end{narrowmultline*}
    so that a structure is completed by a type family $MI$ as above together with
    \[ MT : \big(\tsm{x,y,z:MO} MA(x,y) \times MA(y,z)\times MA(x,z)\big) \to \U. \]
  \item Finally, for $\LcatE$, this type is
    \begin{narrowmultline*}
      \big(\tsm{x,y,z:MO} MA(x,y) \times MA(y,z)\times MA(x,z)\big)\narrowbreak
      + \big(\tsm{x:MO} MA(x,x)\big) + \big(\tsm{x,y:MO} MA(x,y) \times MA(x,y)\big),
    \end{narrowmultline*}
    so that a structure is completed by type families $MI$ and $MT$ as above together with
    \[ ME : \big(\tsm{x,y:MO} MA(x,y) \times MA(x,y)\big) \to \U. \]
  \end{itemize}
\end{examples}

\section{(Iso)morphisms of Structures}\label{sec:structures}

The definition of structures for signatures doesn't require the fact that signatures form an s-category.
But defining \emph{morphisms} of structures will require the pullback of an $\M$-structure along a morphism $\alpha : \L \to \M$ of signatures, defined as follows.

\begin{definition}
For any $\alpha: \hom_{\Sig(n)}(\L, \M)$, we define the \textbf{pullback} $\alpha^*: \Struc{\M} \to \Struc{\L}$ inductively as follows. 

If $n \eqdef 0$, then let $\alpha^*: \Struc{\M} \to \Struc{\L}$ be the identity. 

If $n > 0$, consider $M: \Struc{\M}$. We let $\bottom{ (\alpha^* M)}$ be $\bottom{M} \circ \bottom{\alpha}$. By induction, the morphism 
\[\derivcat{\alpha}{\bottom{M}}:  \hom_{\Sig(n-1)}(\derivcat{\L}{\bottom{M} \circ \bottom{\alpha}} , \derivcat{\M}{\bottom{M}})\]
produces a $(\derivcat{\alpha}{\bottom{M}})^*: \Struc{\derivcat{\M}{\bottom{M}}} \to \Struc{\derivcat{\L}{\bottom{M} \circ \bottom{\alpha}}}$, so we set $\derivdia{(\alpha^* M)} \define (\derivcat{\alpha}{\bottom{M}})^* \derivdia{M}$.
\end{definition}
Pullback is functorial: pullback along a composition of signature morphisms is the composition of pullbacks, and pullback along an identity morphism is the identity.

We now inductively define \emph{morphisms} between structures of a given signature, making $\Struc{\L}$ into an s-category.
\begin{definition}[\defemph{Morphism of structures}]
Consider $\L: \Sig(n)$ and $M,N : \Struc{\L}$.

When $n \eqdef 0$, we let $\hom_{\Struc{\L}}(M, N)\define \onetype$.

When $n > 0$, a morphism $f : \hom_{\Struc{\L}}(M,N)$ consists of
 \begin{enumerate}
 \item $\bottom{f}: \prd{K:\bottom{\L}} \bottom{M}(K) \to \bottom{N}(K)$
 \item $\derivdia{f}: \hom_{\Struc{\derivcat{\L}{\bottom{M}}}}(M',( \derivcat{\L}{\bottom{f}})^*\derivdia{N})$.
 \end{enumerate}
 \end{definition}

\begin{lemma}
 For a signature $\L$ and $\L$-structures $M$ and $N$, the type of morphisms from $M$ to $N$ is fibrant. \qed
\end{lemma}

As a stepping-stone to our SIP for univalent $\L$-structures, we show that \emph{all} $\L$-structures satisfy a tautological ``levelwise'' form of univalence.

\begin{definition}[\defemph{Isomorphism of structures}]
\label{def:iso-of-structures}
 Consider $\L: \Sig(n)$ and $M,N : \Struc{\L}$. 

If $n\eqdef 0$, we define every $f: \hom_{\Struc{\L}}(M,N)$ to be an $\L$-isomorphism. That is, we define
$\isIso_{\L}(f) \define \onetype$.

For $n > 0$,  $f: \hom_{\Struc{\L}}(M,N)$ is an \emph{$\L$-isomorphism} when
 \begin{enumerate}
 \item $\bottom{f}(K)$ is an equivalence of types for all $K: \bottom{\L}$, and
 \item $\derivdia{f}$ is an $\derivcat{\L}{\bottom{M}}$-isomorphism.
 \end{enumerate}
 That is, let 
 \[ \isIso_\L(f) \define \left(\prd{(K: \bottom{L})} \isEquiv(\bottom{f}(K))\right) \times \isIso_{\derivcat{\L}{\bottom{M}}} (\derivdia{f}). \]

We denote the type of $\L$-isomorphisms between two $\L$-structures $M,N$ by $M \cong_{\L}  N$, or simply $M \cong N$.
\end{definition}
\begin{lemma}\label{lem:isisoprop}
For any morphism $f : M \to N$ between two $\L$-structures, the type $\isIso_\L(f)$ is fibrant and a proposition. \qed
\end{lemma}

\begin{definition}[\defemph{Identity isomorphism}]
When $n \eqdef 0$, we define $i_M$ to be the canonical element in $\onetype$.

Otherwise, we have $ 1_{\derivcat{\L}{\bottom{M}}} \steq \derivcat{\L}{1_{\bottom{M}}} $ (from functoriality of $\derivcat{\L}{}$), hence $u:  (1_{\derivcat{\L}{\bottom{M}}})^*M' \steq (\derivcat{\L}{1_{\bottom{M}}} )^*M' $.
We also have a strict equality $v:  M' \steq (1_{\derivcat{\L}{\bottom{M}}})^*M' $ (by the functoriality of pullback).
Then we set $i_{M}$ to be the pair $(1_{\bottom{M}},\idtoiso(v \concat u))$.
\end{definition}

Now define $\idtoiso: \prd{M,N: {\Struc{\L}}} (M = N) \to (M \cong N)$ by sending $\refl_M $ to $ i_M$.

\begin{proposition}\label{prop:uni1}
 For structures $M,N$ of a signature $\L$, the canonical map 
 \[ \idtoiso_{M,N}: (M = N) \to (M \cong N) \]
 is an equivalence of types.
\end{proposition}
\begin{proof}
  When $n \eqdef 0$, $ \idtoiso:\onetype \to \onetype$, hence is an equivalence.

Let $\ua : (\bottom{M} \cong \bottom{N}) \to (\bottom{M}  = \bottom{N} )$ be given by the univalence axiom. 
First we show that $\trans{\ua(e)^{-1}}{\derivdia{N}} = (\derivcat{\L}{e})^*\derivdia{N}$ for any $e: \bottom{M} \cong \bottom{N}$, where $\transfun{\ua(e)^{-1}}$ denotes transport along $\ua(e)^{-1}$.
\begin{figure*}
\(
\xymatrix{
(\bottom{M} \cong \bottom{N}) \ar[r]^{\ua }
& (\bottom{M} = \bottom{N}) \ar[d]^{(-)^{-1}} \ar[r]^-{\idtoiso} & (\bottom{M} \cong \bottom{N}) \ar[r]^-{\derivcat{\L}{-}}&\hom_{\Sig(n)}(\derivcat{\L}{\bottom{M}} , \derivcat{\L}{\bottom{N}}) \ar[d]^{(-)^*}\\
& (\bottom{N} = \bottom{M}) \ar[rr]^-{(-)_*} & & \Struc{\derivcat{\L}{\bottom{N}}} \to \Struc{\derivcat{\L}{\bottom{M}}}}  
\)
\caption{Diagram for Proof of \Cref{prop:uni1}}
\label{fig:diagram}
\Description{A commutative diagram}
\end{figure*}
Now the square in \Cref{fig:diagram} commutes (up to $=$) since both functions $(\bottom{M} = \bottom{N}) \to \Struc{\derivcat{\L}{\bottom{N}}} \to \Struc{\derivcat{\L}{\bottom{M}}} $ send $\refl_{\bottom{M}} $ to $ 1_{\Struc{\derivcat{\L}{M}}}$ (by strict functoriality of the pullback). Precomposing these with $\ua$, we find that $(\derivcat{\L}{e})^*\derivdia{N} = \trans{\ua(e)^{-1}}{\derivdia{N}}$.
Now we have that
 \begin{align*}
 (M = N) & = \sm{p:\bottom{M} = \bottom{N}} \derivdia{M} = \trans{p^{-1}}{\derivdia{N}} \\
 &= \sm{e:\bottom{M} \cong \bottom{N}} \derivdia{M} = \trans{\ua(e)^{-1}}{\derivdia{N}} \\
 &= \sm{e:\bottom{M} \cong \bottom{N}} \derivdia{M} = (\derivcat{\L}{e})^*\derivdia{N} \\
 &= \sm{e:\bottom{M} \cong \bottom{N}} \derivdia{M} \cong (\derivcat{\L}{e})^*\derivdia{N} \\
 &\equiv ( M \cong N )
 \end{align*}
 where the second identification is the univalence axiom and the fourth is our inductive hypothesis. This equivalence, from left to right, is $\idtoiso_{M,N}$.
 \end{proof}

\Cref{prop:uni1} relies on the univalence axiom;
conversely, the univalence axiom can be recovered as an instance of \Cref{prop:uni1}, for the signature consisting of just one sort.

\begin{example}
  When precategories are regarded as $\Lcat$-structures, their isomorphisms are the \emph{isomorphisms of precategories} from \cite[Def.~6.9]{AKS13} and~\cite[Def. 9.4.8]{HTT}: functors that induce equivalences on hom-types and also equivalences on types of objects (relative to homotopical \emph{identifications} of objects, not isomorphisms in the category structure).
\end{example}

\begin{remark}
 We expect that isomorphisms of structures can equivalently be characterized via the existence of a structure morphism in the other direction and composites that are (homotopically) identical to identities.
\end{remark}

The analogue for $\L$-structures of \emph{equivalences} of precategories, called \emph{(split-surjective) equivalences} of $\L$-structures, will be introduced in \Cref{sec:hsip}.
Our main result, \Cref{thm:hsip}, will be that between univalent $\L$-structures these are also equivalent to identifications.
However, first we have to define univalence of $\L$-structures.

\section{Indiscernibility and Univalence}\label{sec:FOLDS-iso-uni}

In this section we define indiscernibility of objects within an $\L$-structure. 
We then define a structure to be univalent when indiscernibility coincides with identification of objects.

Let $M$ be an $\L$-structure, $K : \bottom{\L}$, and $a, b : \bottom{M}K$.
To define indiscernibilities from $a$ to $b$, we
consider a new $\L$-structure obtained by adding to $M$ one element at sort $K$: a ``joker'' element. 
We can substitute this new element by $a$ or by $b$; 
below, we call the obtained structures $\partial_a M$ and $\partial_b M$, respectively. 
An indiscernibility from $a$ to $b$ will be defined below to be an isomorphism of structures from $\partial_a M$ to $\partial_b M$ that is the identity on all the sorts not depending on the joker element.
Intuitively, this means that $a$ and $b$ are isomorphic when one cannot discern one from the other using the rest of the structure $M$.
To make this intuition formal, we need two auxiliary definitions:
\begin{definition}\label{def:ind}
 Consider $L: \U$, $K:L$, $M: L \to \U$, $a: M(K)$. We define the \defemph{indicator function of $K$} to be
 \[ [K] \eqdef  \lambda x. \left(  x = K \right) : L \to \U\] 
 and we define the
  function $a: \prd{x: L} [K](x) \to M(x)$ by sending $\refl_K: [K](K)$ to $a:M(K)$. 
  
Below we consider the pointwise disjoint union $M+[K]$ in $L \to \U$, the canonical injection $\iota_M: \prd{x: L} M(x) \to (M+[K])(x)$, and the induced function $\copair{1_M}{a}:  \prd{x: L}  (M+[K])(x)\to M(x) $.
\end{definition}

\begin{definition}
Consider $\L: \Sig(n+1)$, $K:\bottom{\L}$, $M: \Struc{\L}$, $a: \bottom{M}(K)$.
Define 
\[\partial_a M  \eqdef  (\derivcat{\L}{\copair{1_{\bottom{M}}}{a}})^*M' : \Struc{\derivcat{\L}{\bottom{M} + [K]}}.\]
 \end{definition}
 
Note that we require the type $\bottom{\L}$ to be fibrant so that the fibrant indicator function $[K]$ exists.

Now we can define the type of indiscernibilities between objects within an $\L$-structure: 
 
\begin{definition}[\defemph{Indiscernibility}]
\label{def:iso-within-a-structure} 
Consider $\L : \Sig(n+1)$, $K:\bottom{\L}$, $M: \Struc{\L}$, $a,b: \bottom{M}(  K)$. 
We define the type of \defemph{indiscernibilities from $a$ to $b$} to be
\[ 
(a \fiso b) \eqdef \sm{p \colon  \partial_a M \: = \: \partial_b M} \epsilon_a^{-1} \concat (\derivcat{\L}{\iota_{\bottom{M}}})^*p \concat \epsilon_b =_{M' = M'} \refl_{M'}, 
\]
where $\epsilon_{x}$ is the concatenated identification
\begin{align*}
(\derivcat{\L}{\iota_{\bottom{M}}})^* \partial_x M'  &\converts (\derivcat{\L}{\iota_{\bottom{M}}})^* (\derivcat{\L}{\langle 1_{\bottom{M}}, x \rangle})^*  M'\\ &= (\derivcat{\L}{  \langle 1_{\bottom{M}}, x \rangle \circ \iota_{\bottom{M}}})^* M'=  (\derivcat{\L}{  1_{\bottom{M}}})^* M' = M'.
\end{align*}
\end{definition}

\begin{lemma}
 The type $a \fiso b$ of \Cref{def:iso-within-a-structure} is fibrant. \qed
\end{lemma}

\begin{remark}
  Using identification instead of isomorphism of structures in \Cref{def:iso-within-a-structure} is justified by \Cref{prop:uni1}.
\end{remark}

\begin{lemma}\label{lem:iso-equivalent}
 The type of indiscernibilities $a \fiso b$ of \Cref{def:iso-within-a-structure} is equivalent to the type
\[ 
\pushQED{\qed} 
 \sm{p \colon  \partial_a M \: = \: \partial_b M}  (\derivcat{\L}{\iota_{\bottom{M}}})^*p = \epsilon_a \concat \epsilon_b^{-1}.
 \qedhere
\popQED
\]
\end{lemma}

We now define \emph{univalence of $\L$-structures}. 
For this, we first need to define the canonical map from identifications to indiscernibilities.

\begin{definition}[\defemph{Identity indiscernibility}]
\label{def:id-iso}
For $\L: \Sig(n+1)$, $K:\bottom{\L}$, $M: \Struc{\L}$, and $m : \bottom{M}(K)$, we define the indiscernibility $1 : m \fiso m$.
Let $M: \Struc{\L}$.
For any $a: \bottom{M}(K)$, we have $\refl_{\partial_a M}: \partial_a M = \partial_a M$. Then
\begin{align*} \MoveEqLeft \epsilon_a^{-1} \concat (\derivcat{\L}{\iota_M})^*(\refl_{\partial_a M}) \concat \epsilon_a  
  \\
 &\steq \epsilon_a^{-1} \concat \refl_{(\derivcat{\L}{\iota_M})^*\partial_a M} \concat \epsilon_a  \\
 &=  \refl_{M'} \enspace ,
\end{align*}
where the second identification uses the groupoidal properties of types.
This gives the desired indiscernibility.
\end{definition}

\begin{definition}
Consider $\L: \Sig(n+1)$, $K:\bottom{\L}$, $M: \Struc{\L}$. 
For any $a,b: \bottom{M}(K)$, let $\idtoiso_{a,b}: ( a = b) \to (a \fiso b) $ be the function which sends $\refl_a$ to the identity indiscernibility exhibited in \Cref{def:id-iso}. 

We say that $M$ is \defemph{univalent at $K$} if for all $a,b : \bottom{M}(K)$, $\idtoiso_{a,b} : (a = b) \to (a \fiso b)$ is an equivalence.
\end{definition} 
\begin{definition}[\defemph{Univalence of structures}]
We define by induction what it means for a structure of a signature $\L:\Sig(n)$ to be univalent. 

When $n \eqdef 0$, every structure $M: \Struc{\L}$ is univalent. 

Otherwise, a structure $M: \Struc{\L}$ is univalent if $M$ is univalent at all $K:\bottom{\L}$ and $M'$ is univalent.

Let $\uStruc{\L}$ denote the type of univalent structures of $\L$.
\end{definition} 

\begin{lemma}
 Let $\L$ be a signature.
  \begin{itemize}
  \item The type $\uStruc{\L}$ is fibrant. 
  \item For any $\L$-structure, ``being univalent'' is a proposition.
  \item Identification of univalent $\L$-structures corresponds
         to identification of the underlying $\L$-struc\-tures.
   \qed
  \end{itemize}

\end{lemma}

\begin{example}\label{eg:ht1-univalence}
  Suppose $\L$ has height 1, hence is just a type $\bottom{\L}$.
  Consider an $\L$-structure $M:\bottom{\L}\to \U$ and $a,b:M(K)$.
  Then $\partial_a M$ and $\partial_b M$ are structures for the trivial signature of height 0, hence uniquely identified; thus $(a \fiso b)=\onetype$.
  So any structure of a signature $\L$ of height 1 is univalent just when it consists entirely of propositions.
\end{example}

\begin{example}
  Recall from \cref{egs:folds-sigs} that for $\L=\LcatE$, we have $\bottom{\L} = \onetype$, $\bottom{M}= MO:\U$, $\bottom{\derivcat{\L}{MO}} = MO \times MO$, and $\bottom{(\derivdia{M})} = MA : MO\times MO \to \U$, while
   $\doublederivdia{M}$ consists of the sorts $MT_{x,y,z}(f,g,h)$, $MI_x(f)$, and $ME_{x,y}(f,g)$.
  By \Cref{eg:ht1-univalence}, $\doublederivdia{M}$ is univalent just when all these types are propositions.
  Now for any $a,b:MO$, we have
  \[ (MA+[A(a,b)])(x,y) = MA(x,y) + ((a=x)\times (b=y)). \]
  Thus, the height-1 signature $\derivcat{(\derivcat{\L}{MO})}{MA+[A(a,b)]}$ is 
  \begin{multline*}
    \big(\tsm{x,y,z:MO} (MA(x,y) + ((a=x) \times (b=y))) \narrowbreak \times (MA(y,z) + ((a=y)\times (b=z)))\narrowbreak \times (MA(x,z) + ((a=x)\times (b=z)))\big) \\
    +\big(\tsm{x:MO} (MA(x,x) + ((a=x) \times (b=x)))\big)\\
    + \big(\tsm{x,y:MO} (MA(x,y) + ((a=x) \times (b=y))) \narrowbreak \times (MA(x,y) + ((a=x) \times (b=y))) \big).
  \end{multline*}
  By distributing $\sum$ and $\times$ over $+$
  and contracting some singletons, this is equivalent to
  \begin{align}
    \MoveEqLeft \big(\tsm{x,y,z:MO} MA(x,y)\times MA(y,z) \times MA(x,z)\big) \label{eq:sia0} \\
    &+ \big(\tsm{z:MO} MA(b,z)\times MA(a,z)\big)\label{eq:sia1}\\
    &+ \big(\tsm{x:MO} MA(x,a)\times MA(x,b)\big)\label{eq:sia2}\\
    &+ \big(\tsm{y:MO} MA(a,y)\times MA(y,b) \big)\label{eq:sia3}\\
    &+ \big((a=b)\times MA(a,b)\big)\label{eq:sia4}\\
    &+ \big(MA(a,a)\times (b=b)\big)\label{eq:sia5}\\
    &+ \big((a=a)\times MA(b,b)\big)\label{eq:sia6}\\
    &+ \big((a=b)\times (a=a)\times (b=b)\big)\label{eq:sia7}\\
    &+ \big(\tsm{x:MO} MA(x,x)\big) \label{eq:sia00}\\
    &+ \big((a=b)\big)\label{eq:sia8}\\
    &+ \big(\tsm{x,y:MO} MA(x,y) \times MA(x,y)\big) \label{eq:sia000}\\
    &+ \big(MA(a,b)\big)\label{eq:sia9}\\
    &+ \big(MA(a,b)\big)\label{eq:sia10}\\
    &+ \big((a=a) \times (b=b)\big).\label{eq:sian}
  \end{align}
  Thus for $f,g:MA(a,b)$, an identification $\partial_f M = \partial_g M$ consists of equivalences between instances of the predicates $MT,MI,ME$ indexed over the types~\eqref{eq:sia0}--\eqref{eq:sian}.
  The condition on restriction along $\iota$ says that the equivalences corresponding to~\eqref{eq:sia0}, \eqref{eq:sia00}, and~\eqref{eq:sia000} are the identity, while those corresponding to~\eqref{eq:sia1}--\eqref{eq:sia3}, \eqref{eq:sia4}--\eqref{eq:sia7}, \eqref{eq:sia8}, and~\eqref{eq:sia9}--\eqref{eq:sian} yield respectively the equivalences~\eqref{eq:fia1}--\eqref{eq:fia3}, \eqref{eq:fia4}--\eqref{eq:fia7}, \eqref{eq:fia8}, and~\eqref{eq:fia9}--\eqref{eq:fian} from \Cref{sec:foldscat-uf}.
  Hence, indiscernibilities $f\fiso g$ in the sense of \cref{def:iso-within-a-structure} coincide with the indiscernibilities from \cref{defn:foldsiso-arrows}.

  Now moving back down to the bottom rank, an $(\derivcat{\L}{MO})$-structure consists of $MA : MO\times MO \to \U$ together with appropriately typed families $MT$, $MI$, and $ME$.
  Since \( (MO + [O]) = MO + \onetype\), for $a:MO$ the $0^{\mathrm{th}}$ rank of $\partial_a M$ is
  \[ (\partial_a M)A : (MO+\onetype) \times (MO+\onetype) \to \U \]
  or equivalently 
  \[ (\partial_a M)A : (MO\times MO) + MO + MO + \onetype \to \U \]
  consisting of the types $(MA(x,y))_{x,y:MO}$, $(MA(a,y))_{y:MO}$, $(MA(x,a))_{x:MO}$, and $MA(a,a)$.
  The $1^{\mathrm{st}}$ rank consists of $MT$, $MI$, and $ME$ pulled back appropriately to these families.
  Thus, an identification $\partial_a M = \partial_b M$ consists of equivalences
  \begin{align}
    MA(x,y) &\simeq MA(x,y) \label{eq:sio0}\\
    MA(x,a) &\simeq MA(x,b) \label{eq:sio1}\\
    MA(a,y) &\simeq MA(b,y) \label{eq:sio2}\\
    MA(a,a) &\simeq MA(b,b) \label{eq:sio3}
  \end{align}
  for all $x,y:MO$ that respect the predicates $MT$, $MI$, $ME$.
  The condition on restriction along $\iota$ says that the equivalences~\eqref{eq:sio0} are the identity, while the remaining~\eqref{eq:sio1}--\eqref{eq:sio3} correspond respectively to the equivalences $\phi_{x\bullet}$, $\phi_{\bullet y}$, and $\phi_{\bullet\bullet}$ from \Cref{sec:foldscat-uf}.
  Finally, respect for $MT$, $MI$, $ME$ specializes to~\eqref{eq:Txya}--\eqref{eq:Iaa} together with analogous equivalences for $E$ that are trivial under ``standardness'' of identifications.
  Thus, indiscernibilities $a\fiso b$ in the sense of \cref{def:iso-within-a-structure} coincide with the indiscernibilities from \cref{defn:folds-iso-obj}.
\end{example}

Our first general observations about univalent structures give truncation bounds for their sorts and for the type of such structures.

\begin{theorem}\label{thm:hlevel}
Let $\L: \Sig(n+1)$, $M: \uStruc{\L}$, $K: \bottom{\L}$. Then $\bottom{M}(K)$ is an \nminusone-type.
\end{theorem}

\begin{theorem}\label{thm:hlevel1}
Let $\L: \Sig(n)$. The type of univalent $\L$-struc\-tures is an \nminusone-type.
\end{theorem}

\begin{proof}[Proof of \Cref{thm:hlevel,thm:hlevel1}]
Define the following types.
\begin{align*}
P(n) &\define \prd{\L: \Sig(n+1)} \prd{M: \uStruc{\L}} \prd{K: \bottom{\L}} \istype{\nminusone}(\bottom{M}(K)) \\
Q(n) &\define \prd{\substack{\M,\N:\Sig(n) \\ \alpha: \hom(\M  ,\N)}} \prd{N: \uStruc{\N}} \istype{(n-2)} (\alpha^* N = \alpha^* N)
\end{align*}
The type $P(n)$ is the statement of \Cref{thm:hlevel}, and the type $Q(n)$ implies the statement of \Cref{thm:hlevel1} by \cite[Thm.~7.2.7]{HTT}.
We prove $P(n)$ and $Q(n)$ simultaneously.

For $P(n)$, we need to show that $a = _{\bottom{M} K} b$ is an $(n-2)$-type for all $\L: \Sig(n+1), M: \uStruc{\L}, K: \bottom{\L}, a,b: {\bottom{M} K}$. But since $M$ is univalent, this type is equivalent to
\[ (a \fiso b) \equiv \sum_{e: \partial_a M = \partial_b M} \epsilon_a^{-1} \concat (\derivcat{\L}{\iota_M})^*p \concat \epsilon_b =_{M' = M'} \refl_{M'}.\] Thus, it will suffice to show that $\partial_a M = \partial_b M$ and $\epsilon_a^{-1} \concat (\derivcat{\L}{\iota_M})^*p \concat \epsilon_b =_{M' = M'} \refl_{M'}$ are $(n-2)$-types.

To show $P(0)$ and $Q(0)$ consider $\L: \Sig(1), M: \uStruc{\L}, K: \bottom{\L}, a,b: {\bottom{M} K}, \M,\N:\Sig(0), \alpha: \hom(\M  ,\N), N: \uStruc{\N}$.
We have that $\derivdia{M}, \partial_a M,\partial_b M, \alpha^* N: \onetype$ so the types $\partial_a M = \partial_b M$, $\epsilon_a^{-1} \concat (\derivcat{\L}{\iota_M})^*p \concat \epsilon_b =_{M' = M'} \refl_{M'}$, and $ \alpha^* N = \alpha^* N$  are contractible. Thus, $P(0)$ and $Q(0)$ hold.

Suppose that $P(n)$ and $Q(n)$ hold. We first show $Q(n+1)$. Consider $\M,\N:\Sig(n+1), \alpha: \hom(\M  ,\N), N: \uStruc{\N}$. We have that 
\begin{align*} (\alpha^* N=\alpha^* N) 
&\simeq \Sigma_{e: \bottom{ (\alpha^* N)} = \bottom{ (\alpha^* N)}} (\alpha^*N) ' =  e_*(\alpha^* N)' \\
&\equiv \Sigma_{e: (\bottom{N} \circ \bottom{\alpha}) = (\bottom{N} \circ \bottom{\alpha})} (\derivcat{\alpha}{\bottom{N}})^* \derivdia{N} =  e_*(\derivcat{\alpha}{\bottom{N}})^* \derivdia{N}.
\end{align*}
Our inductive hypothesis $Q(n)$ ensures that $(\derivcat{\alpha}{\bottom{N}})^* \derivdia{N} = (\derivcat{\alpha}{\bottom{N}})^* \derivdia{N}$ is an $(n-2)$-type, and hence $(\derivcat{\alpha}{\bottom{N}})^* \derivdia{N} =  e_*(\derivcat{\alpha}{\bottom{N}})^* \derivdia{N}$  is an $(n-1)$-type by \cite[Thm.~7.2.7]{HTT}. 
It remains to show that $(\bottom{N} \circ \bottom{\alpha}) = (\bottom{N} \circ \bottom{\alpha})$ is an $(n-1)$-type. 
Note that $N$ is a univalent structure of an $(n+1)$-signature, and our inductive hypothesis $P(n)$ then implies that for all $K: \bottom{\N}$, the type $ \bottom{N}(K)$ is an $(n-1)$-type.
Then since $(\bottom{N} \circ \bottom{\alpha})$ is a function which takes values in $(n-1)$-types, we can conclude that $(\bottom{N} \circ \bottom{\alpha}) = (\bottom{N} \circ \bottom{\alpha})$ is an $(n-1)$-type \cite[Thm.~7.1.9]{HTT}. Thus, $Q(n+1)$ holds.

To show that $P(n+1)$ holds, consider $\L: \Sig(n+2), M: \uStruc{\L}, K: \bottom{\L}, a,b: {\bottom{M} K}$.
By \cite[Thm.~7.2.7]{HTT}, $Q(n+1)$ implies that $\partial_a M = \partial_b M$ and $\epsilon_a^{-1} \concat (\derivcat{\L}{\iota_M})^*p \concat \epsilon_b =_{M' = M'} \refl_{M'}$ are $(n-2)$-types. Therefore, $P(n+1)$ holds.
\end{proof}

\begin{example}
 For the signature $\LcatE$ of height $3$, \Cref{thm:hlevel} states that the type of objects of a univalent $\LcatE$-structure, and hence also of a univalent FOLDS-category, is a $1$-type. \Cref{thm:hlevel1} states that the type of univalent $\LcatE$-structures, and hence also the type of univalent FOLDS-categories (as a subtype of the former), is a $2$-type.
\end{example}

\section{Equivalence of Structures and \shortorlong{}{the }Higher SIP}
\label{sec:hsip}

Our Higher Structure Identity Principle (\shortorlong{\Cref{thm:hsip}}{\Cref{thm:hsip2}}) requires a notion of equivalence \emph{of} structures that is \textit{a priori} weaker than \Cref{def:iso-of-structures}.
\shortorlong{Makkai defined a morphism of structures $\mu:M\to N$ to be \emph{very surjective}\footnote{Very surjective morphisms are also known as ``Reedy surjections'' and ``trivial fibrations''.} if it is ``locally'' surjective at all sorts, which in our notation means that $\bottom{\mu}(K)$ is surjective for all $K$ and $\derivdia{\mu}$ is (inductively) very surjective.
For instance, a morphism between FOLDS-categories is very surjective if (1) it is surjective on objects, (2) it is \emph{full}, i.e., each map $MA(x,y) \to NA(fx,fy)$ is surjective, and (3) it reflects the relations $T,I,E$. 
When equality is standard, (3) is equivalent to faithfulness, so this is just a surjective weak equivalence of categories.
With this motivation, we use similar terminology for structures over arbitrary signatures.}%
{In the case of (pre)categories, there are two natural candidates for such a notion:
\begin{itemize}
\item A \emph{weak equivalence} is a functor $f:M\to N$ that is fully faithful (each function $MA(x,y) \to NA(fx,fy)$ is an isomorphism of sets) and essentially surjective ($\prd{y:NO} \Vert \sm{x:MO} (fx \cong y) \Vert$).
\item A \emph{(strong) equivalence} is a functor $f:M\to N$ for which there is a functor $g:N\to M$ and natural isomorphisms $f g \cong 1_N$ and $g f \cong 1_M$.
  By~\cite[Lemma 6.6]{AKS13}, this is equivalent to being fully faithful and \emph{split} essentially surjective ($\prd{y:NO} \sm{x:MO} (fx \cong y)$).
\end{itemize}
In addition, there are two important related auxiliary notions:
\begin{itemize}
\item A \emph{surjective weak equivalence} is a functor $f:M\to N$ that is fully faithful (each function $MA(x,y) \to NA(fx,fy)$ is an isomorphism of sets) and surjective on objects ($\prd{y:NO} \Vert \sm{x:MO} (fx=y) \Vert$).
\item A \emph{split-surjective equivalence} is a functor $f:M\to N$ that is fully faithful (each function $MA(x,y) \to NA(fx,fy)$ is an isomorphism of sets) and split-surjective on objects ($\prd{y:NO} \sm{x:MO} (fx=y)$).
\end{itemize}
Note that the latter two do not require knowing what an isomorphism between objects is.
Furthermore, fully-faithfulness can be split into fullness (each function $MA(x,y) \to NA(fx,fy)$ is surjective\footnote{Or split-surjective; in the presence of faithfulness the two are equivalent.}) and faithfulness (each function $MA(x,y) \to NA(fx,fy)$ is injective), while faithfulness is equivalent to \emph{surjectivity on equalities}: each function $ME_{x,y}(p,q) \to NE_{fx,fy}(fp,fq)$ is surjective (which implies a similar property for $T$ and $I)$.
This suggests the following generalizations that apply to structures over any signature.}

\begin{definition}[\defemph{Split-surjective equivalence}]
  Suppose $f:\hom_{\Struc{\L}}(M,N)$, where $M,N: \Struc{\L}$ and $\L: \Sig(n)$.
  If $n\eqdef 0$, then $f$ is a split-surjective equivalence.
  For $n>0$, $f$ is a split-surjective equivalence if
  \begin{enumerate}
  \item $\bottom{f}(K)$ is a split surjection for every $K:\bottom{\L}$, and
  \item $\derivdia{f}$ is a split-surjective equivalence.
  \end{enumerate}
  \defemph{Surjective weak equivalences} are defined similarly, but only requiring each $\bottom{f}(K)$ to be surjective.
\end{definition}

\shortorlong{}{Makkai defined surjective weak equivalences under the name \emph{very surjective morphisms}; other names for them include \emph{Reedy surjections} and \emph{trivial fibrations}.}
Unfortunately, we are currently unable to prove our desired general result with surjective weak equivalences, so for the present we restrict to the split-surjective equivalences.
We write $\VSS(f)$ for the type ``$f$ is a split-surjective equivalence'', which in the inductive case is
\[ \VSS(f) \define \left(
    \shortorlong{\prd{K: \bottom{\L}} \SplitS{\bottom{f}K }}{ \prd{K: \bottom{\L}}{y:\bottom{N}(K)} \sm{x:\bottom{M}(K)} (\bottom{f}(K)(x) = y) }
  \right) \times \VSS(\derivdia{f}),\]  
and $(M \twoheadrightarrow N) \eqdef \sm{f : \hom_{\Struc{\L}}(M,N)}\VSS(f)$ for the type of split-surjective equivalences.



\begin{lemma}
  The type $M \twoheadrightarrow N$ of split-surjective equivalences is fibrant.\qed
\end{lemma}

\begin{definition}[\defemph{From isomorphisms to split-surjective equivalences}]\label{def:isotovss}
  Let $f : \hom_{\Struc{\L}}(M, N)$; we define $U_f: \isIso(f) \to \VSS(f)$ by induction on $n$.
  If $n \eqdef 0$, $U_f$ is the identity function on $\onetype$.
  For $n>0$, we use that any equivalence of types is a split surjection, and the inductive hypothesis.
  Let $\isotovss_{M,N} \eqdef (1,U) : (M \cong N) \to (M \twoheadrightarrow N)$. 
\end{definition}

\begin{definition}[\defemph{From identifications to split-surjective equivalences}]
  For $\L: \Sig(n)$ and $M,N: \Struc{\L}$ we define 
  \[\idtovss \eqdef \isotovss \circ \idtoiso : (M = N) \to (M \strucequiv N).\]
\end{definition}

Our \shortorlong{}{first }HSIP states that if $M$ is univalent, then $\idtovss_{M,N}$ is an equivalence.
It uses the following lemma.

 \begin{lemma} \label{lem:injwrtiso}
Let $\L: \Sig(n+1)$, $M, N : \Struc{\L}$, $\bottom{f}: \bottom{M} \to \bottom{N}$, and $e: \derivdia{M} = (\derivcat{\L}{\bottom{f}})^* N'$. Then for $x,y: \bottom{M}(K)$, an indiscernibility $\bottom{f}x \fiso \bottom{f} y$ produces an indiscernibility $x \fiso y$.
\end{lemma}
\begin{proof}
By path induction on $e$, we may assume $M' \converts (\derivcat{\L}{\bottom{f}})^* N'$.

Consider the following diagram whose cells commute up to $\steq$ or $=$, as pictured.
\begin{equation}
\begin{tikzcd}[column sep=huge]
\bottom{M} \ar[dr,phantom,"\steq" description]  \ar[rr,phantom,"\steq", bend left=15]  \ar[d,"\bottom{f}"] \ar[r,"\iota_{\bottom{M}}"] \ar[rr, bend left=20, shift left=1,"1_{\bottom{M}}",near start] & \bottom{M} + [K] \ar[d,"\bottom{f}+1"] \ar[r,
"{\copair 1 x}"] & \bottom{M} \ar[d, "\bottom{f}"]\\
\bottom{N}  \ar[rr,phantom,"\steq", bend right=15] \ar[r,"{\iota_{\bottom{N}}}"] \ar[rr, bend right=20, shift right=1, "1_{\bottom{N}}"', near start]
 & \bottom{N} + [K] \ar[ur,phantom,"=" description]  \ar[r,"\copair{1}{\bottom{f}(K)x}"] & \bottom{N} 
\end{tikzcd} \label{diag:lemma}
\end{equation}
This diagram commutes 2-dimensionally, which is to say that the ``pasting'' of all four displayed identities is strictly equal to the strict (indeed, judgemental) equality $\bottom{f}\circ 1_{\bottom{M}} \steq 1_{\bottom{N}} \circ \bottom{f}$.
Applying the composite s-functor $\Struc{\derivcat{\L}{-}}$, we obtain:
\begin{equation}
\begin{tikzcd}[column sep=large]
\Struc{\derivcat{\L}{\bottom{M}}} \ar[dr,phantom,"\steq \scriptstyle(\alpha)" description]  \ar[rr,phantom,"\steq\scriptstyle(\epsilon_x)", bend left=15]  \ar[d,<-,"(\derivcat{\L}{\bottom{f}})^*"] \ar[r,<-,"(\derivcat{\L}{\iota_{\bottom{M}}})^*"] \ar[rr,<-, bend left=20, shift left=1,"1_{\Struc{\derivcat{\L}{\bottom{M}}}}",near start] & \Struc{\derivcat{\L}{\bottom{M} + [K]}} \ar[d,<-,"{(\derivcat{\L}{\bottom{f}+1})^*}"] \ar[r,<-,"(\derivcat{\L}{{\copair 1 x}})^*"] & \Struc{\derivcat{\L}{\bottom{M}}} \ar[d,<-, "(\derivcat{\L}{\bottom{f}})^*"]\\
\Struc{\derivcat{\L}{\bottom{N}}}  \ar[rr,phantom,"\steq\scriptstyle (\epsilon_{\bottom{f}x})", bend right=15] \ar[r,<-,"(\derivcat{\L}{\iota_{\bottom{N}}})^*"'] \ar[rr, <-,bend right=20, shift right=1, "1_{\Struc{\derivcat{\L}{\bottom{N}}}}"', near start]
 & \Struc{\derivcat{\L}{\bottom{N} + [K]}} \ar[ur,phantom,"\mathrlap{=\scriptstyle(\beta_x)}" description]  \ar[r,<-,"(\derivcat{\L}{\copair{1}{\bottom{f}x}})^*"'] & \Struc{\derivcat{\L}{\bottom{N}}}
\end{tikzcd} \label{diag:lemma2}
\end{equation}
which commutes in the same way.
Moreover, the upper and lower strict equalities in this diagram are $\epsilon_x$ and $\epsilon_{\bottom{f}x}$ respectively; we call the others $\alpha$ and $\beta_x$. 

We have an analogous diagram for $y$, in which the left-hand square $\alpha$ is the same.

Then since $ \partial_{\bottom{f}(K) x} N \converts (\derivcat{\L}{\copair{1}{\bottom{f}(K)x}})^* \derivdia{N}$,  $M' \converts (\derivcat{\L}{\bottom{f}})^* N'$, and $\partial_{x} M \converts (\derivcat{\L}{\copair1x })^* \derivdia{M}$, we have an identification
\begin{equation*}
\beta_x N : (\derivcat{\L}{\bottom{f} +1 })^*\partial_{\bottom{f}(K) x} N = \partial_x M. 
\end{equation*}
The same can be shown for $y$.

Consider an indiscernibility $\bottom{f}x \fiso \bottom{f} y$ which consists, by \Cref{lem:iso-equivalent}, of (1) an identification $i : \partial_{\bottom{f}x} N = \partial_{\bottom{f}y} N$ and (2) an identification 
$ j $ between $(\derivcat{\L}{\iota_{\bottom{N}}})^*i$ and the concatenation
\[(\derivcat{\L}{\iota_{\bottom{N}}})^*(\derivcat{\L}{\copair{1}{\bottom{f}x}})^*N' \overset{\epsilon_{\bottom{f}x}}{=}
  N' \overset{\epsilon_{\bottom{f}y}^{-1}}{=}
  (\derivcat{\L}{\iota_{\bottom{N}}})^*(\derivcat{\L}{\copair{1}{\bottom{f}y}})^*N'
\]
(which is strict, though $i$ is not).

 We need to construct an indiscernibility $x \fiso y$ which consists of (1) an identification $k : \partial_{x} M = \partial_{y} M$ and (2) an identification $(\derivcat{\L}{\iota_{\bottom{M}}})^*k  = \epsilon_{ x}^{} \concat \epsilon_{ y}^{-1}$.

The first component, $k$, of our desired indiscernibility $x\fiso y$ is the following concatenation:
\begin{align*}
  (\derivcat{\L}{\copair{1}{x}})^* (\derivcat{\L}{\bottom{f}})^* N'
  &\overset{\beta_x}{=} (\derivcat{\L}{\bottom{f}+1})^* (\derivcat{\L}{\copair{1}{\bottom{f}x}})^* N'\\
  &\overset{\mathllap{(\derivcat{\L}{\bottom{f}+1})^*}i}{=} (\derivcat{\L}{\bottom{f}+1})^* (\derivcat{\L}{\copair{1}{\bottom{f}y}})^* N'\\
  &\overset{\beta_y^{-1}}{=} (\derivcat{\L}{\copair{1}{y}})^* (\derivcat{\L}{\bottom{f}})^* N'
\end{align*}

Now we need $ (\derivcat{\L}{\iota_{\bottom{M}}})^*k  =  \epsilon_{ x}^{} \concat \epsilon_{ y}^{-1}$. 
Consider the commutative diagram in \Cref{diag:biglemma} where straight lines denote strict equalities, squiggly lines denote identifications, and double (squiggly) lines denote identifications between identifications. 
The 2-dimensional identification labeled $\nu$ arises from naturality, while those labeled $\sigma$ arise from the 2-di\-men\-sio\-nal commutativity of Diag.\ (\ref{diag:lemma2}).
The concatenation of the three top horizontal identifications in \Cref{diag:biglemma} is $ (\derivcat{\L}{\iota_{\bottom{M}}})^*k$.
Thus, \Cref{diag:biglemma} exhibits an identification of this with $\epsilon_{x}^{} \concat \epsilon_{y}^{-1}$.
\end{proof}
 
\begin{figure*} \begin{tikzcd}[column sep=large]
(\derivcat{\L}{\iota_{\bottom{M}}})^* \partial_x M  
\ar[ddrr,bend right,"\epsilon_x^{}"',no head,""{name=E}]
\ar[r,"(\derivcat{\L}{\iota_{\bottom{M}}})^*\beta_x ",squiggly, no head] 
\ar[r,phantom,""{above,name=A},shift left=4] 
\ar[r,phantom,""{below, name=B},shift right=4]
& (\derivcat{\L}{\iota_{\bottom{M}}})^* (\derivcat{\L}{\bottom{f} +1 })^* \partial_{\bottom{f}x} N
\ar[rr,"(\derivcat{\L}{\iota_{\bottom{M}}})^*(\derivcat{\L}{\bottom{f} +1 })^*i ",""{below,name=C},squiggly, no head]
\ar[d, no head,"\alpha "]
&& (\derivcat{\L}{\iota_{\bottom{M}}})^*(\derivcat{\L}{\bottom{f} +1 })^*\partial_{\bottom{f}y} N
\ar[r,"(\derivcat{\L}{\iota_{\bottom{M}}})^*(\beta_y )^{-1}",squiggly, no head]
\ar[r,phantom,""{above,name=A},shift left=4] 
\ar[r,phantom,""{below, name=B},shift right=4]
 \ar[d, no head,"\alpha "]
& (\derivcat{\L}{\iota_{\bottom{M}}})^* \partial_y M 
\\
&|[alias=X]|  (\derivcat{\L}{{\bottom{f}}})^*(\derivcat{\L}{\iota_{\bottom{N}}})^* \partial_{\bottom{f}x} N 
\ar[rr, "(\derivcat{\L}{{\bottom{f}}})^*(\derivcat{\L}{\iota_{\bottom{N}}})^*i"{name=D},""{below,name=DD}, squiggly, no head]
\ar[from=C, to=D,squiggly,shift left=0.5,no head,"\nu"]
\ar[from=C, to=D,squiggly,shift right=0.5, no head]
\ar[rd, "(\derivcat{\L}{{\bottom{f}}})^*( \epsilon_{\bottom{f}x}^{}  )"{description}, no head, bend right=10]
\ar[from=DD, to=Z,squiggly,shift right=0.5,no head]
\ar[from=DD, to=Z,squiggly,shift left=0.5, no head,"(\derivcat{\L}{{\bottom{f}}})^* j"]
&& |[alias=Y]|  (\derivcat{\L}{{\bottom{f}}})^*(\derivcat{\L}{\iota_{\bottom{N}}})^*\partial_{\bottom{f}y} N 
\\
&& |[alias=Z]| M \ar[uurr,bend right,"\epsilon_y^{-1}"',no head,""{name=F}]
\ar[ur, "(\derivcat{\L}{{\bottom{f}}})^*( \epsilon_{\bottom{f}y}^{-1}  )"{description}, no head, bend right=10]
\ar[from=E, to=X, equal,"\sigma"']
\ar[from=F, to=Y, equal,"\sigma"]
\end{tikzcd} 
\Description[]{Diagram for proof of lemma}
\caption{Diagram for proof of \Cref{lem:injwrtiso}}
\label{diag:biglemma}
\end{figure*}
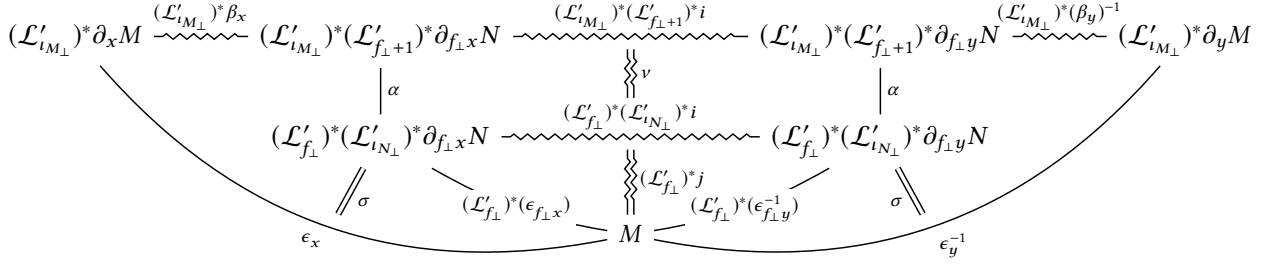

\begin{theorem}[\defemph{Higher SIP\shortorlong{}{, split-surjective case}}]
\label{thm:hsip}
Consider $\L: \Sig(n)$ and $M,N: \Struc{\L}$ such that $M$ is univalent.
The morphism $\idtovss: (M = N) \to (M \twoheadrightarrow N)$ is an equivalence.
\end{theorem}

\begin{proof}

It suffices to show that each $U_f$ of \Cref{def:isotovss} is an equivalence. 
We proceed by induction on $n$.
When $n \eqdef 0$, each $U_f$ is a endofunction on $\onetype$, and so is an equivalence.

When $n > 0$,
we first construct a map $F_f: \VSS(f) \to  \isIso(f)$.
Consider an element of $\VSS(f)$: a right inverse $s(K)$ of $\bottom{f}(K)$ for each $K : \bottom{\L}$, and $s' : \VSS(\derivdia{f})$.
Since $M'$ is univalent, the inductive hypothesis for $s'$ implies $\derivdia{f}$ is an isomorphism; thus it remains to show each $\bottom{f}(K)$ is an equivalence.

Since $s(K)$ is a right inverse of $\bottom{f}(K)$, it remains to show that we have $s(K)\bottom{f}(K) m = m$ for any $m: \bottom{M}(K)$. We have $\bottom{f}(K)s(K)\bottom{f}(K) m = \bottom{f}(K)m$ and thus $\bottom{f}(K)s(K)\bottom{f}(K) m \fiso \bottom{f}(K)m$.
We have already shown that $\derivdia{f}$ is an isomorphism $M' \cong ({\bottom{f}})^* N'$, so by \Cref{prop:uni1}, we get $M' = (\derivcat{\L}{\bottom{f}})^* N'$.
Thus, by \Cref{lem:injwrtiso}, we have $s(K)\bottom{f}(K) m \fiso m$; and since $M$ is univalent this yields $s(K)\bottom{f}(K) m = m$.

Thus, given our $(\lambda K.s(K), s'):\VSS(f)$, we have constructed an element of $\isIso(f)$; this defines $F_f: \VSS(f) \to  \isIso(f)$.
Since $\isIso(f)$ is a proposition (by \Cref{lem:isisoprop}), $F_f U_f = 1$.
Moreover, we constructed $F_f$ and $U_f$ such that $U_f F_f = 1$.\footnote{Since we showed that $\bottom{f}(K)$ was an equivalence by making $s(K)$ a homotopy inverse of it, and $U_f$ remembers not just the inverse map but one of the homotopies, we technically have to use here the fact that a homotopy inverse of a function $g$ can be enhanced to an element of $\isEquiv(g)$ while changing at most one of the constituent homotopies.}
Hence, $U_f : \isIso(f) \to \VSS(f)$ is an equivalence.

Thus, the function $\isotovss_{M,N} : (M \cong_L N) \to (M \twoheadrightarrow N  )$ is also an equivalence. Using \Cref{prop:uni1}, we find then that $\idtovss: (M = N) \to (M \twoheadrightarrow N)$ is an equivalence.
\end{proof}

\begin{example}\label{eg:vss-cat}
  \shortorlong{A split-surjective equivalence}{An equivalence} between univalent FOLDS-categories is the same as a fully faithful and split \shortorlong{(essentially)}{essentially} surjective functor, which by \cite[Lemma~6.6]{AKS13} is the same as an equivalence of categories.
  Thus, \Cref{thm:hsip} specializes to~\cite[Theorem 6.17]{AKS13}.
\end{example}

\shortorlong{
\begin{remark}
  With some more care, it is also possible to define a general analogue of non-surjective equivalences of categories and prove an HSIP for them.
  The idea is to replace surjectivity by essential surjectivity, i.e., surjectivity up to an indiscernibility.
  However, this is a bit tricky to make precise, because we want these indiscernibilities at all dimensions to live in the codomain structure $N$ and its derivatives; whereas with a na\"{i}ve inductive definition these indiscernibilities would end up living in pullbacks of these structures such as $f^*\derivdia{N}$ (the codomain of $\derivdia{f}$), which may not be univalent even if $N$ is.
  The details can be found in the extended version \cite{hsip_arxiv} of this article.
\end{remark}
}{We now move on to consider equivalences that are only \emph{essentially} surjective.
Makkai was unable to define non-surjective equivalences directly, instead considering \emph{spans} of surjective equivalences; but with our notion of indiscernibility we can avoid this detour.
However, we do have to be careful in the inductive step, because when considering $f:\hom_{\Struc{\L}}(M,N)$ we want all the indiscernibilities to lie in $N$ and its derivatives directly, \emph{not} in their pullbacks to derivatives at $M$.
This forces us define a somewhat more general notion.

For $a,b:\bottom{M}(K)$, we write $a\fiso^M_K b$ instead of $a\fiso b$ if needed to eliminate ambiguity.

\begin{definition}[\defemph{Relative equivalence}]
  Let $\L,\M:\Sig(n)$ and $\alpha:\hom_{\Sig(n)}(\L,\M)$, let $M: \Struc{\L}$ and $N:\Struc{\M}$, and let $f:\hom_{\Struc{\L}}(M,\alpha^* N)$.
  If $n\eqdef 0$, then $f$ is an equivalence relative to $\alpha$.
  For $n>0$, $f$ is an equivalence relative to $\alpha$ if
  \begin{enumerate}
  \item For all $K:\bottom{\L}$ and $y:\bottom{N}(\bottom{\alpha}(K))$, we have a specified $x:\bottom{M}(K)$ and indiscernibility $\bottom{f}(x) \fiso^N_{\bottom{\alpha}(K)} y$.
  \item $\derivdia{f}: \hom_{\Struc{\derivcat{\L}{\bottom{M}}}}(\derivdia{M},( \derivcat{\L}{\bottom{f}})^* \derivdia{(\alpha^* N)})$\shortorlong{\\ \phantom{x}\hfill}{} $\converts\hom_{\Struc{\derivcat{\L}{\bottom{M}}}}(\derivdia{M},( \derivcat{\L}{\bottom{f}})^* (\derivcat{\alpha}{\bottom{N}})^* \derivdia{N})$
   \narrowbreak is an equivalence relative to the composite
    \[ \derivcat{\L}{\bottom{M}} \xrightarrow{\derivcat{\L}{\bottom{f}}} \derivcat{\L}{\bottom{N}\circ \bottom{\alpha}} \xrightarrow{\derivcat{\alpha}{\bottom{N}}} \derivcat{\M}{\bottom{N}}. \]
  \end{enumerate}
  \defemph{Relative weak equivalences} are defined similarly, but requiring only 
  \( \Vert \sm{x:\bottom{M}(K)} (\bottom{f}(x) \fiso^N_{\bottom{\alpha}(K)} y) \Vert \) for each $K,y$.

  An unadorned \defemph{equivalence} means one relative to $\alpha\eqdef 1_{\L}$.
  We write $\streqv_\alpha(f)$ for the type ``$f$ is an equivalence relative to $\alpha$'', which in the inductive case means
  \begin{narrowmultline*}
    \streqv_\alpha(f) \eqdef \narrowbreak
    \left(\prd{K:\bottom{\L}}{y:\bottom{N}(\bottom{\alpha}(K))}\sm{x:\bottom{M}(K)} \left(\bottom{f}(K)(x) \fiso^N_{\bottom{\alpha}(K)} y\right)\right)
    \narrowbreak \times \streqv_{\derivcat{\alpha}{} \circ \bottom{f}} (\derivdia{f})
  \end{narrowmultline*}
  and $(M\simeq N) \eqdef \sm{f:\hom_{\Struc{\L}}(M,N)} \streqv_1(f)$ for the type of equivalences.
\end{definition}

Importantly, $\bottom{f}(x) \fiso^N_{\bottom{\alpha}(K)} y$ is distinct from $\bottom{f}(x) \fiso^{\alpha^*N}_{K} y$, even though $\bottom{(\alpha^* N)}(K) \converts \bottom{N}(\bottom{\alpha}(K))$ by definition.

\begin{lemma}\label{thm:vss-to-eqv}
  For $f:\hom_{\Struc{\L}}(M,\alpha^* N)$, we have a map $\VSS(f) \to \streqv_\alpha(f)$, which is an equivalence if $N$ is univalent.
\end{lemma}
\begin{proof}
  By induction on $n$.
  When $n\eqdef 0$, both are $\onetype$.
  For $n>0$, the desired map consists of the inductively defined $\VSS(\derivdia{f}) \to \streqv_{\derivcat{\alpha}{\bottom{N}}\circ \derivcat{\L}{\bottom{f}}}(\derivdia{f})$ together with a morphism
  \begin{narrowmultline*}
  \left(\prd{K: \bottom{\L}}{y:\bottom{N}(\bottom{\alpha}(K))} \sm{x:\bottom{M}(K)} \left(\bottom{f}(K)(x) =_{\bottom{N}(\bottom{\alpha}(K))} y\right)\right)\narrowbreak
    \to
    \left(\prd{K:\bottom{\L}}{y:\bottom{N}(\bottom{\alpha}(K))}\sm{x:\bottom{M}(K)} \left(\bottom{f}(K)(x) \fiso^N_{\bottom{\alpha}(K)} y\right)\right)
  \end{narrowmultline*}
  that is simply induced by $\idtoiso_{\bottom{f}(K)(x),y}$.
  The latter is an equivalence when $N$ is univalent by definition, as is the inductively defined map since $\derivdia{N}$ is univalent.
  (This last step would fail if we worked only with absolute equivalences, since $\alpha^*N$ can fail to be univalent even if $N$ is so.)
\end{proof}

\begin{theorem}[\defemph{Higher SIP}]\label{thm:hsip2}
  Consider $\L: \Sig(n)$ and $M,N: \Struc{\L}$ such that $M$ and $N$ are both univalent.
  The canonical morphism $\idtoeqv: (M = N) \to (M \simeq N)$ is an equivalence.
\end{theorem}
\begin{proof}
  Combine \Cref{thm:hsip,thm:vss-to-eqv}.
\end{proof}


}

\shortorlong{We do not know any analogue of~\cite[Lemma 6.8]{AKS13}, i.e., an HSIP for \emph{weak} equivalences.}{
  One might also hope for an analogue of~\cite[Lemma 6.8]{AKS13} about weak equivalences, i.e., an HSIP for \emph{weak} equivalences.
  A natural way to try to prove this would be by enhancing \cref{lem:injwrtiso} to say that some induced map ``$f : (x\fiso y)\to (f x \fiso f y)$'' is an equivalence, so that a weak equivalence between univalent structures would be an embedding and hence an equivalence.
  Unfortunately, an arbitrary morphism between structures does not induce any such map on types of indiscernibilities, even when it is an identity on derived structures as in \cref{lem:injwrtiso}.

  \begin{example}
    Let $\L$ be the height-2 signature such that an $\L$-structure consists of a type $MA$ and a binary relation $MR : MA \to MA \to \U$.
    Univalence at $R$ means that each $MR(x,y)$ is a proposition; while $a\fiso b$, for $a,b:MA$, means that $MR(x,a)\leftrightarrow M(x,b)$ for all $x$, $MR(a,x)\leftrightarrow M(b,x)$ for all $x$, and $MR(a,a)\leftrightarrow MR(b,b)$.

    Let $M$ be the $\L$-structure with $MA = \{a,b\}$ and $MR(x,y)$ always false, and $N$ the $\L$-structure with $NA=\{a,b,c\}$ with $NR$ always false except that $NR(a,c)$ is true.
    Let $\bottom{f}:MA\to NA$ be the inclusion, so that $\derivdia{M} = (\derivcat{\L}{\bottom{f}})^* \derivdia{N}$.
    Then $a\fiso b$ in $M$, but $\bottom{f}a \not\fiso \bottom{f}b$ in $N$.
  \end{example}

  Functors between categories (and the other categorical examples to be discussed in \Cref{sec:examples}) do generally preserve indiscernibilities, but only because the indiscernibilities in such cases they admit an equivalent ``diagrammatic'' characterization by a suitable ``Yoneda lemma'' (as described for categories in \cref{sec:folds-cats}).
  Note that this depends on the theory (i.e., the axioms) as well as the signature.
  We do not know a general condition on a theory ensuring that morphisms between its structures preserve indiscernibilities, nor do we know how to prove any general HSIP for weak equivalences.
}

\section{Axioms and Theories}

One of Makkai's goals was to define, for a given (FOLDS)-signature $\L$, a language for properties that are invariant under $\L$-equivalence.
\shortorlong{}
{
He calls such invariance the ``Principle of Isomorphism'' \cite{MakSFAM}:
\begin{quote}
    The basic character of the Principle of Isomorphism is that of a constraint on the
    language of Abstract Mathematics; a welcome one, since it provides for the separation of sense from nonsense. 
\end{quote}
}
Working in 2LTT, we do not need to devise a language for invariant properties ourselves; instead, we can rely on the homotopical fragment of 2LTT to sufficiently constrain our language.%
\footnote{An internally-syntactic description of such ``axioms'' might nevertheless prove useful in the future.}
\begin{definition}
 Let $\L$ be a signature. An \defemph{$\L$-axiom} is a function 
 $\Struc{\L} \to \PropU$.
\end{definition}

\begin{example}
 The axioms given in \Cref{eq:comp-unique,eq:E-cong} straightforwardly give rise to axioms for the signature $\LcatE$.
\end{example}

Any $\L$-axiom restricts to a predicate on univalent $\L$-structures. Our main result implies that any $\L$-axiom is invariant under equivalence of univalent $\L$-structures:
\begin{corollary}
 Given an $\L$-axiom $t$, univalent $\L$-structures $M$, $N$, and a split-surjective equivalence $M \twoheadrightarrow N$, we have $t(M) \leftrightarrow t(N)$.
\end{corollary}

A \defemph{theory} is a pair $(\L,T)$ of a signature $\L$ and a family $T$ of $\L$-axioms whose indexing type is fibrant; e.g., a list of five axioms can be specified by a family indexed by the standard finite fibrant type of five elements. A \defemph{(univalent) model of a theory $(\L,T)$} then consists of a (univalent) $\L$-structure $M$ together with a proof $t(M)$ for each axiom $t$ of $T$. The type of such models is fibrant.
In the next section, we discuss some particular theories and their univalent models.

\section{Examples of Theories}\label{sec:examples}

In this section, we discuss some examples of theories and their (univalent) models.

\begin{example}[First-order logic]
  Consider a many-sorted first-order theory with relations and equality (assumed to be a congruence):
  
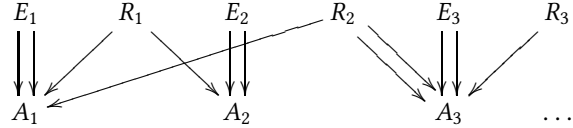
\[ \xymatrix{
    E_1 \ar@<1mm>[d] \ar@<-1mm>[d] & R_1 \ar[dl] \ar[dr] & E_2 \ar@<1mm>[d] \ar@<-1mm>[d] & R_2 \ar[dlll] \ar@<1mm>[dr] \ar@<-1mm>[dr] & E_3 \ar@<1mm>[d] \ar@<-1mm>[d] & R_3 \ar[dl]  \\
    A_1 && A_2 && A_3 & \dots
  }\]
  \begin{itemize}
  \item Univalence at $E_i$ and $R_i$ makes them proposition-valued.
  \item Univalence at $A_i$ makes it a set whose equality is $E_i$.
  \end{itemize}
  We recover first-order logic with equality.
  Any instance of this example, with sorts $(A_i)_{i : I}$, is also an instance of the SIP \cite[Section~9.9]{HTT} over $\Set^I$, including the examples of posets, monoids, groups, and fields mentioned in \Cref{sec:sip0}.
\end{example}

\begin{example}[$\dagger$-categories]
    A $\dagger$-category is a category with coherent isomorphisms 
    $(\_)^\dagger : \hom(x,y) \to \hom(y,x)$.
    \shortorlong{}{This is an interesting example to consider in structural approaches to category theory, since the correct notion of ``sameness'' for objects of a $\dagger$-category is \emph{not} ordinary isomorphism but rather \emph{unitary} isomorphism (one satisfying $f^{-1} = f^\dagger$), and similarly ``$\dagger$-structure'' on a category does not transport naturally across equivalence of categories.
    In our framework we can deal with this by incorporating the $\dagger$-structure into the signature, represented of course by its graph.}
A signature for $\dagger$-categories is as follows,
  \[
  \vcenter
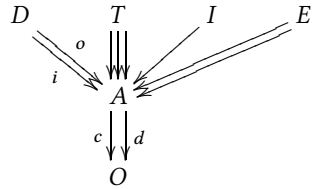
{\xymatrix@R=1.5pc{
      D \ar@<0mm>[dr]^{o} \ar@<-1mm>[dr]_{i} & T \ar[d] \ar@<1mm>[d] \ar@<-1mm>[d] & I \ar[dl] & E \ar@<1mm>[dll] \ar@<0mm>[dll] \\
      & A \ar@<1mm>[d]^{d} \ar@<-1mm>[d]_{c} \\
      & O}}
  \]
  with $co \steq di$ and $do \steq ci$ and the strict equalities of \Cref{fig:signatures}.
  Then for $\dagger$-categories regarded as structures:
  \begin{itemize}
  \item 
    An indiscernibility between $x,y:MO$ is a \emph{unitary isomorphism} $f:x \cong y$\shortorlong{, i.e., one satisfying $f^{-1} = f^\dagger$.}{.}
  \item In a univalent $\dagger$-category, $O$ is the groupoid of objects and unitary isomorphisms.
  \item \shortorlong{Split-surjective equivalences}{Equivalences of structures} 
    are \emph{$\dagger$-equivalences}, involving unitary natural isomorphisms.
  \end{itemize}
\end{example}

\shortorlong{}
{
\begin{example}[Presheaves]\label{eg:presheaves}
 The signature of a category with a presheaf on it is as follows:
 
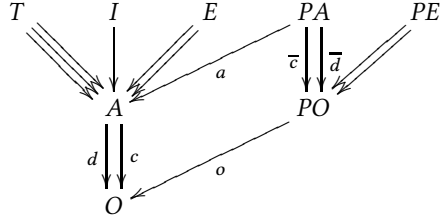
\[
 \xymatrix{ 
      T \ar[dr] \ar@<1mm>[dr] \ar@<-1mm>[dr] & I \ar[d] & E \ar@<.5mm>[dl] \ar@<-.5mm>[dl] & PA \ar@<1mm>[d]^{\overline{d}} \ar@<-1mm>[d]_{\overline{c}} \ar[dll]^{a} & PE\ar@<.5mm>[dl] \ar@<-.5mm>[dl] \\
      & A \ar@<1mm>[d]^{c} \ar@<-1mm>[d]_{d} & & PO \ar[dll]^{o}\\
      & O}
\]
with the equations $da \steq o\overline{d}$ and $ca \steq o\overline{c}$.
Given a structure $M$ of this signature, univalence at $PA$ means that $MPA(f, a, b)$ is a proposition, intuitively indicating that the function $P(f)$ maps $b$ to $a$.
In particular, one of the axioms we need to impose to carve out the presheaves among the structures is
\[ PA(f, a, b) \to PA(f, a', b) \to PE(a, a').\]
Univalence at $PO$ means that $MPO(x)$ is a set with equality $PE$, as expected for a presheaf.
Univalence at $A$ means that $MA(x,y)$ is a set with equality $ME$, as in a category: the additional dependency $PA$ doesn't disrupt this since we impose additional axioms saying that $E$ is a congruence for it as well.
Finally, an indiscernibility between elements of $MO$ consists of an indiscernibility (hence just an isomorphism) in the underlying category, together with a coherent bijection on values of the presheaf; but by an argument similar to that of \Cref{sec:folds-cats}, the latter part is trivial.
Thus, a univalent structure satisfying appropriate axioms is precisely a univalent category together with a presheaf on it.

A map $f : \hom(M,N)$ of structures consists of a functor $f$ between the underlying categories (given by the components $O$, $A$, $T$, $I$, $E$) of the signature, together with a natural transformation between the presheaves specified by $M$ and $N$ (given by the components $PO$, $PA$, and $PE$). Here, the component of $f$ on $PA$ encodes naturality.
If $f$ is an equivalence of structures, then its underlying functor is an equivalence; moreover essential surjectivity on the component $PA$ implies injectivity of the underlying natural transformation, whereas on $PO$ it implies that the natural transformation is pointwise surjective---thus it is a natural isomorphism.
\end{example}
}

\shortorlong{}
{
\begin{example}[Functors; {\cite[Section~6]{MFOLDS}}]\label{eg:anafunctors}
Just as we can represent the composition operation in a category by the relation $T$, we can describe two categories and a functor between them by adding ``relations'' between their objects and arrows:
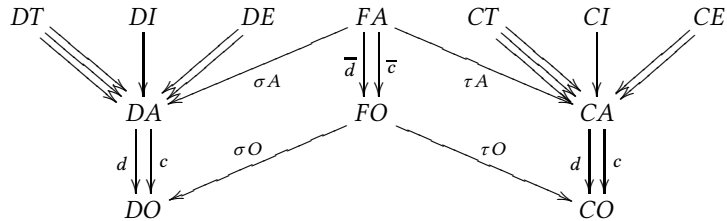
\[
 \xymatrix{ 
      DT \ar[dr] \ar@<1mm>[dr] \ar@<-1mm>[dr] & DI \ar[d] & DE \ar@<.5mm>[dl] \ar@<-.5mm>[dl] & FA \ar@<1mm>[d]^{\overline{c}} \ar@<-1mm>[d]_{\overline{d}} \ar[dll]^{\sigma A}
               \ar[drr]_{\tau A}
      & CT \ar[dr] \ar@<1mm>[dr] \ar@<-1mm>[dr] & CI \ar[d] & CE \ar@<.5mm>[dl] \ar@<-.5mm>[dl]
      \\
      & DA \ar@<1mm>[d]^{c} \ar@<-1mm>[d]_{d} & & FO \ar[dll]_{\sigma O} \ar[drr]^{\tau O}
      & & CA \ar@<1mm>[d]^{c} \ar@<-1mm>[d]_{d} \\
      & DO & & & & CO}
\]      
with the obvious equations on arrows. Here, $DO$ is the sort of objects of the domain category (with $D$ in $DO$ standing for ``domain''), $CO$ the sort of objects in the codomain, and $FO(x,y)$ the sort of ``witnesses that $F(x) = y$''.
For instance, if $F$ is a cartesian product functor, then an element of $FO((x_1,x_2),y)$ would be a product diagram $x_1 \leftarrow y \to x_2$.
Note that $FO$ does not generally consist of mere propositions: an object $y:CO$ can ``be the image'' of $x:DO$ in more than one way.
(For instance, a single object can be a cartesian product of two objects $x_1,x_2$ in more than one way.)

Given two such witnesses $w_1:FO(x_1,y_1)$ and $w_2:FO(x_2,y_2)$ and morphisms $g:DA(x_1,x_2)$ and $h:CA(y_1,y_2)$, an element of $FA(w_1,w_2,g,h)$ represents the assertion that $F(g) = h$ according to the witnesses $w_1$ and $w_2$.
We impose axioms stating that $F$ does have a value on each possible input object and arrow, e.g., $\forall x : DO, \exists y : CO, FO(x,y)$, and that composition, identities, and equality are preserved.
The result is what Makkai~\cite{makkai:avoiding-choice} calls an \emph{anafunctor}, whose ``values'' can be specified only up to isomorphism; if dependencies are forgotten, it can be thought of as a span of functors between (pre)categories in which the first leg is a surjective equivalence.

As in \Cref{eg:presheaves}, univalence at $DA$ and $CA$ just means they are sets with equalities $DE$ and $CE$.
Univalence at $FO$ is more subtle since it has no specified equality; of course it implies that $FO$ is a set, but it also says (roughly) that two witnesses $w_1,w_2:FO(x,y)$ are equal if they induce the same action on all arrows whose domain or codomain (or both) is $x$.
This is the uniqueness part of ``saturatedness'' for an anafunctor; the rest can be imposed by a pure existence axiom.

Univalence at $DO$ and $CO$ does \emph{not} reduce to ordinary univalence of the domain and codomain categories.
An indiscernibility in $DO$ is an isomorphism in the domain category equipped with a transport function for $FO$, i.e., an isomorphism $x_1\cong x_2$ together with bijections $FO(x_1,y) \cong FO(x_2,y)$ that respect the functorial action on arrows ($FA$), and similarly for $CO$.
Note that since $F$ acts on isomorphisms, transport for isomorphisms in the domain can be obtained from transport in the codomain.
Thus, if an anafunctor is a univalent structure for this signature, then its domain and codomain are univalent categories in the usual sense if and only if it is saturated.

Finally, a morphism between two \emph{saturated} anafunctors \emph{qua} structures consists of functors between the underlying categories (parts $D$ and $C$ of the signature) together with an (ana)natural isomorphism inhabiting the appropriate square.
The components of the latter come from the action of the morphism on $FO$, while naturality comes from its action on $FA$.
A morphism between unsaturated anafunctors is a square that commutes up to specified witnesses in $FO$, which may be more restricted than arbitrary isomorphisms.
\end{example}
}

\shortorlong{}
{
\begin{example}[Profunctors]
A profunctor from $C$ to $D$ is a functor $F: C^{\text{op}} \times D \to \Set$.
We can represent two categories and a profunctor between them using the following signature:
 
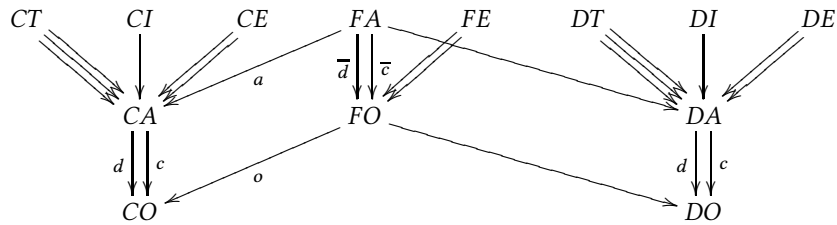
\[
 \xymatrix{ 
      CT \ar[dr] \ar@<1mm>[dr] \ar@<-1mm>[dr] & CI \ar[d] & CE \ar@<.5mm>[dl] \ar@<-.5mm>[dl] & FA \ar@<1mm>[d]^{\overline{c}} \ar@<-1mm>[d]_{\overline{d}} \ar[dll]^{a}
               \ar[drrr] 
      & FE \ar@<.5mm>[dl] \ar@<-.5mm>[dl] & DT \ar[dr] \ar@<1mm>[dr] \ar@<-1mm>[dr] & DI \ar[d] & DE \ar@<.5mm>[dl] \ar@<-.5mm>[dl]
      \\
      & CA \ar@<1mm>[d]^{c} \ar@<-1mm>[d]_{d} & & FO \ar[dll]^{o} \ar[drrr]
      & & & DA \ar@<1mm>[d]^{c} \ar@<-1mm>[d]_{d} \\
      & CO & & & & & DO}
\]
This looks very much like the signature for anafunctors, but we include an equality relation on $FO$, and moreover the composition equations are different: we have $da \steq o\overline{c}$ and $ca \steq o\overline{d}$ imposing contravariance in the first factor.
The rest of the theory of this structure is just a two-sided version of \Cref{eg:presheaves}.
\end{example}
}

\shortorlong{}
{
\begin{example}[Semi-displayed categories; see also {\cite[p.~107]{MFOLDS}}]
Displayed categories \cite{DBLP:journals/lmcs/AhrensL19} were developed, in particular, as a framework to define, in type theory, fibrations of categories without referring to equality of objects.
A displayed category  $\D$  over a category $\C$
is given by, for any $c : \ob{\C}$, a type $\D(c)$ of ``objects over $c$'', and,
for any morphism $f : \C(a,b)$ and $x : \D(a)$ and $y : \D(b)$, a type $\D_f(x,y)$ of ``morphisms from $x$ to $y$ over $f$'', together with suitably typed composition and identity operations.
This definition can be translated directly into a FOLDS-signature:
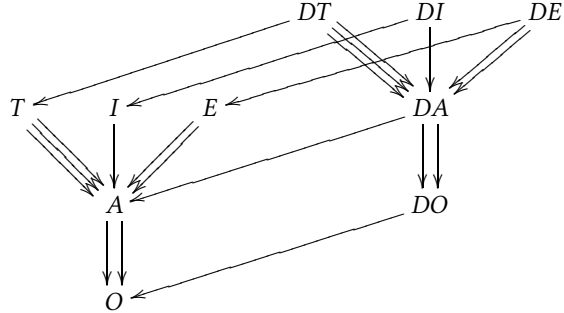
\[
 \xymatrix{ 
         &   & & DT \ar[dlll] \ar[dr] \ar@<1mm>[dr] \ar@<-1mm>[dr] &  DI \ar[dlll] \ar[d] & DE \ar[dlll] \ar@<.5mm>[dl] \ar@<-.5mm>[dl] 
           \\
      T \ar[dr] \ar@<1mm>[dr] \ar@<-1mm>[dr] & I \ar[d] & E \ar@<.5mm>[dl] \ar@<-.5mm>[dl] & & DA \ar[dlll] \ar@<1mm>[d] \ar@<-1mm>[d] \\
      & A \ar@<1mm>[d] \ar@<-1mm>[d] & & & DO \ar[dlll]\\
      & O}
\]
but the result is not well-behaved.
In particular, since $A$ has rank 1 in a height-4 signature, it might not be a set even in a univalent structure, and similarly $O$ might not be a 1-type.
When discussing fibrations, Makkai makes essentially the same point by saying that that this signature disqualifies $A$ from carrying an equality relation, since it only makes sense to impose equality relations, in the usual sense, on sorts that are only one level below the top.

One way to solve this problem would be to allow the base category to be a bicategory (though the fibers are only 1-categories), as in \Cref{eg:bicats} below.
However, we can avoid this complexity with the following signature due to Makkai, whose only dependency is for the objects:
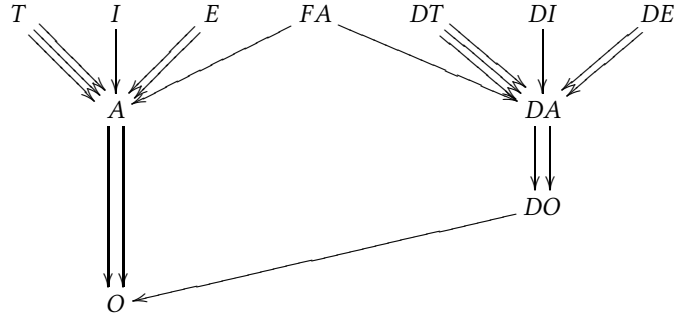
\[
 \xymatrix{ 
   T \ar[dr] \ar@<1mm>[dr] \ar@<-1mm>[dr] & I \ar[d] & E \ar@<.5mm>[dl] \ar@<-.5mm>[dl] & FA \ar[rrd] \ar[lld] & DT  \ar[dr] \ar@<1mm>[dr] \ar@<-1mm>[dr] &  DI  \ar[d] & DE \ar@<.5mm>[dl] \ar@<-.5mm>[dl] \\
   & A \ar@<1mm>[dd] \ar@<-1mm>[dd] & & & & DA \ar@<1mm>[d] \ar@<-1mm>[d] \\
   & & & & & DO \ar[dllll]\\
   & O}
\]
The dependency $DA \to A$ is replaced by the relation $FA$, asserted to be a functional relation, and the dependencies $DT \to T$ and $DI \to I$ are replaced by axioms, e.g., $ DI_{c,x}(\overline{f}) \wedge FA_{c,c,x,x}(f,\overline{f}) \to I_{c}(f)$.
A structure for this signature might be called a ``semi-displayed category''; it consists of, for any $c : \ob{\C}$, a type $\D(c)$ of objects over $\C$, and for any $x : \D(a)$ and $y : \D(b)$ a type $\D(x,y)$ with a function $\D(x,y) \to \C(a,b)$.
While they may appear more \textit{ad hoc} than displayed categories, semi-displayed categories do suffice to define notions involving strict fibers of functors, such as fibrations of categories.

As usual, we assert that $E$ and $DE$ are congruences for all the relations, including $FA$.
Thus, in a univalent structure $M$, all the top sorts are propositions, both $A$ and $DA$ are sets with standard equality, and each fiber category over $c:MO$ is a univalent category in the usual sense.
An indiscernibility between objects $c,d:MO$ consists of an ordinary isomorphism $\phi:c\cong d$ in the underlying base category together with all possible liftings of it in both directions, e.g., for any $x:MDO(c)$ a choice of a $y:MDO(d)$ and an isomorphism $x\cong y$ over $\phi$, and dually.
Since (assuming univalence at $DO$ and above) such liftings are unique when they exist, the type of such indiscernibilities is a subtype of that of ordinary isomorphisms.
Thus, in a univalent semi-displayed category, $MO$ is a 1-type, even though \Cref{thm:hlevel} only implies that it is a 2-type.
Moreover, when $M$ is univalent, the underlying ordinary category of the base category is univalent if and only if the semi-displayed category is an isofibration (which is a pure existence axiom; cf.\ also \cite[Problem~5.11]{DBLP:journals/lmcs/AhrensL19}).

A morphism of structures consists of a functor between the underlying categories and a ``semi-displayed functor'' above it; it is an equivalence of structures when both functors are equivalences.
\end{example}
}

\begin{example}[Multicategories/Colored operads]
 A (non-symmetric) multicategory (or colored non-symmetric operad) (see, e.g., \cite[Section~I.2]{higher-ops-higher-cats}) has arrows of different arity, generalizing the notion of $n$-ary functions on sets.
 The data of a multicategory is specified via the signature below:
 
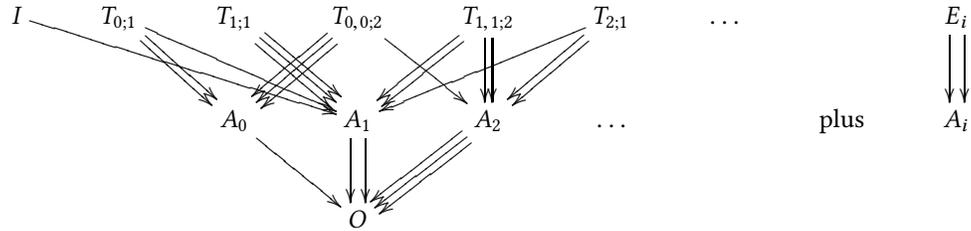
\[
  \begin{xy}
   \xymatrix{
         I\ar[rrrd] & T_{0;1} \ar[rrd] \ar@<-.5mm>[rd]\ar@<.5mm>[rd] & T_{1;1} \ar[dr] \ar@<1mm>[dr] \ar@<-1mm>[dr] 
         &
         T_{0,0;2}\ar[dl] \ar@<1mm>[dl] \ar@<-1mm>[dl] \ar[dr]
         & 
         T_{1,1;2} \ar@<-.5mm>[ld]\ar@<.5mm>[ld]\ar@<-.5mm>[d]\ar@<.5mm>[d]
        &
        T_{2;1}  \ar[dll] \ar@<-.5mm>[dl] \ar@<.5mm>[dl] & \ldots
        & & E_i \ar@<1mm>[d] \ar@<-1mm>[d]
       \\
       & & A_0 \ar[rd] & A_1 \ar@<1mm>[d] \ar@<-1mm>[d] & A_2 \ar[dl] \ar@<1mm>[dl] \ar@<-1mm>[dl] & \ldots
       & & \text{plus} & A_i
       \\
           & & & O
   }
  \end{xy}
 \]
 In this signature, $A_1$ is the sort of arrows known from categories, with one ``input'' object. The sort $A_0$ denotes arrows with no inputs, the sort $A_2$ arrows of 2 inputs, and so on. Accordingly, we have composition operations for such morphisms: the sort $T_{1;1}$ denotes the composition of two arrows with one input each---the composition known from categories. The sort $T_{1,1;2}$ denotes composition of two unary arrows with one arrow of two inputs, resulting in a composite arrow of two inputs.
 
 The signature for multicategories also includes an equality sort $E_i \rightrightarrows A_i$ for each $i$, but for readability we have omitted these from the main diagram.
 On a structure of this signature we can impose suitable axioms for the composition and identity in such a way that a model of the resulting theory is precisely a multicategory in the usual sense.
 
 An isomorphism $\phi : a \cong b$ in a multicategory is analogous to an isomorphism in a category: it consists of a morphism $f : A_1(a,b)$ together with $g : A_1(b,a)$ that is both pre- and post-inverse to $f$.
 Given a structure for the signature above, an indiscernibility $\phi : a \fiso b$ consists, in particular, of equivalences as in \crefrange{item:foldsiso1a}{eq:Eaa} (where $A$ needs to be replaced by $A_1$); 
 we have established in \Cref{thm:iso-foldsiso} that this data determines uniquely an isomorphism in the multicategory.
 In addition, the indiscernibility $\phi$ consists of further analogous equivalences for the sorts of $n$-ary arrows $A_i$ and their compositions.
 For instance, it includes a family of equivalences
 \begin{equation*}
       \phi_{xy\bullet} : A_2(x,y;a) \cong A_2(x,y;b)
 \end{equation*}
 and a family of equivalences
 \begin{equation}\label{eq:nonsym-multicat-comp}
  (T_{2,1})_{w,x,y,a}(f, g, h) \leftrightarrow (T_{2,1})_{w,x,y,b}(f, \phi_{y\bullet}(g),\phi_{wx\bullet}(h)).
 \end{equation}
 This latter equivalence with $h \eqdef f$ and $g\eqdef 1_a$ shows that the family $\phi_{xy\bullet}$ is given by postcomposition with the isomorphism corresponding to $\phi$, and similarly for the other families of maps.
 Thus, indiscernibilities in a multicategory also coincide with isomorphisms, so a multicategory is univalent precisely when its underlying category is.
 
 A morphism of structures accordingly corresponds to a functor between multicategories; it is an equivalence of structures if the functor is an equivalence.
\end{example}

\begin{example}[(Fat) symmetric multicategories]
 A symmetric multicategory is a multicategory with an action of the symmetric group $S_n$ on the arrows $A_n$; e.g., morphisms $A(x,y; z)$ correspond uniquely to morphisms $A(y,x; z)$. 
 These actions are furthermore asserted to be compatible with composition.
 
 Here, we consider an equivalent formulation of symmetric multicategories, via the notion of \emph{fat} symmetric multicategories (see, e.g., \cite[Appendix A.2]{higher-ops-higher-cats}).
 Its signature is similar to that of non-symmetric multicategories, but the arrows are instead indexed by \emph{unordered} finite sets $[n]$ of cardinality $n$.
 That is, the type $\L(1)$ of rank-1 sorts is $\FinSet$, a 1-type that is not a set.
 (Note that $\Nat$ is the 0-truncation of $\FinSet$, and also the type of \emph{ordered} finite sets.)
 Similarly, the compositions are indexed by the type
 \[ \tsm{X:\FinSet} (X \to \FinSet) \]
 where $(X,Y)$ denotes the composition of one morphism whose inputs are indexed by $X$ with a family of $X$ morphisms of which the $x^{\mathrm{th}}$ has inputs indexed by $Y(x)$.
 The fact that $\L(1)$ and $\L(2)$ are not sets makes it hard to draw this signature non-misleadingly, but we can give it a try, denoting the identifications in these types by loops:
 
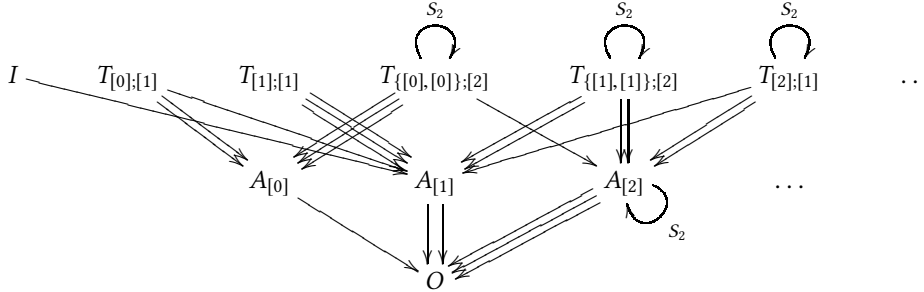
\[
   \xymatrix{
         I\ar[rrrd] & T_{[0];[1]} \ar[rrd] \ar@<-.5mm>[rd]\ar@<.5mm>[rd] & T_{[1];[1]} \ar[dr] \ar@<1mm>[dr] \ar@<-1mm>[dr] 
         &
         T_{\{[0],[0]\};[2]}\ar[dl] \ar@<1mm>[dl] \ar@<-1mm>[dl] \ar[dr] \ar@(ul,ur)^{S_2}
         & 
         T_{\{[1],[1]\};[2]} \ar@<-.5mm>[ld]\ar@<.5mm>[ld]\ar@<-.5mm>[d]\ar@<.5mm>[d] \ar@(ul,ur)^{S_2}
        &
        T_{[2];[1]}  \ar[dll] \ar@<-.5mm>[dl] \ar@<.5mm>[dl] \ar@(ul,ur)^{S_2} & \ldots
       \\
            & & A_{[0]} \ar[rd] & A_{[1]} \ar@<1mm>[d] \ar@<-1mm>[d] & A_{[2]} \ar[dl] \ar@<1mm>[dl] \ar@<-1mm>[dl] \ar@(r,d)^{S_2} & \ldots
       \\
           & & & O
   }
 \]
  Here $[n]$ denotes the standard $n$-element finite set $\{0,1,2,\dots,n-1\}$, while $S_n$ denotes its automorphism group, the symmetric group on $n$ elements.
  These loop ``arrows'' here are not morphisms in the inverse semi-category that represents our FOLDS-signature, but they do also act (by transport) on the types $M A_{[n]}$ in any structure.

 The composition sorts drawn above, representing the elements of the 0-truncation of the types of sorts, are almost like the sorts for non-symmetric multicategories, but some get identified.
 For instance, for non-symmetric multicategories there are two different sorts $T_{0,1;2}$ and $T_{1,0;2}$ for composition of a binary operation with a(n ordered) pair of a zeroary and a unary operation, but for the symmetric variant, these two compositions collapse into one connected component that we have written $T_{\{[0],[1]\}; [2]}$, which has an $S_2$ symmetric action.
 In general, the isotropy group of $T_{\{[k_1],\dots,[k_n]\};[n]}$ is the semidirect product $(S_{k_1}\times\cdots\times S_{k_n})\rtimes S_n$.
 As in the non-symmetric case, we have omitted the equality sorts $E_{[n]}$ on $A_{[n]}$ for readability.
 We assert associativity and unitality axioms for composition as spelt out in \cite[Appendix A.2]{higher-ops-higher-cats}.

 Univalence at the top-level sorts entails that the equality, composition, and identity sorts are pointwise propositions; 
 at $A_{[n]}$, it entails that $A_{[n]}$ are pointwise sets with equality given by their respective equality sorts.
 
  
 An indiscernibility $a \fiso b$ in $O$ consists of equivalences of sorts, e.g.,
 for $n\eqdef 1$ we have $\phi_{x\bullet} : A_{[1]}(\{x\};a) \cong A_{[1]}(\{x\}; b)$, and similar for the other hole and both holes in $A_{[1]}$.
 These equivalences are furthermore coherent with respect to the sorts $I$ and $T$, e.g., they satisfy the analog to \Cref{eq:nonsym-multicat-comp}.
 Given $a \fiso b$, we obtain in particular $\phi \eqdef \phi_{a\bullet}(1_a) : A_{[1]}(\{a\}; b)$. The morphism $\phi$ is an isomorphism; by the coherence with respect to $T$, the other equivalences for the sorts $A_{[n]}$ are given by suitable composition with $\phi$ or its inverse.
 Thus an indiscernibility $a \fiso b$ in $O$ is exactly an isomorphism $a \cong b$.
 
 A morphism of such structures is precisely a functor of symmetric multicategories; it is an equivalence if the functor is an equivalence.
\end{example}

\shortorlong{}
{
\begin{example}[Bicategories; {\cite[Section 7]{MFOLDS}}]\label{eg:bicats}
  It is not possible to describe strict 2-categories of any sort as univalent structures for one of our signatures, since the notion of a strict 2-category requires the 1-cells to form a set (so that composition can be strictly associative, i.e., up to a proposition-valued equality), but the 1-cells have two levels of dependency above them and so cannot be expected to be a set in a univalent structure.
  However, we can represent bicategories with the following signature from~\cite[p.~110]{MFOLDS} (with equality added):
  
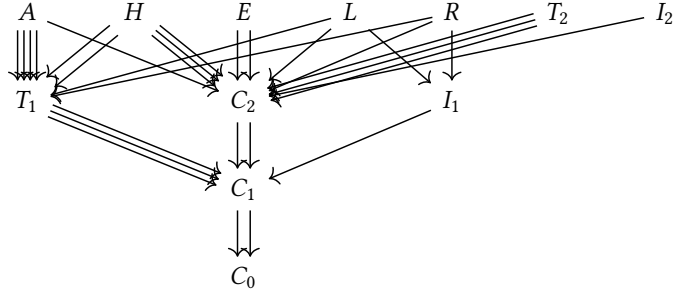
\[
    \begin{tikzcd}[ampersand replacement=\&]
      A \ar[d,shift left=.5] \ar[d,shift right=.5] \ar[d,shift left=1.5] \ar[d, shift right=1.5] \ar[drr] \&
      H \ar[dr] \ar[dr,shift left] \ar[dr,shift right] \ar[dl,shift left] \ar[dl,shift right] \&
      E \ar[d,shift left] \ar[d,shift right] \&
      L \ar[dr] \ar[dlll] \ar[dl] \&
      R \ar[d] \ar[dllll] \ar[dll] \&
      T_2 \ar[dlll] \ar[dlll,shift left] \ar[dlll,shift right] \&
      I_2 \ar[dllll] \\
      T_1 \ar[drr] \ar[drr,shift left] \ar[drr,shift right] \& \& 
      C_2 \ar[d,shift left] \ar[d,shift right] \& \&
      I_1 \ar[dll] \\
      \& \& C_1 \ar[d,shift left] \ar[d,shift right] \\
      \& \& C_0
    \end{tikzcd}
  \]

  Here $C_0,C_1,C_2$ are the sorts of objects, 1-cells, and 2-cells.
  The relations $T_2$, $I_2$, and $E$, with their axioms, make the 1-cells and 2-cells into hom-categories.
  The type $T_1$, which depends on a triangle of elements of $C_1$, represents composition: $t:T_1(f,g,h)$ is a ``reason why'' $g\circ f$ is equal to $h$.
  In general this is not a proposition, so composition is an anafunctor; thus (again following Makkai) we are actually representing ``anabicategories''.
  Similarly, $I_1(f)$ is the type of witnesses that $f$ is an identity 1-cell.
  The relation $A$ specifies the associativity isomorphisms: given $t_{12} : T_1(f_1,f_2,f_{12})$ and $t_{123} : T_1(f_{12},f_3,f_{123})$ and also $t_{23}:T_1(f_2,f_3,f_{23})$ and $t_{123}' : T_1(f_1,f_{23},f_{123}')$, the relation $A$ specifies a 2-cell in $C_2(f_{123},f_{123}')$ that plays the role of the associativity isomorphism.
  Similarly, $L$ and $R$ specify the left and right unit isomorphisms.
  Finally, the relation $H$ specifies the ``horizontal'' composite of two 2-cells along an object, given witnesses for how to compose their domains and codomains.

  Note that if we drop the sorts $I_1,I_2,L,R$ relating to identities, the resulting inverse category is a truncation of the coface maps in Joyal's category $\Theta_2$~\cite{joyal:disks}, which can be used to define non-algebraic notions of 2-category and $(\infty,2)$-category~\cite{rezk:theta,ara:nqcats}.
  The identity sorts are a ``fattening'' of this to incorporate the degeneracies while remaining inverse, as done for the simplex category in~\cite{kock:weakids}.
  
  We assert as an axiom that $E$ is a congrence for all the top-rank relations, so univalence at $C_2$ makes it consist of sets with equality $E$.
  As was the case for $FO$ in \Cref{eg:anafunctors}, univalence at $T_1$ means that two witnesses of composition are equal if they induce all the same horizontal composites of 2-cells and the same associators, and similarly for $I_1$.
  Thus, an indiscernibility between two elements of $C_1$ is a 2-cell isomorphism together with transport functions acting on $T_1$ and $I_1$.
  So if an anabicategory is univalent, then its hom-categories are individually univalent in the ordinary sense if and only if it is a saturated  anabicategory (which, as before, can be imposed axiomatically as a pure existence condition).
  Finally, a ``two-sided bicategorical Yoneda lemma'' implies that indiscernibilities in $C_1$ are equivalent to internal adjoint equivalences, so univalence at $C_0$ means that these are equivalent to identifications.

  The possible failure of saturation is due to the fact that we have two different ``shapes of 2-cells'': globes in $C_2$ and simplices in $T_1$.
  This can be avoided by building on a non-algebraic definition of bicategory involving only one shape of 2-cell, such as a semisimplicial version of the Street--Duskin~\cite{street:orientals,duskin:bicatnerve} simplicial nerve, or the opetopic nerve~\cite{bd:hda3,Mult1,cheng:opetopic-bicats}.

  A morphism of saturated anabicategories corresponds to a (pseudo) functor of bicategories.
  It is an equivalence of structures if the functor is a (strong) biequivalence, i.e., such that the maps on hom-types of all dimensions are split essentially surjective.
\end{example}
}

\begin{example}[Double (ana-bi)categories]
 A \emph{double category} is similar to a 2-category or bicategory, but has two families of 1-cells, called \emph{vertical} and \emph{horizontal}, respectively.
 The 2-cells take the shape of fillers for squares of 1-cells.
 Curiously, it is quite difficult to define a double category in which composition is weak in both directions.
 The closest approximation in the literature is the \emph{double bicategories} of Verity \cite[Definition~1.4.1]{verity-phd-tac}; in addition to squares, these have vertical and horizontal 2-cells of the usual ``globular'' shape, forming two separate bicategories with the same objects, together with operations by which the squares are acted on by the appropriate kind of globular 2-cells on all four sides.
 
 A suitable signature for double bicategories hence looks as follows:
   
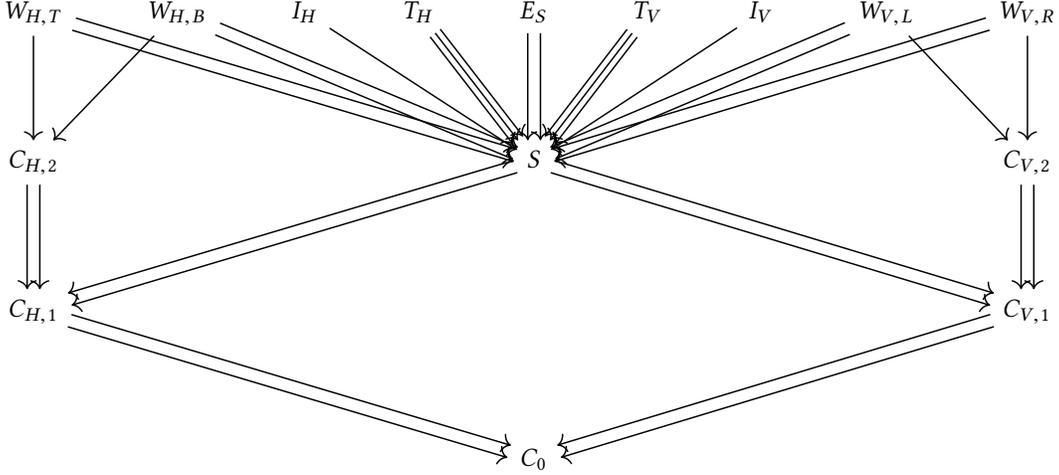
\[
    \begin{tikzcd}[ampersand replacement=\&, row sep = 40]
%
%
      W_{H,T} \ar[d] \ar[drrrr,shift right] \ar[drrrr,shift left] \&
      W_{H,B} \ar[dl] \ar[drrr,shift right] \ar[drrr,shift left] \&
      I_H \ar[drr] \&
      T_H \ar[dr] \ar[dr, shift left] \ar[dr, shift right]\&
      E_S  \ar[d,shift right] \ar[d,shift left] \&
      T_V \ar[dl] \ar[dl, shift left] \ar[dl, shift right] \&
      I_V \ar [dll] \&
      W_{V,L} \ar[dr] \ar[dlll,shift right] \ar[dlll,shift left] \&
      W_{V,R} \ar[d] \ar[dllll,shift right] \ar[dllll,shift left]
      \\
      C_{H,2} \ar[d,shift left] \ar[d,shift right] 
      \& \&      
      \& \&
      S \ar[dllll,shift left] \ar[dllll,shift right] \ar[drrrr,shift left] \ar[drrrr,shift right] 
      \& \&
      \& \&
      C_{V,2} \ar[d,shift left] \ar[d,shift right]
      \\
      C_{H,1} \ar[drrrr,shift left] \ar[drrrr,shift right] 
      \& \& \& \&
      \& \&
      \& \&
      C_{V,1} \ar[dllll,shift left] \ar[dllll,shift right]
      \\
      \& \& \& \& C_0
    \end{tikzcd}
  \]
  where we omit the bicategorical structure on $C_{X,2} \rightrightarrows C_{X,1} \rightrightarrows C_0$ for $X = H,V$ (see \Cref{eg:bicats}) for readability.
  Intuitively, an element $\lambda : S_{w,x,y,z}(f, g, h, k)$ can be pictured as a filler 
  \[
   \begin{tikzcd}[ampersand replacement = \&]
           z \ar[d,"h"'] \ar[r, "k"] 
           \ar[dr, phantom, "\lambda"]
           \&  
           x \ar[d, "f"]
           \\
           y \ar[r, "g"'] \& w
   \end{tikzcd}
  \]
  and the vertical action $(W_{V,R})_{w,x,y,z,f,g,h,k,f'}(\alpha, \lambda, \lambda')$ attaches a vertical 2-cell $\alpha : f \Rightarrow f'$ on the right of $\lambda$ to yield a filler $\lambda'$ of a square of 1-cells $f',g,h,k$.
  Similarly, we have vertical action on the left ($W_{V,L}$) and horizontal action on the top and bottom.
  As usual, these relations are asserted to be functional.
  Squares can be composed vertically ($T_V$) and horizontally ($T_H$), and we have identities $I_V$ and $I_H$ for these compositions.
  We assert that the equalities (not pictured) on vertical and horizontal 2-cells, as well as the equality $E_S$ on squares $S$, are congruences with respect to these operations.
  
  Univalence at $S$ says that $S$ is pointwise a set with equality given by $E_S$.
  Given two vertical 1-cells $f, f' : C_{V,1}(x,w)$, an indiscernibility between them is given by an isomorphism $\phi : f \cong f'$ in the underlying vertical bicategory together with a transport function for $S$, e.g., $\transfun \phi : S_{w,x,y,z}(f,g,h,k) \to S_{w,x,y,z}(f',g,h,k)$. 
  But coherence with respect to $W_{V,R}$ says that this transport function is given by action with the 1-isomorphism $\phi$, and similar for the other variables.
  
  Given $a, b : \C_0$, an indiscernibility $a \fiso b$ consists of a pair of a horizontal adjoint equivalence $\phi_H : a \simeq_H b$ and a vertical adjoint equivalence $\phi_V : a \simeq_V b$ together with transport functions for the sort $S$ that are coherent with respect to the top-level sorts.
  In particular, we have a transport function $S_{a,a,a,a}(1,1,1,1) \cong S_{b,a,a,a}(\phi_V,\phi_H,1,1)$; call $\Psi : S_{b,a,a,a}(\phi_V,\phi_H,1,1)$ the image of the identity filler under this isomorphism.
  Analogously, we have a transport function $S_{b,b,b,b}(1,1,1,1) \cong S_{b,b,b,a}(1,1,\phi_V,\phi_H)$; call $\Phi : S_{b,b,b,a}(1,1,\phi_V,\phi_H)$ the image of the identity filler under this isomorphism.
  The coherence laws for the top-sorts then entail that all the other transport functions are fully determined by the choice of $\Psi$ and $\Phi$, and that $\Psi$ and $\Phi$ compose with each other along $\phi_V$ and $\phi_H$ to identities.
  In summary, an indiscernibility in $C_0$ consists of a quadruple $(\phi_H,\phi_V,\Psi, \Phi)$ with these properties, known as an (invertible) \emph{companion pair} (see, e.g., \cite[\S1.2]{CTGDC_2004__45_3_193_0}).
  A morphism of structures is exactly a \emph{horizontal map} as defined in \cite[Definition~1.4.7]{verity-phd-tac}.
\end{example}

\begin{example}[Thunk-force categories]
  A \emph{thunk-force category} or \emph{abstract Kleisli category}~\cite{fuhrmann:abstract-kleisli} is a category $\D$ equipped with
  \begin{itemize}
  \item A functor $L:\D\to\D$,
  \item A natural transformation $\varepsilon : L \to 1_\D$, and
  \item An unnatural transformation $\vartheta : 1_\D \to L$,
  \end{itemize}
  such that $(L,\vartheta L, \varepsilon)$ is a comonad (so that in particular $\vartheta L : L \to L L$ \emph{is} a natural transformation) and each $\vartheta_x : x \to L x$ equips $x$ with the structure of an $L$-coalgebra.
  If $\D$ is the non-univalent Kleisli category $\C_T$ for a monad $T$, with corresponding adjunction $F : \C \rightleftarrows \D : U$ such that $F$ is the identity on objects, we can give it this structure where $(L,\vartheta L, \varepsilon) = (F U, F \eta U, \varepsilon)$ is the comonad induced by the adjunction and $\vartheta_x : \C_T(x,F U x) = \C(x,TTx)$ is the composite $\eta_{Tx} \circ \eta_x$.

  A signature for thunk-force categories is
  
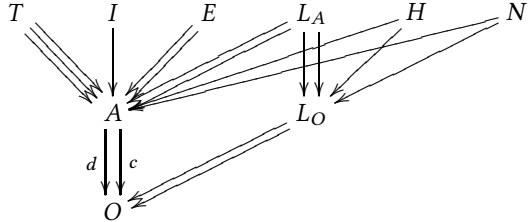
\[
    \xymatrix{ 
      T \ar[dr] \ar@<1mm>[dr] \ar@<-1mm>[dr] & I \ar[d] & E \ar@<.5mm>[dl] \ar@<-.5mm>[dl] &
      L_A \ar@<1mm>[d] \ar@<-1mm>[d]
      \ar@<.5mm>[dll] \ar@<-.5mm>[dll] & H \ar[dl] \ar[dlll] &
      N \ar[dll] \ar[dllll] \\
      & A \ar@<1mm>[d]^{c} \ar@<-1mm>[d]_{d} & & L_O \ar@<.5mm>[dll] \ar@<-.5mm>[dll] \\
      & O}
  \]
  where $N$ and $H$ represent $\varepsilon$ and $\vartheta$ respectively.
  The axioms are straightforward to formulate, and we find that the indiscernibilities in $O$ are the isomorphisms on which $\vartheta$ is natural.
  In general, morphisms (not necessarily isomorphisms) on which $\vartheta$ is natural (that is, morphisms that are $L$-coalgebra maps) are called \emph{thunkable}.
  As shown in~\cite{fuhrmann:abstract-kleisli} they form a (non-full, but wide) subcategory that can be equipped with a monad whose Kleisli category is the given thunk-force category.
  (Indeed, they are the full subcategory of the Eilenberg-Moore category of the comonad $(L,\vartheta L, \varepsilon)$ on the objects $(x,\vartheta_x)$.)
  In this sense a thunk-force category is precisely ``what is left of a Kleisli category when we forget the underlying category''.

  In a non-univalent Kleisli category $\C_T$, the functor $F:\C\to\C_T$ lands inside the thunkable morphisms.
  Thus, if $\C$ is a univalent category, then $\C_T$ is univalent as a thunk-force category just when every thunkable isomorphism in $\C_T$ is the $F$-image of a unique isomorphism in $\C$.
  This is the case for any monad such that
  \[
    \begin{tikzcd}
      x \ar[r,"\eta"] & T x \ar[r,shift left,"T \eta"] \ar[r,shift right,"\eta T"'] & T T x
    \end{tikzcd}
  \]
  is an equalizer diagram, which happens frequently but not always.
  For instance, the trivial monad on $\Set$ defined by $T x = 1$ admits a thunkable isomorphism $0 \cong 1$ in $\Set_T$, but there is no isomorphism $0\cong 1$ in $\Set$.

  The opposite of a thunk-force structure is called a \emph{runnable monad} (thus a thunk-force structure could also be called a ``corunnable comonad''), and the duals of thunkable morphisms are called \emph{linear}.
\end{example}

Thunk-force categories are used to model call-by-value programming languages, while runnable monads are used for call-by-name languages.
Since real-world programming languages allow functions to take more than one argument, these structures generally need to be enhanced with some kind of product; but in the presence of computational effects this is something weaker than a monoidal structure.

\begin{example}[Premonoidal categories]\label{eg:premonoidal}
  A \emph{premonoidal category}~\cite{pr:premonoidal,power:premonoidal} is like a monoidal category, but the tensor product operation is only required to be functorial in each variable separately, rather than jointly.
  That is, for objects $x,y$ we have a tensor product object $x\otimes y$, and for any $f:x\to x'$ we have $f\otimes y : x\otimes y \to x'\otimes y$ and for $g:y\to y'$ we have $x\otimes g : x\otimes y \to x\otimes y'$, but there is no ``$f\otimes g : x\otimes y \to x'\otimes y'$'', and the square
  \[
    \begin{tikzcd}
      x\otimes y \ar[r,"x\otimes g"] \ar[d,"f\otimes y"'] & x\otimes y' \ar[d,"f\otimes y'"] \\
      x'\otimes y \ar[r,"x'\otimes g"'] & x'\otimes y'
    \end{tikzcd}
  \]
  need not commute.
  If for some $f$ this square does commute for all $g$, and a dual condition holds with $f$ on the right, we say that $f$ is \emph{central}.
  The associativity and unit isomorphisms in a premonoidal category are additionally asserted to be central; note that naturality of the associator has to be formulated as three different axioms relative to morphisms in the three possible places.

  If $T$ is a strong and costrong monad on a monoidal category $\C$, then its Kleisli category $\C_T$ is premonoidal: its tensor product is that of $\C$, while for $f\otimes y$ is the composite $x\otimes y \xrightarrow{f \otimes y} T x' \otimes y \to T(x'\otimes y)$ with the strength, and dually for $x\otimes g$.

  We can obtain a signature for premonoidal categories by splitting the sort $\otimes_A$ of a monoidal category in two, for the two functors $x\otimes -$ and $-\otimes y$:
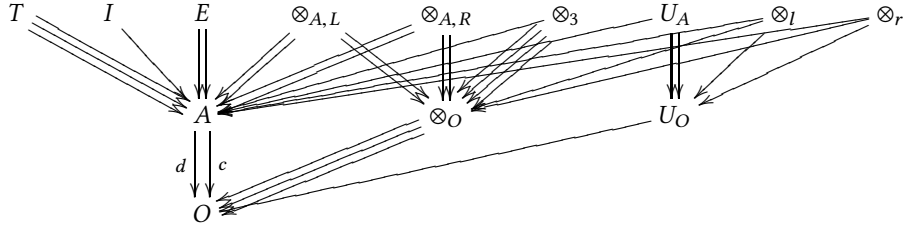
  \[
    \xymatrix{ 
      T \ar[drr] \ar@<1mm>[drr] \ar@<-1mm>[drr] & I \ar[dr] & E \ar@<.5mm>[d] \ar@<-.5mm>[d] &
      \otimes_{A,L} \ar@<.5mm>[dr] \ar@<-.5mm>[dr] \ar@<.5mm>[dl] \ar@<-.5mm>[dl] &
      \otimes_{A,R} \ar@<.5mm>[d] \ar@<-.5mm>[d] \ar@<.5mm>[dll] \ar@<-.5mm>[dll] &
      \otimes_3 \ar@<.5mm>[dl] \ar@<-.5mm>[dl] \ar@<1.5mm>[dl] \ar@<-1.5mm>[dl] \ar[dlll] &
      U_A \ar@<-.5mm>[d] \ar@<.5mm>[d] \ar[dllll] &
      \otimes_l \ar[dlll] \ar[dl] \ar[dlllll] &
      \otimes_r \ar[dllll] \ar[dll] \ar[dllllll] \\
      & & A \ar@<1mm>[d]^{c} \ar@<-1mm>[d]_{d} &&
      \otimes_O \ar@<1mm>[dll] \ar@<-1mm>[dll] \ar[dll] & &
      U_O \ar[dllll] \\
      & & O}
  \]
  That is, for $w:\otimes_O(x,y,z)$ and $w':\otimes_O(x',y,z')$ with $f:A(x,x')$ and $h:A(z,z')$, the relation $\otimes_{A,L}(w,w',f,h)$ says that $f\otimes y=h$ relative to $w$ and $w'$, and similarly for $\otimes_{A,R}$.
  We assert the usual axioms of a premonoidal category, including unique existence of an $h$ as in the previous sentence, which we denote $f\otimes^L_{w,w'} y$; similarly we have $x\otimes^R_{w,w'} g$ for $g:A(y,y')$.
  We define a morphism $f:A(x_1,x_2)$ to be \emph{central} if for any $g:A(y_1,y_2)$ and \emph{any} $w_{ij}:\otimes_O(x_i,y_j,z_{ij})$ (for $i,j\in \{1,2\}$) the following square commutes:
  \[
    \begin{tikzcd}[column sep=huge]
      z_{11} \ar[r,"x_1\otimes^R_{w_{11},w_{12}} g"] \ar[d,"f\otimes^L_{w_{11},w_{21}} y_1"'] & z_{12} \ar[d,"f\otimes^L_{w_{12},w_{22}} y_2"] \\
      z_{21} \ar[r,"x_2\otimes^R_{w_{21},w_{22}} g"'] & z_{22},
    \end{tikzcd}
  \]
  as well as a dual property on the other side.
  However, the naturality of the isomorphisms between any two values of an anafunctor means that it suffices if this holds for \emph{some} $w_{ij}$.
  Recall also that the axioms of a premonoidal category include centrality of the associator and unit isomorphisms.

  We also assert that for any $w:\otimes_O(x,y,z)$ and $w':\otimes_O(x,y,z')$, we have $1_x \otimes^L_{w,w'} y = x \otimes^R _{w,w'} 1_y$.
  In other words, if we have two values of $x\otimes y$, the canonical isomorphisms between them obtained from the two anafunctors $x\otimes -$ and $-\otimes y$ coincide, giving a morphism that we denote $1_x \otimes_{w,w'}  1_y$.
  This is an ``anafunctorial'' version of the standard condition that the two functors are ``equal on objects''.
  In particular, it is necessary to prove that \emph{identity morphisms are central}: for any $x:O$ and $g:A(y_1,y_2)$ with $w_j:\otimes_O(x,y_j,z_j)$ the following squares are equal:
  \[
    \begin{tikzcd}[column sep=huge]
      z_{1} \ar[r,"x\otimes^R_{w_{1},w_{2}} g"] \ar[d,"1_x \otimes^L_{w_{1},w_{1}} y_1"'] & z_{2} \ar[d,"1_x \otimes^L_{w_{2},w_{2}} y_2"] \\
      z_{1} \ar[r,"x\otimes^R_{w_{1},w_{2}} g"'] & z_{2}
    \end{tikzcd}
    \qquad
    \begin{tikzcd}[column sep=huge]
      z_{1} \ar[r,"x\otimes^R_{w_{1},w_{2}} g"] \ar[d,"x \otimes^R_{w_{1},w_{1}} 1_{y_1}"'] & z_{2} \ar[d,"x \otimes^R_{w_{2},w_{2}} 1_{y_2}"] \\
      z_{1} \ar[r,"x\otimes^R_{w_{1},w_{2}} g"'] & z_{2}
    \end{tikzcd}
  \]
  and the right-hand square commutes by functoriality of $\otimes^R$.

  We can also show that if $g:A(y_1,y_2)$ is central, then so is any $x\otimes_{w_1,w_2}^R g$.
  For if we have $h:A(u_1,u_2)$ with appropriate witnesses of the tensor product, we can form the following diagram:
  \[
    \begin{tikzcd}[row sep=huge, column sep=huge]
      {} \ar[rrr,"(x\otimes y_1) \otimes^R h"] \ar[ddd,"(x\otimes^R g) \otimes^L u_1"'] \ar[dr] & & &
      {} \ar[ddd,"(x\otimes^R g) \otimes^L u_2"] \ar[dl] \\
      & {} \ar[r,"x\otimes ^R (y_1\otimes^R h)"] \ar[d,"x\otimes^R (g \otimes^L u_1)" description] &
      {} \ar[d,"x\otimes^R (g\otimes^L u_2)" description] \\
      & {} \ar[r,"x\otimes^R (y_2\otimes^R h)"'] & {} \\
      {} \ar[ur] \ar[rrr,"(x\otimes y_2) \otimes^R h"'] & & & {} \ar[ul]
    \end{tikzcd}
  \]
  Here the inner square commutes by centrality of $g$ and functoriality of $x\otimes^R -$, while the diagonal arrows are components of the associativity isomorphism and the trapezoids commute by naturality.
  Therefore, the outer square commutes.
  Together with a similar argument on the other side, this implies that $x\otimes_{w_1,w_2}^R g$ is central when $g$ is.
  Similarly, $f\otimes^L y$ is central as soon as $f$ is.

  In particular, it follows that the isomorphism $1_x \otimes_{w,w'} 1_y$ between any two values of $x\otimes y$ is central.
  Therefore, the existential saturation condition must be similarly restricted: we assert that given $w:\otimes_O(x,y,z)$ and a \emph{central} $h:z\cong z'$, there exists a $w':\otimes_O(x,y,z')$ such that $1_x \otimes_{w,w'} 1_y = h$.

  As usual, $E$ is required to be a congruence for all rank-2 relations, so that univalence at $A$ means it is a set with $E$ as equality, and univalence at $U_O$ means it is a saturated ana-object.
  Now consider $w_1,w_2:\otimes_O(x,y,z)$; the indiscernibility type $w_1\fiso w_2$ is the proposition that $w_1$ and $w_2$ act the same on all arrows on both sides (the dependency of the natural transformations is automatically transportable by naturality).
  In other words, it says that $g \otimes^L_{w_1,w'} y = g\otimes^L_{w_2,w'} y$ for any $w':\otimes_O(x',y,z')$ and $g:A(x,x')$, and similarly on the other side.
  As for ordinary anafunctors, by functoriality this is equivalent to its special case $1_x \otimes_{w_1,w_2} 1_y = 1_z$.
  Thus, univalence at $\otimes_O$ means that if $1_x \otimes_{w_1,w_2} 1_y = 1_z$ then $w_1=w_2$, hence that the $w'$ asserted to exist in the existential saturation axiom is unique.

  Finally, an indiscernibility $x_1 \fiso x_2$ in $O$ consists of an isomorphism $\phi:x_1\cong x_2$ together with equivalences such as $\phi_{\bullet y} : \otimes_O(x_1,y,z) \simeq \otimes_O(x_2,y,z)$ and so on for the other holes, which respect all the rank-2 relations.
  Respect for $\otimes_{A,L}$ implies in particular that for $w:\otimes_O(x_1,y,z)$ we have $\phi \otimes^L_{w,\phi_{\bullet y}(w)} y = 1_{x_2} \otimes^L_{\phi_{\bullet y}(w),\phi_{\bullet y}(w)} 1_y = 1_{z}$.
  Now respect for $\otimes_{A,R}$ implies that for any $w_1:\otimes_O(x_1,y_1,z_1)$ and $w_2:\otimes_O(x_1,y_2,z_2)$ with $g:A(y_1,y_2)$ we have $\otimes_{A,R}(w_1,w_2,g,h) \leftrightarrow \otimes_{A,R}(\phi_{\bullet y_1}(w_1),\phi_{\bullet y_2}(w_2),g,h)$, or equivalently $x_1 \otimes^L_{w_1,w_2} g = x_2 \otimes^R_{\phi_{\bullet y_1}(w_1),\phi_{\bullet y_2}(w_2)} g$.
  But since $\phi\otimes^L_{w_1,\phi_{\bullet y_1}(w_1)} y_1 = 1_{z_1}$ and $\phi\otimes^L_{w_2,\phi_{\bullet y_2}(w_2)} y_2 = 1_{z_2}$, this implies that the following square commutes, since its vertical arrows are identities and its horizontal arrows are equal:
  \[
    \begin{tikzcd}[row sep=large,column sep=huge]
      {} \ar[d,"\phi\otimes^L_{w_1,\phi_{\bullet y_1}(w_1)} y_1"'] \ar[r,"x_1 \otimes^L_{w_1,w_2} g"] &
      {} \ar[d,"\phi\otimes^L_{w_2,\phi_{\bullet y_2}(w_2)} y_2"] \\
      {} \ar[r,"x_2 \otimes^R_{\phi_{\bullet y_1}(w_1),\phi_{\bullet y_2}(w_2)} g"'] & {}.
    \end{tikzcd}
  \]
  Together with a similar argument on the other side, this implies that $\phi:x_1\cong x_2$ is necessarily central.

  From here the usual sort of arguments imply that the rest of the structure of an indiscernibility (such as the equivalences $\phi_{\bullet y}$ used above) is uniquely determined by saturation applied to $\phi$.
  This is perhaps least obvious in the case of the equivalences $\phi_{\bullet\bullet} : \otimes_O(x_1,x_1,z) \simeq \otimes_O(x_2,x_2,z)$, since $\otimes$ is not jointly functorial in its arguments.
  But once we have shown that $\phi_{x_1\bullet}$ and $\phi_{\bullet x_2}$ are uniquely determined, respect for $\otimes_{A,R}$ tells us that for any $w:\otimes_O(x_1,x_1,z)$ we have $\otimes_{A,R}(w,\phi_{x_1\bullet}(w),\phi,h) \leftrightarrow \otimes_{A,R}(\phi_{\bullet\bullet}(w),\phi_{\bullet x_2}(\phi_{x_1\bullet}(w)),1_{x_2},h)$, which uniquely determines $\phi_{\bullet\bullet}(w)$ by saturation.

  Thus $x_1 \fiso x_2$ is equivalent to the type of central isomorphisms $x_1\cong x_2$, and so in a univalent premonoidal category $x_1=x_2$ is also equivalent to this type.
  In particular, since as we noted above the values of the tensor product are also determined uniquely up to unique central isomorphism, the type of such values is contractible, so we obtain an actual function $\otimes : O\to O \to O$ as we would hope.

  A morphism of structures is a functor that preserves the tensor product up to a specified central isomorphism that is natural in both possible ways and commutes with the associativity and unit isomorphisms.
  Note that if its domain is not univalent, such a functor need not preserve centrality; a counterexample can be found in~\cite[Section 5.2]{sl:univprop-impure}.
  Thus, we have a ``real-world'' example where morphisms of structures need not preserve indiscernibility.
  In particular, if there is a ``univalent completion'' operation for premonoidal categories, there will be morphisms between non-univalent structures that do not extend to their univalent completions.

  Similarly to the situation for thunk-force categories, a Kleisli category $\C_T$ is univalent as a premonoidal category just when every central isomorphism in $\C_T$ is the image of a unique isomorphism in $\C$.
  This fails, for instance, when $T$ is a commutative monad, in which case every morphism is central but not every isomorphism of free algebras is in the image of the free functor (e.g., the nontrivial automorphism of the free abelian group on one generator).
\end{example}

\begin{example}[Duploids]
  A \emph{duploid}~\cite{munchmaccagnoni:thesis} is a structure that combines call-by-value structure (such as in a thunk-force category) and call-by-name structure (such as in its dual, a runnable monad) in one.
  It starts with a \emph{pre-duploid}, which is almost like a category equipped with a map to the chaotic category on two objects $\{+,-\}$, except that the associativity law $(h \circ g) \circ f = h\circ (g\circ f)$ need only hold if either the codomain of $f$ (i.e., the domain of $g$) lies over $-$ (``is negative'') or the codomain of $g$ (i.e., the domain of $h$) lies over $+$ (``is positive'').
  A signature for pre-duploids is as follows:
  
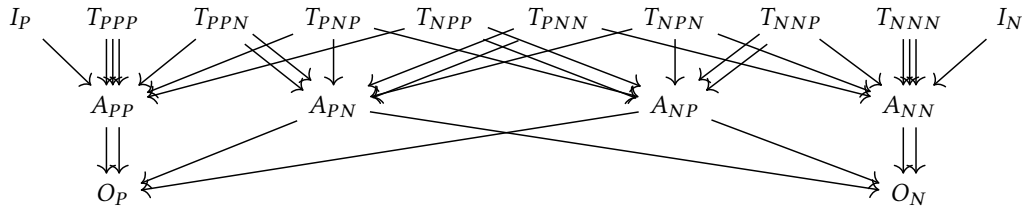
\[
    \begin{tikzcd}[column sep=small]
      I_P \ar[dr] &
      T_{PPP} \ar[d] \ar[d,shift left] \ar[d,shift right] &
      T_{PPN} \ar[dl] \ar[dr,shift left] \ar[dr,shift right] &
      T_{PNP} \ar[dll] \ar[d] \ar[drrr] &
      T_{NPP} \ar[dlll] \ar[drr,shift left] \ar[drr,shift right] &
      T_{PNN} \ar[dll,shift left] \ar[dll,shift right] \ar[drrr] &
      T_{NPN} \ar[dlll] \ar[d] \ar[drr] &
      T_{NNP} \ar[dl,shift left] \ar[dl,shift right] \ar[dr] &
      T_{NNN}  \ar[d] \ar[d,shift left] \ar[d,shift right] &
      I_N \ar[dl] \\
      & A_{PP} \ar[d,shift left] \ar[d,shift right] &&
      A_{PN} \ar[dll] \ar[drrrrr] &&&
      A_{NP} \ar[dlllll] \ar[drr] &&
      A_{NN} \ar[d,shift left] \ar[d,shift right] \\
      & O_P &&&&&&& O_N
    \end{tikzcd}
  \]
  plus equality congruences on all four sorts $A_{\bullet\bullet}$ that we have omitted to write, which thus coincide with the indiscernibilities on those sorts.

  Since the positive objects form a category in their own right, an indiscernibility $p_1\fiso p_2$ between $p_1,p_2:O_P$ consists in particular of an isomorphism $\phi:p_1\cong p_2$ in that category, together with equivalences $\phi_{\bullet n} : A_{PN}(p_1,n) \simeq A_{PN}(p_2,n)$ and $\phi_{n\bullet} : A_{NP}(n,p_1) \simeq A_{NP}(n,p_2)$ respecting composition of all sorts.
  The usual arguments imply that $\phi_{\bullet n}$ and $\phi_{n\bullet}$ are given by composition with $\phi$ or its inverse, so it remains to consider respect for composition.
  This includes, for instance, $T_{PNN}(g\circ \phi,h,k\circ \phi) \leftrightarrow T_{PNN}(g,h,k)$, which is to say that $h\circ (g\circ \phi) = (h\circ g) \circ \phi$ for all $g:A_{PN}(p_2,n_1)$ and $h:A_{NN}(n_1,n_2)$; and similarly for $T_{PNP}$.
  That is, the associativity law that isn't generally asserted in a pre-duploid does hold when $\phi$ is the first morphism.
  In the context of a pre-duploid, this is taken as the definition of when a morphism is \emph{thunkable}.
  The remaining conditions are automatic, so the indiscernibilities between positive objects are precisely the thunkable isomorphisms.
  Dually, the indiscernibilities between negative objects are precisely the \emph{linear} isomorphisms: those for which the missing associativities hold when they are the last morphism in the triple composite.

  A duploid is a pre-duploid together with ``parity shift'' functions $\Uparrow$ taking positive objects to negative ones and $\Downarrow$ taking negative objects to positive ones, together with unnatural families of linear isomorphisms $\mathsf{force}:\Uparrow p \cong p$, for positive $p$, and thunkable isomorphisms $\mathsf{wrap}:n \cong \Downarrow n$, for negative $n$.
  However, it turns out that $\Uparrow$ and $\Downarrow$ can in fact be made into functors, and $\mathsf{force}$ and $\mathsf{wrap}$ natural.
  Thus, this additional structure does not change the notions of indiscernibility or univalence.
\end{example}

\shortorlong{
\noindent
Further examples of higher structures that can be specified via a suitable signature are:
\begin{itemize}
 \item a category with a presheaf on it;
 \item two categories and a(n ana)functor between them;
 \item two categories and a profunctor between them;
 \item a displayed category (or fibration) over some base.
\end{itemize}
These examples are explained in detail in the extended version \cite{hsip_arxiv} of this article.
}
{}

\begin{example}[$T_0$-spaces]
  We end by sketching some signatures whose structures include topological spaces, showing that our abstract signatures are more general than FOLDS-signatures.
  Since a topology is a structure on one underlying set, it suffices to consider height-2 signatures with $\bottom{\L}\eqdef \onetype$, with $\derivdia{\L}{} : \U\to\U$ remaining to be specified.

  A first guess might be $\derivcat{\L}{M} \eqdef (M \to \PropU)$, so that an $\L$-structure would be a type $M$ with a predicate on its ``type of subsets'' $M \to \PropU$ representing ``is open''.
  Unfortunately, this is not a covariant s-functor.
  We can make it covariant via direct images (using propositional truncation), but this is not \emph{strictly} functorial, and moreover the resulting morphisms of structures would be open maps rather than continuous ones.

  Covariant strict functoriality does hold, however, for the double-powerset functor $M\mapsto ((M\to \PropU) \to \PropU)$, so we can use a definition of topological spaces that refers to sets of subsets instead of individual subsets.
  For instance, a topology is equivalent to a \emph{convergence} relation between filters (which are sets of subsets) and points, hence can be regarded as a particular $\L$-structure with
  \[\derivcat{\L}{M} \eqdef ((M\to \PropU) \to \PropU)\times M.\]
  The covariant functoriality specializes to the direct image of filters, so the $\L$-structure morphisms between topological spaces will be functions that preserve convergence, which is equivalent to continuity.
  Finally, univalence means that convergence is a proposition, that $M$ is a set, and that two points are identified if exactly the same filters converge to them; the latter is an equivalent way of saying the space is $T_0$.
  Of course, not every $\L$-structure is a topological space; we could hope to single out the spaces with a theory in an appropriate logic based on our signatures.

  We can also take $\derivcat{\L}{M} \eqdef ((M\to \PropU) \to \PropU)$ and associate a topological space $M$ to the set of all sets of subsets $\mathcal{T}$ such that all open subsets belong to $\mathcal{T}$.
  Once again the structure morphisms between topological spaces are the continuous maps, and the univalent such structures are the $T_0$-spaces.
  Other topological structures such as uniform spaces and proximity spaces can similarly be represented as structures over suitable height-2 signatures.
\end{example}

\section{Conclusion}\label{sec:conclusion}

Using a relativized form of the \emph{identity of indiscernibles},
we defined a general notion of indiscernibility of objects in a categorical structure, yielding a notion of univalence for such structures.
These notions depend only on the shape of the structures as specified by the signature, not on any axioms they satisfy.
We then showed a Structure Identity Principle for univalent structures that specializes to known results for first-order logic and univalent 1-categories, as well as many other important examples.

Regarding the setting we have chosen for our work, it seems impossible to define a fully coherent notion of
signature without 2LTT.  A sufficiently-coherent
``wild'' notion might suffice for our particular results,
but further development of the theory may require the fully coherent version. In addition, 2LTT is
necessary for treating FOLDS-signatures of arbitrary height (cf.\ \shortorlong{\cite{hsip_arxiv}}{\cref{sec:folds-sigs}}).

Some goals for further work include:
\begin{itemize}
 \item Removing the splitness condition from \cref{thm:hsip}, as discussed at the end of \cref{sec:hsip}.
 \item Developing a completion operation for structures, i.e., a universal way to turn a structure into a univalent one, generalizing the Rezk completion for categories \cite[Section~8]{AKS13}.
 \item Formalizing the results presented here in a computer proof assistant implementing 2LTT.
\end{itemize}

\shortorlong{}
{
\section{Version History}

This article is an extended version of an article of the same title, published in 
LICS 2020 \cite{hsip_lics}.
Compared to that version, the present version contains
\begin{itemize}
 \item a variant of the HSIP for \emph{essentially} split-surjective equivalences (see \Cref{thm:hsip2});
 \item an appendix on FOLDS-signatures in 2LTT and a translation from FOLDS- to abstract signatures (see \Cref{sec:folds-sigs}); and
 \item more examples in \Cref{sec:examples}.
\end{itemize}
}

\begin{acks}                            

  We thank Nicolai Kraus for helpful discussions on 2LTT, Paul Blain Levy for suggesting a connection to his work on contextual isomorphisms, and the anonymous referees for their constructive criticism.
  
  Ahrens and North acknowledge the support of the \grantsponsor{}{Centre for Advanced Study (CAS)}{} in Oslo, Norway, which funded and hosted the research project \emph{\grantnum{}{Homotopy Type Theory and Univalent Foundations}} during the 2018/19 academic year.
  
This work was partially funded by \grantsponsor{}{EPSRC}{https://epsrc.ukri.org/} under agreement number \grantnum{EP/T000252/1}{EP/T000252/1}.
  
This material is based
on research sponsored by \grantsponsor{id}{The United States Air Force Research
Laboratory}{https://www.wpafb.af.mil/AFRL/} under agreement number \grantnum{FA9550-15-1-0053}{FA9550-15-1-0053}, \grantnum{FA9550-16-1-0212}{FA9550-16-1-0212}, and \grantnum{FA9550-17-1-0363}{FA9550-17-1-0363}. The
U.S.\ Government is authorized to reproduce and distribute reprints
for Governmental purposes notwithstanding any copyright notation
thereon. The views and conclusions contained herein are those of
the author and should not be interpreted as necessarily
representing the official policies or endorsements, either
expressed or implied, of the United States Air Force Research
Laboratory, the U.S. Government, or Carnegie Mellon University.

  This material is based upon work supported by the
  \grantsponsor{GS100000001}{National Science
    Foundation}{http://dx.doi.org/10.13039/100000001} under Grant
  No.~\grantnum{DMS-1554092}{DMS-1554092}. 
  Any opinions, findings, and
  conclusions or recommendations expressed in this material are those
  of the author and do not necessarily reflect the views of the
  National Science Foundation.
\end{acks}


\appendix

\section{FOLDS-Signatures and Translation to Abstract Signatures}
\label{sec:folds-sigs}

\begin{definition}
Let $p: \Nat^s$. An \defemph{inverse semi-category\footnote{Makkai uses honest categories, but any inverse category is freely generated by a semi-category, so it is simpler to leave out the identity morphisms.} $\L$ of height $p$} consists of
\begin{itemize}
\item A family of types of objects $\L : \Nat^s_{< p} \to \Ustrict$.
  We call $\L(n)$ the family of \emph{sorts of rank $n$}.
\item A family of types of morphisms
  \[\hom_\L : \tprd{n:\Nat^s_{< p}}{m:\Nat^s_{< n}} \L(n) \to \L(m) \to \Ustrict.\]
\item A suitably typed composition operation on morphisms
  \[ (\cdot) : \hom_\L (x,y) \to \hom_\L (y,z) \to \hom_\L (x,z) \]
  that is strictly associative: 
  $ f \cdot (g \cdot h) \steq (f \cdot g) \cdot h$.
\end{itemize}
\end{definition}

The objects of $\L$ denote \emph{sorts} and the morphisms \emph{dependencies}. Morphisms can only go ``downwards'', that is, a sort can only depend on sorts ``below'' it; see \Cref{fig:signatures}.

\begin{definition}
  Given an inverse semi-category $\L$, the \defemph{fan\-out type of $K:\L(n)$ at $m<n$} is 
  \[ \fanout{K}{m} \define \sm{L:\L(m)} \hom_\L(K,L).\]
\end{definition}

\begin{definition}\label{def:material-sig}
A \defemph{FOLDS-signature of height $p$} is an inverse semi-category $\L$ of height $p$ for which each $\L(n)$ is fibrant and 
each $\fanout{K}{m}$ is cofibrant. The type of FOLDS-signatures of height $p$ is denoted by $\IC(p)$.
\end{definition}




\begin{examples}
  We denote the FOLDS-signatures shown in \Cref{fig:signatures} by $\Lrg$, $\Lcat$, and $\LcatE$ respectively.
  At each rank we have a finite-fibrant type of objects of that rank.
\end{examples}

\begin{definition}
For $\L,\M: \IC(p)$, a \defemph{strict semi-func\-tor $F:\L\to \M$} consists of functions
\begin{itemize}
\item $F : \prd{n:\Nat^s_{<p}} \L(n) \to \M(n)$
\item $F: \prd{m: \Nat^s_{<p}}{n: \Nat^s_{< m}}{x : \L(m)}{y: \L(n)}$\shortorlong{\\ \phantom{x}\hfill}{} $ \hom_\L (x,y) \to \hom_\M (Fx,Fy)$
\end{itemize}
which strictly preserves composition.
\end{definition}

Note that we take the ranks of objects to be specified data, which are preserved strictly by strict semi-functors.
In particular, a strict semi-functor $F:\L\to\M$ induces a function $\fanoutfun{F}{K}{m} : \fanout{K}{m} \to  \fanout{FK}{m} $ for every $K: \L(n)$ and $m<n$. 
We sometimes write $F$ instead of $\fanoutfun{F}{K}{m}$ when no confusion can arise.

\begin{definition}
A strict semi-functor $F:\L\to \M$ is a \defemph{discrete opfibration} when all the functions \[\fanoutfun{F}{K}{m} : \fanout{K}{m} \to  \fanout{FK}{m} \] are isomorphisms.
Let $\hom_{\IC(p)}(\L,\M)$ denote the type of such discrete opfibrations.
\end{definition}

We can think of each $\fanout{K}{m}$ as the structure on which $K$ depends, and then discrete opfibrations are exactly those functors which preserve this dependency structure.

\begin{proposition}\label{prop:material-sigs-cat}
  The type $\IC(p)$, the types of morphisms $\hom_{\IC(p)}(\L,\M)$ given by discrete opfibrations, and the obvious composition and identity form an s-category.
\end{proposition}
\begin{proof}
Consider $\L,\M,\N: \IC(p)$ and $F:\hom_{\IC(p)}(\L,\M)$,  $F:\hom_{\IC(p)}(\M,\N)$. Let $(G \circ F)(n) \define G(n) \circ F(n)$ and $(G \circ F)_{x,y} \define G_{F(n)(x), F(n)(y)} \circ F_{x,y}$. It is clear that this forms a strict semi-functor. 
It forms a discrete opfibration since for every $n: \Nat^s_{< p},K: \L(n)$, $ G \circ F: \fanout{K}{n} \to  \fanout{GFK}{n}$ is the composition of the isomorphisms $F : \fanout{K}{n} \to  \fanout{FK}{n}$ and $G : \fanout{FK}{n} \to  \fanout{GFK}{n}$.

For any $\L \in \IC(p)$, there is an identity discrete opfibration $1_\L: \hom_{\IC(p)}(\L,\L)$ whose components are all identity functions.

The composition defined above is clearly associative and left and right unital.
\end{proof}

A FOLDS-signature $\L$ will be mapped to an abstract signature whose $\bottomlevel$ is $\L(0)$, so we define:
\begin{notation}
For $p > 0$, $\L: \IC(p)$, let $\bottom{\L}\define \L(0)$.
\end{notation}

Now we formally define the derivative operation on FOLDS-signatures discussed in \Cref{sec:signatures}.


\begin{definition}\label{app:defn:derivcat}
Consider $p >0$, $\L: \IC(p)$, and $M : \bottom{\L} \to \U$. The \defemph{derivative of $\L$ with respect to $M$} is the inverse semi-category $\derivcat{\L}{M}$ of height $p-1$ with objects and morphisms defined as follows:
\begin{itemize}
\item $\derivcat{\L}{M}(n) \define \sum_{K : \L(n+1)} \prd{F: \fanout{K}{0}}  M(\pi_1 F)$
\item 
  $\hom_{\derivcat{\L}{M}}((K_1, \alpha_1), (K_2, \alpha_2)) \define$ \shortorlong{\\ \phantom{x}\hfill}{} $\sm{f: \hom(K_1,K_2)} \prd{F: \fanout{K_2}{0}}\alpha_1 (F \circ f) \steq \alpha_2(F)$
\end{itemize}
where $\pi_1 : \fanout{K}{0} \to \bottom{\L}$ is the projection and $F \circ f$ denotes the 
function $\fanout{K_2}{0} \to \fanout{K_1}{0}$ given by precomposition.
\end{definition}

\begin{example}\label{ex:deriv-empty}
  If $p \converts 1$ then $\L_{>0}$ is empty.
  Thus, no matter what $M : \bottom\L \to \U$ we choose, $\derivcat{\L}{M}$ is the empty signature.
\end{example}

\begin{example}\label{ex:deriv-height-one}
  If $\L$ is a signature of height $2$, then it consists of two types of sorts $\L(0)$ and $\L(1)$ and a family of hom-types $\hom_\L: \L(1) \to \L(0)\to \Ustrict$.
  Then for any $M : \L(0) \to \U$, the signature $\derivcat{\L}{M}$ has height $1$ consisting of just a single type of sorts of rank $0$.
  Each such sort is, by definition, a sort $K:\L(1)$ in $\L$ of rank $1$ together with, for any $L:\L(0)$ and $g:\hom_{\L}(K,L)$, an element of $\MatK{M}{L}$.
\end{example}

\begin{example}\label{ex:loopatob}
We have $\bottom{\Lrg} \equiv \lbrace O \rbrace$. Let $M$ be (a function picking out) the two-element set $\lbrace a,b \rbrace$.
Then $\derivcat{(\Lrg)}{M}$ is the following signature, with four sorts of rank $0$ and two sorts of rank $1$:
 \[
 \xymatrix{
 1 & I(a,a) \ar[d]^{i} & & &I(b,b) \ar[d]^{i} \\
 0 & A(a,a) &A(a,b) &A(b,a) &A(b,b) 
 }
 \]
\end{example}

The extra conditions on FOLDS-signatures in \Cref{def:material-sig} ensure that these signatures are closed under derivation:

\begin{proposition}\label{prop:deriv-good}
For $p>0$, $\L: \IC(p)$, and $M : \bottom{\L} \to \U$, the inverse semi-category $\derivcat{\L}{M}$ is a FOLDS-signature.
\end{proposition}

\begin{proof}
Since each $\fanout{K}{0}$ is cofibrant and each $M(\pi_1 F)$ is fibrant, we have that $\prd{F: \fanout{K}{0}} M(L)$ is fibrant. Since $\L(n+1)$ is fibrant, so is 
\[\sm{K : \L(n+1)} \prd{F: \fanout{K}{0}} M(\pi_1 F).\]

  Now consider $n: \Nat^s_{< p}$, $m: \Nat^s_{< n}$, and $(K,\alpha):\derivcat{\L}{M}(n)$, 
  We have
  \begin{align*}
    &\fanout{K,\alpha}{m}
    \\
    &\converts \sm{(L,\beta): \derivcat{\L}{M}(m)} \hom_{\derivcat{\L}{M}} (K, L)
    \\ 
   & \strictiso \sum_{\substack{L: \L(m+1) \\ \beta: \prd{\fanout{L}{0}} M(L) \\  f: \hom( K,L)}}
   \prd{(N,g): \fanout{L}{0} } \alpha (N, gf) \steq \beta(N, g) 
   \\
   & \strictiso \sum_{\substack{G: \fanout{K}{m+1} \\ \beta: \prd{\fanout{L}{0}}M(L)}}  
 \prd{(N,g): \fanout{L}{0} } \alpha (N, gf) \steq \beta(N, g)
   \\
      & \strictiso \sm{G: \fanout{K}{m+1}} 
  \onetype   \\
  & \strictiso \fanout{K}{m+1}
  \end{align*}
  Here, we expand ${\derivcat{\L}{M}(m)}$ and $ \hom_{\derivcat{\L}{M}} (K, L)$ to get the first isomorphism. We rearrange pairs and use the definition of $\fanout{K}{m+1}$ to get the second isomorphism.
  To get the third, observe that 
  \[    \sm{\beta: \prd{\fanout{L}{0}} M(L)}  
   \prd{(N,g): \fanout{L}{0} } \alpha (N, gf) \steq \beta(N, g) 
 \]
 is isomorphic to $\onetype$.
 
Since $\fanout{K}{m+1}$ is cofibrant, so is $ \fanout{K,\alpha}{m}$.
\end{proof}

\begin{definition}
  \label{def:derivation_functorial_action}
 Consider $p : \Nat^s$, $\L, \M: \IC(p)$, $H: \hom_{\IC(p)} (\L,\M)$, $L: \bottom{\L} \to \U$, $M: \bottom{\M} \to \U$, and $h:\prd{K:\bottom{\L}} L K \to M \bottom{H} K$.
 

  We define the functor $\derivcat{H}{h} : \derivcat{\L}{L} \to \derivcat{\M}{M}$  as follows.
  \begin{itemize}
  \item Consider an $n: \Nat^s_{< p-1}$ and a $(K, \alpha): \derivcat{\L}{L}(n)$, (so $\alpha : \prd{F:\fanout{0}{K}}L(\pi_1 F)$). We define $\derivcat{H}{h}(K, \alpha) : \derivcat{\M}{M}(n)$ to be $(H(K), \beta)$, where 
  \begin{align*}
     \beta(F) \eqdef & h_{\pi_1 (H^{-1}F)} \big(\alpha(H^{-1}F)\big) : M\bottom{H} (\pi_1 H^{-1}F) \steq M(\pi_1 F)
  \end{align*}

 for $F: \fanout{HK}{0}$. 
  \item 
  Consider a morphism $(f,\phi):\hom_{\derivcat{\L}{L}}( (K_1,\alpha_1) , (K_2, \alpha_2))$. We define $\derivcat{H}{h}(f, \phi): \hom_{\derivcat{\M}{M}}( (HK_1,\beta_1) , (HK_2, \beta_2))$ to be $(Hf,\psi)$. To define $\psi$ on $F : \fanout{HK_2}{0}$, we must check that
  \[ h_{\pi_1 (H^{-1}(F \circ Hf)}(\alpha_1(H^{-1}(F \circ Hf))) \steq h_{\pi_1 (H^{-1}F)} (\alpha_2(H^{-1}F)).\]
  But we have $H^{-1}( F \circ Hf) \steq H^{-1}( F ) \circ f $ (since applying the isomorphism $H$ produces $F \circ Hf$ on both sides) and $\alpha_1(H^{-1}( F ) \circ f) \steq \alpha_2(H^{-1}( F ))$ by $\phi$.
    \end{itemize}
\end{definition}

\begin{lemma}\label{lem:der_fun}
We have that 
\begin{enumerate}
\item for $\L: \IC(p)$

\[\derivcat{(1_\L)}{\lambda x. 1_{Lx}} \steq 1_{\derivcat{\L}{L}} \]
  
\item 
  for $\L, \M, \N: \IC(p)$, $H: \hom_{\IC(p)} (\L,\M)$, \\
$I:\hom_{\IC(p)} (\M,\N)$,  $L: \bottom{\L} \to \U$, $M: \bottom{\M} \to \U$, $N: \bottom{\N} \to \U$, $h:\prd{x:\bottom{\L}} L x \to M \bottom{H} x$, and $i:\prd{x:\bottom{\M}} M x \to N \bottom{I} x$,
\[\derivcat{I}{i} \circ \derivcat{H}{h} \steq \derivcat{(I \circ H)}{i \circ h}\] 
where 
$i \circ h (x,\ell) \define i(\bottom{H}x, h(x,\ell))$.
\end{enumerate}
\end{lemma}
\begin{proof}
Note that the desired strict equalities are obvious on the first components of objects and morphisms in ${\derivcat{\L}{L}}$. The desired strict equalities on the second components of morphisms follow from UIP and function extensionality. Thus, we check the strict equalities just on the second components of objects.

To check the strict equality (1) on objects, observe that 
\begin{align*}
\pi_2 \derivcat{(1_\L)}{1_{L}}(K, \alpha) & \steq  \lambda F.  \lambda x. 1_{Lx} (\pi_1 (1_\L^{-1} F), \alpha(1_\L^{-1} F) \\
& \steq \alpha.
\end{align*}

To check the strict equality (2) on objects, calculate that 
\begin{align*}
&\left( \pi_2 \derivcat{(I \circ H)}{i \circ h}(K, \alpha) \right) (F)
\\
& \steq   ih(\pi_1 ((IH)^{-1}F), \alpha((IH)^{-1}F))
  \\ & \steq  i(\bottom{H} \pi_1 ((IH)^{-1}F) , h(\pi_1 ((IH)^{-1}F), \alpha((IH)^{-1}F)))
\\ 
& \steq i(\pi_1 (I^{-1}F), h(\pi_1 (H^{-1}I^{-1}F), \alpha(H^{-1}I^{-1}F)))
\\
& \steq \left( \pi_2 (\derivcat{I}{i} \circ \derivcat{H}{h})(K, \alpha) \right) (F).\qedhere
\end{align*}
\end{proof}

\begin{lemma}\label{lem:Uisfib}
Consider $\L: \IC(p)$ and $M: \bottom{\L} \to \U$.
Let $\L_{>0}$ be the inverse semi-category given by $\L_{>0}(n) \define \L_{>0}(n+1)$ and 
$\hom_{\L_{>0}}(x,y) \define \hom_{\L}(x,y)$.

The evident forgetful functor $U: \derivcat{\L}{M}\to \L_{>0}$ is a discrete opfibration.
\end{lemma}
\begin{proof}
Consider a $(K,\alpha):\derivcat{\L}{M}$ where $K: \L(n+1)$, $\alpha: \prd{F: \fanout{K}{0}} M(\pi_1 F)$ and a $(L,f): \fanout{K}{m}$ where $L: \L(m)$, $f: \hom(K,L)$. Let $U^{-1} (L,f)$ be $((L,\beta),(f,\gamma))$ where we define $\beta: \prd{F: \fanout{L}{0}} M(\pi_1 F)$ and $\gamma: \prd{F: \fanout{L}{0}} \alpha(F \circ f) \steq \beta (F)$ as follows. Let $\beta(F) \define \alpha(F \circ f)$. Then we have $\gamma$ by construction.

Clearly, $U U^{-1} \steq 1_{\fanout{K}{m}}$. To show \[U^{-1} U \steq 1_{\fanout{(K,\alpha)}{m}} \enspace , \] consider a $((L,\beta),(f,\gamma)): \fanout{(K,\alpha)}{m}$. 
We get that 
\[ U^{-1} U((L,\beta),(f,\gamma)) \steq ((L,\beta'),(f,\gamma'))\] 
and
\[\gamma^{-1} * \gamma': \prd{F: \fanout{L}{0}}  \beta (F) \steq \beta'(F)  
\]
By function extensionality, $\beta \steq \beta'$ and by UIP and function extensionality, $\gamma \steq \gamma'$.
\end{proof}

\begin{lemma}\label{lem:fibleftcan}
Consider $\L,\M,\N:\IC(p)$. 
Consider functors $F$ from $\L$ to $\M$ and $G$ from $\M$ to $\N$ such that $G \circ F$ and $G$ are discrete opfibrations. Then $F$ is a discrete opfibration.
\end{lemma}
\begin{proof}
Consider a
 $K:\L(n)$. The following strictly commutative diagram shows that $F$ is an isomorphism on fanouts, and thus a discrete opfibration.
 \[ \raisebox{1.5cm}{\xymatrix{
 \fanout{K}{m} \ar[r]^{F} \ar[dr]_{G \circ F}^\cong &  \fanout{FK}{m} \ar[d]^{G}_\cong\\
 & \fanout{GFK}{m} 
 }}\qedhere \]
\end{proof}

\begin{proposition}\label{prop:deriv-opfib}
The functor $\derivcat{H}{h}: \derivcat{\L}{L} \to \derivcat{\M}{M}$ from \Cref{def:derivation_functorial_action} is a discrete opfibration.
\end{proposition}

\begin{proof}
The following square commutes.
\[
  \begin{tikzcd}
    \derivcat{\L}{L} \ar[r,"\derivcat{H}{h}"] \ar[d,"U"'] & \derivcat{\M}{M} \ar[d,"U"] \\
    \L_{>0} \ar[r,"{H_{>0}}"'] & \M_{>0}
  \end{tikzcd}
\]
Note that since $H$ is a discrete opfibration, so is ${H_{>0}}$.
Since both instances of $U$ are also discrete opfibrations (\Cref{lem:Uisfib}), we find (using \Cref{lem:fibleftcan}) that $\derivcat{H}{h}$ is a discrete opfibration.
\end{proof}

We now have all the ingredients for our translation. 

\begin{theorem}\label{thm:translation}
For each $p: \Nat^s$, define an s-functor $E_p: \IC(p) \to \Sig(p)$ by induction on $p$ as follows.

Since $\Sig(0)$ is the trivial category on $\onetype$, there is a unique s-functor $\IC(0) \to \Sig(0)$ (which is actually an equivalence).

For $p > 0$ and $\L: \IC(p)$, let $E_p(\L)$ consist of:
\begin{enumerate}
\item The type $\bottom{\L} : \U$.
\item The functor $E_{p-1}  \derivcat{\L}{-}: (\bottom{\L} \to \U) \to \Sig(p-1)$ defined on objects as in \Cref{app:defn:derivcat} and \Cref{prop:deriv-good}
and on morphisms as 
in \Cref{def:derivation_functorial_action} and \Cref{prop:deriv-opfib}.
\end{enumerate}
For $\L,\M: \IC(p)$ and $F: \hom(\L,\M)$, let $E_p (F)$ consist of:
\begin{enumerate}
\item The function $\bottom{F} : \bottom{\L} \to \bottom{\M}$.
\item The natural transformation with underlying function 
\[E_{p-1}  \derivcat{F}{\lambda x. 1_{-  x}}: \prd{M : \bottom{\M} \to \U} \hom(\derivcat{\L}{M \circ \bottom{F}},\derivcat{\M}{M}) 
\]
defined in \Cref{def:derivation_functorial_action} and \Cref{prop:deriv-opfib}.
\end{enumerate}
\end{theorem}

\begin{proof}
We check that $E_p$ is functorial. 

For any $\L : \IC(p)$, we have the following.
\begin{align*} 
\pi_1 E_p(1_\mathcal L) &\steq 1_{\bottom{\L}} \\
& \steq \pi_1(1_{E_p \mathcal L}) \\
\pi_2 E_p(1_\mathcal L) &\steq E_{p-1} \circ \derivcat{1_\L}{\lambda x. 1_{-  x}}\\
& \steq 1_{E_{p-1} \derivcat{1_\L}{ \lambda x. 1_{-  x}}} \\
& \steq \pi_2(1_{E_p \mathcal L}) 
\end{align*}

For any $\M, \N, \mathcal P: \IC(p)$, $F: \hom(\M, \N)$, $G: \hom(\N,\mathcal P)$, we have the following.
\begin{align*} 
\pi_1 E_p(G \circ F) &\steq \bottom{(G \circ F)} \\
& \steq \bottom{G} \circ \bottom{F} \\
& \steq \pi_1 (E_p G \circ E_p F) \\
\left(  \pi_2 E_p(G \circ F) \right)(M) &\steq E_{p-1} \circ \derivcat{(G \circ F)}{\lambda x. 1_{M  x}}\\
&\steq E_{p-1}  (\derivcat{G}{\lambda x. 1_{M  x}} \circ \derivcat{F}{\lambda x. 1_{M  x}})\\
& \steq (E_{p-1}  \derivcat{G}{\lambda x. 1_{M  x}}) \circ (E_{p-1}  \derivcat{F}{\lambda x. 1_{M  x}}) \\
& \steq \pi_2 E_{p-1} G \circ \pi_2 E_{p-1} F. \qedhere
\end{align*}
\end{proof}

Intuitively, this can be thought of as mapping into the s-category \emph{coinductively} defined by a derivative functor, with the result landing inside the inductive part (our abstract signatures) because our FOLDS-signatures have finite height.

\begin{remark}
 We expect structures, morphisms of structures, isomorphisms of structure, and split-surjective equivalences of structures over the abstract signature induced by a FOLDS-signature to correspond to Reedy fibrant diagrams, natural transformations of such, levelwise equivalences, and ``Reedy split-surjections'' (cf.\ \cite{shulman_2015, shulman_univalence_ei, 2LTT}).  Proving this precisely is left for future work.
\end{remark}


\bibliography{../foldssatrefs}

\end{document}